\documentclass{article}

\usepackage[a4paper, total={7in, 8in}]{geometry}
\usepackage[utf8]{inputenc}   
\usepackage[T1]{fontenc}      
\usepackage{multicol}
\usepackage{amsmath,amsthm} 
\usepackage{amssymb,mathrsfs} 
\usepackage{amssymb}
\usepackage{amsfonts}
\usepackage[normalem]{ulem}
\usepackage{nicefrac}
\usepackage{graphicx}
\usepackage{enumerate}
\usepackage{graphicx} 
\usepackage{lmodern} 
\usepackage{mathtools}
\usepackage{dsfont}
\usepackage{physics}

\usepackage{framed}
\usepackage{color}

\newcommand\N{\mathbb{N}}
\newcommand\E{\mathbb{E}}

\newcommand\C{\mathbb{C}}
\newcommand\Mm{\mathscr{M}}

\renewcommand\H{\mathscr{H}}

\newcommand\R{\mathbb{R}}

\newcommand\X{\mathcal{X}}

\newcommand\psig{\psi_{\gamma}}

\newcommand\PX{\mathcal{P}(\mathcal{X})}

\newcommand\intX{\int_{\mathcal{X}}}
\newcommand\Binfty{B^{\infty}}

\newcommand\proba{\mathbb{P}}
\newcommand\e{\mathrm{e}}
\newcommand\LFD{\mathcal{L}_{\mathrm{FD}}}
\newcommand\LS{\mathcal{L}_{\mathrm{S}}}
\newcommand\LA{\mathcal{L}_{\mathrm{A}}}
\newcommand\IS{I_{\mathrm{S}}}
\newcommand\IA{I_{\mathrm{A}}}
\newcommand\liminfb{\underline{\lim}}
\newcommand\limsupb{\overline{\lim}}
\newcommand\Lc{\mathcal{L}}
\newcommand\Lham{\mathcal{L}_{\mathrm{ham}}}
\newcommand\muref{\mu}

\newcommand\ind{\mathds{1}}
\newcommand\Cb{C_{\mathrm{b}}}

\newcommand\Cinfc{C_{\mathrm{c}}^{\infty}}
\newcommand\Cinfty{C^{\infty}}

\newcommand\I{\mathcal{I}}

\newcommand\id{\mathrm{Id}}
\newcommand\Y{\mathcal{Y}}
\newcommand\IF{I_{\mathrm{F}}}
\newcommand\IV{I_{\mathrm{V}}}
\newcommand\Ps{\Psi}
\newcommand\kap{\kappa}

\newcommand{\vertiii}[1]{{\left\vert\kern-0.25ex\left\vert\kern-0.25ex\left\vert #1 
    \right\vert\kern-0.25ex\right\vert\kern-0.25ex\right\vert}}

\renewcommand{\leq}{\leqslant}

\renewcommand{\geq}{\geqslant}

\newtheorem{theorem}{Theorem}
\newtheorem{lemma}{Lemma}
\newtheorem{assumption}{Assumption}
\newtheorem{proposition}{Proposition}
\newtheorem{corollary}{Corollary}
\newtheorem{remark}{Remark}

\iffalse    
   \excludecomment{corrige}
\else
  
\fi

\begin{document}

\title{Large deviations of empirical measures of diffusions in weighted topologies}
\author{Grégoire Ferré and Gabriel Stoltz\\
\small Université Paris-Est, CERMICS (ENPC), Inria, F-77455 Marne-la-Vallée, France
}

\date{\today}

\maketitle



\abstract{
 We consider large deviations of empirical measures of diffusion processes.
  In a first part, we present conditions to obtain a large deviations principle (LDP)
  for a precise class of unbounded functions. This provides an analogue
  to the standard Cram\'{e}r condition in the context of diffusion processes, which turns out
  to be related to a spectral gap condition for a Witten--Schrödinger operator.
  Secondly, we study more precisely the properties of the Donsker--Varadhan rate functional
  associated with the LDP.
  We revisit and generalize some standard duality results as well as a more original
  decomposition of the rate functional with respect to the symmetric and antisymmetric parts
  of the dynamics.
  Finally, we apply our results to overdamped and underdamped Langevin dynamics, showing the
  applicability of our framework for degenerate diffusions in unbounded configuration
  spaces.}

\section{Introduction}
\label{sec:intro}

Empirical averages of diffusion processes and their convergence are commonly studied
in statistical mechanics, probability theory and machine learning. In statistical physics,
an observable averaged along the trajectory of a diffusion typically converges
to the expectation with respect to its stationary distribution, which provides some macroscopic
information on the system~\cite{keller1998equilibrium,lelievre2010free}.
For reversible dynamics, this convergence is known to be characterized by an entropy
functional~\cite{touchette2009large,barato2015formal}, which generalizes results for small
fluctuations such as the central limit theorem~\cite{komorowski2012fluctuations} or
Berry-Esseen type inequalities~\cite{nourdin2012normal}. It has been shown for some time that the
approach can be extended to nonequilibrium systems by considering generalized entropy and
free energy functionals, as provided by the theory of large
deviations~\cite{dembo2010large,ellis2007entropy,touchette2009large}. From a more computational
perspective, studying the convergence of empirical averages is an important problem
for the efficiency of Monte Carlo Markov Chain
methods~\cite{asmussen2007stochastic,robert2013monte,rey2015irreversible,moulines2019markov}.

Since its initiation by Cram\'{e}r in the 30s~\cite{cramer1938,cramer2018new}, large deviations
theory has been
given many extensions. The theory takes its origin in the study of fluctuations for sums of
independent variables, leading to the celebrated Sanov theorem~\cite{den2008large}. Interestingly,
the necessity of Cram\'{e}r's exponential moment condition for Sanov's theorem to hold in a
Wasserstein topology has been proved only recently~\cite{wang2010sanov}.

Due to the above mentioned applications,
it is natural to try to apply such a theory to diffusions, or more generally Markovian dynamics.
This is useful for instance in statistical physics, when considering Gallavotti--Cohen
fluctuation relations for irreversible
dynamics~\cite{gallavotti1995dynamical,lebowitz1999gallavotti,kurchan1998fluctuation}, as well
as for characterizing dynamical phase transitions in physical
systems~\cite{garrahan2009first,baek2018dynamical,nemoto2017finitephase,nyawo2017minimal}.
From a more computational perspective, studying the rate function associated with
a given dynamics is interesting for designing better sampling
strategies~\cite{dupuis2012infinite,rey2015irreversible,rey2016improving},
which is important for instance in a Bayesian framework~\cite{box2011bayesian,blei2017variational}
or for molecular dynamics~\cite{leimkuhler2016computation,lelievre2016partial}.
The approach can also be used for deriving concentration results such as Bernstein-type
inequalities~\cite{gao2014bernstein,birrell2018uncertainty} and uncertainty quantification
bounds~\cite{katsoulakis2017scalable,gourgoulias2017biased}.

However, proving a large deviations principle for correlated processes turns out
to be a difficult task. A milestone in the theory is the series of papers by Donsker
and Varadhan~\cite{donsker1975asymptoticI,donsker1975asymptoticII,donsker1976asymptoticIII,donsker1983asymptoticIV}
and the dual approach followed by Gärtner and Ellis~\cite{gartner1977large,ellis1984large}.
The strategy of the former works is to build explicitly lower and upper large deviations bounds from
the Girsanov theorem and the Tchebychev inequality~\cite{varadhan1984large}. On the other
hand, the Gärtner--Ellis theorem relies on the existence and regularity of a free energy
functional. This technique has been later related to optimal control problems through the
so-called weak convergence approach~\cite{dupuis2011weak,dupuis2018large}.

Whichever strategy is chosen, proving large deviations principles for empirical measures
of diffusions in unbounded configuration spaces remains difficult. Indeed, studying the stability
of unbounded Markov processes is already challenging, and often relies on Lyapunov function
techniques~\cite{meyn2012markov,mattingly2002ergodicity,bellet2006ergodic,hairer2011yet}. Such
a Lyapunov function can be interpreted as an energy associated with the system,
which decreases in average and provides a control on the excursions of the process far away from
the origin. This technique can be used for proving LDPs, see for
instance~\cite[Section~9]{varadhan1984large}
and~\cite{deuschel2001large,wu2001large,dupuis2018large}. However, the LDPs of the above mentioned works
are stated in the so-called strong (resp. weak) topology, \textit{i.e.} with respect to
the topology on measures associated with the convergence of measurable bounded (resp. continuous
bounded) functions. To the best of our knowledge, convergence in Wasserstein-like topologies
(\textit{i.e.} associated with unbounded functions) for diffusions
has only been addressed in~\cite{kontoyiannis2005large},
and~\cite[Section~2.2]{wu2001large}. Unfortunately,
the nonlinear approach of~\cite{kontoyiannis2005large} does not allow to characterize
precisely the set of functions for which the LDP holds, while~\cite{wu2001large} considers
a particular system (Langevin dynamics). In both cases, the rate function is not related to the
standard Donsker--Varadhan theory~\cite{donsker1975variational}. Our first result is to derive the
LDP in a weak topology associated with unbounded functions, under very natural conditions, and to express the rate function
in duality with a free energy. From a practical
point of view, this allows to compute the rate function from the free energy, a standard
procedure~\cite{giardina2006direct,touchette2009large,chetrite2015nonequilibrium,nemoto2016population,ferre2018adaptive}.

Once a large deviations principle has been derived, providing alternative expressions
of the rate function is an important problem. This can be useful for computing this function
more efficiently, or for interpreting some key aspects of the dynamics (such as irreversibility
for physical systems). Our first contribution in this direction is to
derive a variational representation of the rate function similar to the Donsker--Varadhan
formula~\cite{donsker1975variational}. This provides a variational representation of the principal
eigenvalue for any non-symmetric linear second order differential operator
associated with a diffusion, under confinement and regularity conditions.  To the best of our
knowledge, there is no such formula in an unbounded setting, a
fortiori for unbounded functions.
Finally, it has been shown in a pioneering work~\cite{bodineau2008large}, for a specific choice of dynamics, that
the above mentioned duality allows to decompose the rate function into two
parts: one corresponding to a ``reversible'' part and the other
to an ``irreversible part'' of the dynamics. We extend these results to
general diffusions by using Sobolev seminorms, a feature inspired by the small fluctuations
framework developed in~\cite{komorowski2012fluctuations}. This decomposition turns out to be useful
for various purposes. For illustration we apply it to study more precisely the rate function of the Langevin
dynamics, in particular its dependence on the friction both in the Hamiltonian and overdamped
limits.

We now sketch the main results of the paper, the precise setting being presented in Section~\ref{sec:setting}.

\paragraph{Main results.} Consider a diffusion process~$(X_t)_{t\geq 0}$ over a state
space $\X\subset\R^d$ with generator~$\Lc$, invariant probability
measure~$\mu$, and empirical measure
\begin{equation}
  \label{eq:Ltintro}
L_t:=\frac{1}{t}\int_0^t \delta_{X_s}\,ds, \quad t\geq 0,
\end{equation}
where~$\delta_x$ is the Dirac measure at $x\in\X$.

Our first contribution is to prove a large
deviations principle for the empirical measure~$(L_t)_{t\geq 0}$ in a
weak topology associated with an unbounded function $\kap:\X\to [1,+\infty)$.
That is, we prove the following type of long time scaling: for $\Gamma\subset \PX$,
\begin{equation}
  \label{eq:proba}
\proba \big( L_t \in \Gamma \big)\asymp \e^{- t \inf_{\nu\in\Gamma}\,I(\nu)},
\end{equation}
where~$I$ is a rate function. Here,~$\PX$ denotes the set of probability measures
on~$\X$, and the above scaling holds for the weak topology on~$\PX$ associated with
measurable functions~$f$ satisfying
\begin{equation}
  \label{eq:fintro}
\|f \|_{\Binfty_\kap}:=\sup_{x\in\X}\frac{|f(x)|}{\kap(x)}<+\infty.
\end{equation}

As is standard for LDPs on unbounded
state spaces~\cite{varadhan1984large,wu2001large}, our result relies on the existence of a
twice differentiable Lyapunov function $W:\X\to [1,+\infty)$ such that 
  \begin{equation}
      \label{eq:psiintro}
\Ps := - \frac{\Lc W}{W}
\end{equation}
has compact level sets (in other words, it goes to infinity at infinity). Unlike previous
works, where this condition implies the asymptotic equivalence~\eqref{eq:proba}
in the weak topology corresponding to the convergence of measures tested against
bounded test functions~\cite{varadhan1984large,dupuis2018large,wu2001large},
we show in Section~\ref{sec:LDP} that the LDP holds for the weak topology associated with
any cost function~$\kap$ controlled by~$\Ps$ (see Section~\ref{sec:setting} for details).
Moreover, the associated rate function $I:\mathcal{P}(\X)\to [0,+\infty]$, also called
entropy, reads
\[
\forall\,\nu\in \mathcal{P}(\X),\quad I(\nu) =
\sup_{\|f\|_{\Binfty_\kap}<+\infty}\ \big\{\nu(f) - \lambda(f) \big\},
\]
where
\begin{equation}
  \label{eq:SCGFintro}
\lambda(f)  = \lim_{t\to+\infty} \frac{1}{t}\log \E\left[ \e^{\int_0^t f(X_s)\,ds}\right],
\end{equation}
is the cumulant or free energy function.

We mention that our strategy relies on the
Gärtner--Ellis theorem, according to which the existence and regularity of~\eqref{eq:SCGFintro}
implies the large deviations principle. We actually show that~\eqref{eq:SCGFintro} is
well-defined because it matches the principal eigenvalue of the Feynman--Kac operator
\begin{equation}
  \label{eq:FKintro}
P_t^f :\varphi \mapsto \E\left[ \varphi(X_t)\,\e^{\int_0^t f(X_s)\,ds}\right].
\end{equation}
A key remark for defining the above operator is that the process
\begin{equation}
  \label{eq:martingaleintro}
M_t = W(X_t)\,\e^{-\int_0^t \frac{\Lc W}{W}(X_s)\,ds}
\end{equation}
is a local martingale, as noted by Wu in~\cite{wu2001large}. This allows to
define~\eqref{eq:FKintro} for functions~$\varphi$ such that $\|\varphi\|_{\Binfty_W}<+\infty$,
as soon as~$f$ is dominated by the function~$\Ps$ defined in~\eqref{eq:psiintro}.
As a result, for any such~$f$,  the operator~\eqref{eq:FKintro} can be shown to be
compact over the space of functions controlled by~$W$
(see~\cite{gartner1977large,ferre2018more}), and the functional~\eqref{eq:SCGFintro} is obtained
as the largest eigenvalue of the operator~\eqref{eq:FKintro} through a generalized Perron--Frobenius
theorem (the Krein--Rutman theorem~\cite{deimling2010nonlinear}).

The second part of our work consists in rewriting the rate function~$I$.
For this, we first show that
\begin{equation}
  \label{eq:Idonskerintro}
\forall\,\nu\in \mathcal{P}(\X),\quad I(\nu)
= \sup\left\{ -\intX  \frac{\Lc u}{u}\,d\nu,\ u\in \mathcal{D}^+(\Lc)\right\},
\end{equation}
where $\mathcal{D}^+(\Lc)$ is an appropriate domain defined in Section~\ref{sec:Donsker}. This formula
is similar to the one proved in~\cite{donsker1975variational}, but differs
by additional growth conditions in the definition of~$\mathcal{D}^+(\Lc)$. This
result leads to a variational formula for the largest eigenvalue~$\lambda(f)$ of the operator
$P_t^f$ defined on a suitable functional space through
\[
\lambda(f) = \sup_{\nu\in\mathcal{P}(\X)}\big\{\nu(f) - I(\nu)\big\}.
\]
We mention that the proof of~\eqref{eq:Idonskerintro} relies on the spectral
problem associated with the Feynman--Kac operator~\eqref{eq:FKintro}, and uses tools
from the recent work~\cite{ferre2018more}.

Finally, the variational representation~\eqref{eq:Idonskerintro} allows to
generalize the results of~\cite{bodineau2008large} by splitting~$I$ into two parts. More
specifically, denoting
by $\Lc = \LS + \LA$ the decomposition into symmetric and antisymmetric parts of the generator
considered on~$L^2(\mu)$, we obtain, for any $\nu\ll\mu$:
\[
I(\nu) = \frac{1}{4} \left| \log \frac{d\nu}{d\mu}\right|_{\H^1(\nu)}^2
+ \frac{1}{4} \left|\LA \left(\log \frac{d\nu}{d\mu}\right) \right|_{\H^{-1}(\nu)}^2,
\]
where $|\cdot|_{\H^1(\nu)}$ and $|\cdot|_{\H^{-1}(\nu)}$ refer to Sobolev seminorms defined
in Section~\ref{sec:setting}. Interestingly, the proof relies on a generalized Witten
transform performed in the variational representation~\eqref{eq:Idonskerintro}, which
we may therefore call variational Witten transform. This shows that, for a given invariant
measure, an irreversible dynamics ($\LA\neq 0$) produces more entropy than a reversible one, in
accordance with the second law of thermodynamics.
This decomposition is useful for instance to study the entropy production of the Langevin
dynamics, which is irreversible but has a particular structure. In this case, there is a natural
identification of the effect of the reversible and irreversible parts of the dynamics on fluctuations.

\paragraph{Organization of the work.}
The paper is organized as follows. In Section~\ref{sec:LDP} we prove the large deviations
principle  under Lyapunov
and regularity conditions. In Section~\ref{sec:Donsker}
we rewrite the rate function and give its decomposition into symmetric
and antisymmetric parts. Some examples of application are given in
Section~\ref{sec:applications}, in particular for overdamped and underdamped Langevin dynamics.
Section~\ref{sec:discussion} discusses possible extensions and connections with related works.
Finally, most of the proofs are postponed to Section~\ref{sec:proofs}.

\section{Large deviations principle}
\label{sec:LDP}

\subsection{Setting}
\label{sec:setting}
This section introduces the main notation used throughout the paper. 
We consider a diffusion process~$(X_t)_{t\geq 0}$ evolving in $\X= \R^d$
with $d\in\N\setminus\{0\}$, and satisfying the following stochastic differential equation (SDE):
\begin{equation}
  \label{eq:SDE}
dX_t = b(X_t)\, dt + \sigma (X_t)\, dB_t,
\end{equation}
where $b:\X \to \R^d$, $\sigma:\X \to \R^{d\times m}$ and~$(B_t)_{t\geq 0}$ is a
$m$-dimensional Brownian motion for some $m\in\N^*$.

\begin{remark}
The analysis can easily be extended
with appropriate modifications to other spaces~$\X$ such as $\X = \mathbb{T}^d$ or $\X=\mathbb{T}^d\times \R^d$,
where~$\mathbb{T}^d$ is the $d$-dimensional torus. The last case is motivated by
applications to the Langevin equation, where~$\mathbb{T}^d$ would be a bounded position space
and~$\R^d$ the unbounded momentum space (see Section~\ref{sec:langevin}).
\end{remark}

The generator of the
dynamics~\eqref{eq:SDE}, denoted by~$\Lc$, reads
\begin{equation}
  \label{eq:gen}
\Lc = b\cdot\nabla +   S : \nabla^2,\quad\mbox{ with }\ S=\frac{\sigma \sigma^T}{2},
\end{equation}
where $\sigma^T$ denotes the transpose of the matrix~$\sigma$ and~$\cdot$ is the
scalar product on~$\R^d$. Moreover,~$\nabla^2$ stands for the Hessian matrix, and for two
matrices~$A,B$ belonging to~$\R^{d \times d}$, we write $A:B=\mathrm{Tr}(A^TB)$.
The conditions on~$b$ and~$\sigma$ will be made precise in
Section~\ref{sec:results}. The function~$S$ takes values in the set of symmetric positive
matrices (not necessarily definite). We also introduce the \emph{carré du champ}
operator~\cite{bakry2013analysis}
associated with~$\Lc$ defined by, for two regular functions $\varphi$, $\psi$:
\begin{equation}
  \label{eq:carre}
\mathscr{C}(\varphi,\psi)= \frac{1}{2}\big( \Lc(\varphi\psi) -  \varphi \Lc \psi - \psi \Lc \varphi 
\big) = \nabla \varphi\cdot S \nabla \psi.
\end{equation}
We will use the space~$\Cinfc(\X)$ (resp.~$\Cb(\X)$) of smooth functions
with compact support (resp. continuous and bounded functions), as well as the space of smooth
functions growing at most polynomially and whose derivatives also grow at most polynomially:
\[
\mathscr{S}=\left\{\varphi\in \Cinfty(\X)\, \middle|\,\forall\,\alpha\in\N^d, \, \exists\,N>0\
\mbox{such that}\
\sup_{x\in\X}\,\frac{|\partial^{\alpha}\varphi(x)|}{(1 + |x|^2)^N}<+\infty \right\},
\]
where $\partial^{\alpha} = \partial_{x_1}^{\alpha_1} \dots \partial_{x_d}^{\alpha_d}$ with $\alpha = (\alpha_1,\dots\alpha_d)$.

The space of bounded measurable functions, denoted by~$\Binfty(\X)$, is endowed with the norm
\[
\| \varphi \|_{\Binfty} = \underset{x\in\X}{\sup}\ | \varphi(x) |.
\]
Moreover, we will need weighted function spaces and the corresponding probability measure
spaces, which commonly appear in Markov
chain theory~\cite{meyn2012markov,kontoyiannis2005large,hairer2011yet}.
For any measurable function $W:\X\to[ 1,+\infty)$ we define 
\begin{equation}
  \label{eq:BW}
\Binfty_W(\X) = \left\{
\varphi : \X\to \R \, \textrm{measurable} \ \middle| \ \|\varphi \|_{\Binfty_W}:=
\underset{x\in\X}{\sup}\ \frac{ |\varphi(x)|}{W(x)} <+\infty
\right\},
\end{equation}
and the associated space of probability measures
(see~\cite[Chapter~2]{rudin2006real} for duality results on measure spaces):
\begin{equation}
  \label{eq:PW}
\mathcal{P}_W(\X) = \Big\{ \nu \in\PX \ \Big| \ \nu(W) < +\infty
\Big\}.
\end{equation}
The associated weighted total variation distance is (see for instance~\cite{hairer2011yet}):
\begin{equation}
  \label{eq:dw}
\forall\, \nu,\eta\in \mathcal{P}_W(\X),\quad
\mathrm{d}_W(\nu,\eta) = \underset{\| \varphi \|_{\Binfty_W}\leq 1}{\sup} \left\{
\intX \varphi \, d\nu - \intX \varphi \, d\eta\right\} = \intX W(x) |\nu - \eta|(dx),
\end{equation}
where~$|\nu - \eta|$ denotes the total variation measure associated to~$\nu - \eta$,
see~\cite[Chapter~6]{rudin2006real}.

\begin{remark}
Note that the spaces~\eqref{eq:BW} and~\eqref{eq:PW} are defined for
an arbitrary measurable function $W\geq 1$.
It is possible to weaken the assumption $W\geq 1$ but we
will not need these refinements in this paper.
\end{remark}

We denote by $\tau$-topology the weak topology on~$\PX$ associated with the convergence of measures tested against functions belonging to~$\Binfty(\X)$ (we may also use the notation~$\sigma(\PX,\Binfty)$); see~\cite{deuschel2001large}. This means that for a sequence~$(\nu_n)_{n\in\N}$ in~$\PX$, $\nu_n\to\nu$ in the $\tau$-topology if $\nu_n(\varphi)\to\nu(\varphi)$ for any $\varphi \in \Binfty(\X)$. Recall that the $\tau$-topology is stronger than the usual weak topology $\sigma(\PX,\Cb(\X))$ on~$\PX$, which corresponds to the convergence $\nu_n(\varphi)\to\nu(\varphi)$ for any $\varphi \in \Cb(\X)$. The $\tau$-topology can be extended to account for convergence of measures tested against the larger class of functions $\varphi\in\Binfty_W(\X)$. We denote by~$\tau^W$ the associated topology $\sigma(\mathcal{P}_W(\X),\Binfty_W(\X))$, see~\cite{wu2001large,kontoyiannis2005large}.


We associate to the dynamics~$(X^{(x)}_t)_{t\geq 0}$ started from~$X_0^{(x)} = x \in \X$ the semigroup~$(P_t)_{t\geq 0}$ 
defined through
\begin{equation}
  \label{eq:Pt}
\forall\, \varphi\in\Binfty(\X),\quad (P_t\varphi)(x)= \E \left[ \varphi\left(X^{(x)}_t\right) \right],
\end{equation}
where~$\E$ stands for the expectation with respect to all realizations of the Brownian
motion in~\eqref{eq:SDE}. Let us mention that, with some abuse of notation but for the sake of readability, we will not write out explicitly the dependence of~$X_t$ on~$x$ in the proofs presented in Section~\ref{sec:proofs}, see the discussion at the beginning of this section. We say that $\mu\in\PX$ is invariant with respect to the dynamics~$(X_t^{(x)})_{t\geq 0}$
if $(\mu P_t)(\varphi) = \mu(\varphi)$ for any $\varphi \in \Cb(\X)$, with the notation
\[
(\mu P_t)(\varphi) =\mu(P_t \varphi) = \intX \E \left[ \varphi\left(X_t^{(x)}\right) \right]\mu(dx).
\]
This implies in particular that $\mu(\Lc \varphi) =0$ for $\varphi \in \Cinfc(\X)$, see~\cite[Proposition~9.2]{EK86}.

We now follow the path of~\cite[Chapter~2]{komorowski2012fluctuations} for defining other
useful functional spaces. For any probability measure~$\mu\in\PX$, let
\begin{equation}
  \label{eq:hilbert}
  L^2(\mu) = \left\{ \varphi \mbox{ measurable}\ \left| \ \intX |\varphi|^2\,d\mu <+\infty\right.
  \right\}.
\end{equation}
For $\varphi\in\Cinfc(\X)$, we introduce the seminorm
\begin{equation}
  \label{eq:normH1}
  | \varphi |_{\H^1(\mu)}^2 = \intX \mathscr{C}(\varphi,\varphi)\,d\mu,
\end{equation}
and the equivalence relation~$\sim_1$ through: $\varphi\sim_1\psi$ if and only if
$|\varphi - \psi |_{\H^1(\mu)}=0$. We denote by~$\H^1(\mu)$ the closure of~$\Cinfc(\X)$
quotiented by~$\sim_1$ for the norm~$|\cdot|_{\H^1(\mu)}$.
Note that~$\H^1(\mu)$ and~$L^2(\mu)$ are not subspaces of each other in general,
but $\H^1(\mu)\subset L^2(\mu) $ for instance if~$\mu$ satisfies a Poincaré inequality and~$S$ is positive
definite. The difference between~$L^2(\mu)$ and~$\H^1(\mu)$ is however important for
degenerate dynamics, see the application in Section~\ref{sec:langevin}. We now construct a
space dual to~$\H^1(\mu)$ with the same density argument by introducing the 
seminorm: for $\varphi\in \Cinfc(\X)$,
\[
| \varphi |_{\H^{-1}(\mu)}^2 = \sup_{\psi\in\Cinfc(\X)}
\left\{ 2 \intX \varphi\psi\,d\mu - | \psi |_{\H^1(\mu)}^2\right\}.
\]
We define similarly the equivalence relation~$\sim_{-1}$ on~$\Cinfc(\X)$
by $\varphi\sim_{-1}\psi$ if and only if~$|\varphi - \psi|_{\H^{-1}(\mu)} =0$.
The space~$\H^{-1}(\mu)$ is then the closure of~$\Cinfc(\X)$ quotiented by~$\sim_{-1}$.
This is actually the dual space of~$\H^{1}(\mu)$,
see~\cite[Section~2.2, Claim~F]{komorowski2012fluctuations}.

Let us relate~$\H^1(\mu)$ to the more standard~$H^1(\mu)$ Sobolev space.
If~$\mu$ is invariant with respect to~$\Lc$ then, for $\varphi\in\Cinfc(\X)$,
it holds (using that $\Lc(\varphi^2) = 2\varphi\Lc\varphi + 2 \mathscr{C}(\varphi,\varphi)$)
\[
| \varphi |_{\H^1(\mu)}^2 = 2 \intX \varphi(-\Lc \varphi) \,d\mu.
\]
In particular, when $S=\id$ we have
\[
| \varphi |_{\H^1(\mu)}^2 =  \intX |\nabla \varphi|^2\,d\mu.
\]
In this case, $|\cdot|_{\H^1(\mu)}$ is the standard~$H^1(\mu)$ Sobolev
seminorm~\cite{lelievre2016partial}. An in-depth discussion on the space~$\H^1(\mu)$
and its use for proving central limit theorems for Markov processes is provided in~\cite[Chapter~2]{komorowski2012fluctuations}.

\begin{remark}
The space $\H^{-1}(\mu)$ has a role comparable to the subspace~$L_0^2(\mu)$ of functions in~$L^2(\mu)$ with average zero with respect to~$\mu$ since $\H^{-1}(\mu) \cap L^2(\mu) \subset L_0^2(\mu)$ (but of course the functions of~$\H^{-1}(\mu)$ do not belong to~$L^2(\mu)$ in general). Assume indeed that $\varphi\in L^2(\mu)$ (so $\varphi\in  L^1(\mu)$), $\intX\varphi\,d\mu\geq 0$ (which is
not restrictive upon considering $-\varphi$) and
$|\varphi|_{\H^{-1}}<+\infty$. We may choose $\psi\in\Cinfc(\X)$ such that
\[
\psi(x) = \left\{
\begin{aligned}
  1, & \quad \mathrm{if}\ |x|\leq 1, \\
  0, & \quad \mathrm{if} \ |x|\geq 2,
\end{aligned}
\right.
\]
and set $\psi_n(x) = n\psi(x/n)$, so $|\psi_n|_{\H^1(\mu)}\leq C$ for some constant $C>0$
independent of~$n$. The definition of~$\H^{-1}(\mu)$ shows that
\[
|\varphi|_{\H^{-1}(\mu)}\geq 2 n \int_{|x|\leq 2n} \psi \varphi \,d\mu - C.
\]
By the dominated convergence theorem it holds
\[
\int_{|x|\leq 2n}\psi \varphi \,d\mu\xrightarrow[n\to +\infty]{} \intX \varphi\,d\mu\geq 0.
\]
Since $|\varphi|_{\H^{-1}(\mu)}<+\infty$, we obtain by letting $n\to +\infty$ 
that $\mu(\varphi)=0$.
\end{remark}

%

We also introduce some notation concerning the growth of functions.
A function $f:\X\to\R$ is said to have \emph{compact level sets} if for any $M\in \R$ the set
\[
\big\{ x\in\X\, \big| \, f(x)\leq M \big\}
\]
is compact (with the convention that~$\emptyset$ is compact). 
A function~$g$ is said to be negligible with respect to~$f$ (denoted by $g\ll f$)
if~$f/g$ has compact level sets, and~$g$ is said to be equivalent to~$f$ (denoted by $g\sim f$) if
there exist constants~$c,c'> 0$ and $R,R'\in\R$ such that
\[
\forall\, x \in \X, \quad c' g(x) - R' \leq f(x) \leq c g(x) + R.
\]

\begin{remark}
  The above definitions are useful when the state space~$\X$ is unbounded.  A sufficient condition
  for~$f$ to have compact level sets in this case is for this function to be lower semicontinuous
  and to go to infinity at infinity (\textit{i.e.} to be coercive).
  If~$\X$ was bounded, all these criteria would be automatically met for smooth functions.
\end{remark}

Finally, we denote by $\liminfb$ and $\limsupb$ the inferior and superior limits respectively,
while for a subset $A\subset \mathcal{Y}$ of a topological space~$\mathcal{Y}$, $\mathring{A}$
and~$\bar{A}$ denote the interior and closure of~$A$ for the chosen topology
on~$\mathcal{Y}$. The function $\ind_A$ denotes the indicator function of the set~$A$,
\textit{i.e.} $\ind_A(x)=1$ if $x\in A$ and $\ind_A(x)=0$ otherwise. 
For a Banach space~$E$,~$\mathcal{B}(E)$
refers to the Banach space of bounded linear operators over~$E$ with the usual norm.
We recall some elements of large deviations theory in Appendix~\ref{sec:tools} for the reader's
convenience.

\subsection{Statement of the main results}
\label{sec:results}

The large deviations principle relies on three standard assumptions: hypoellipticity of the
generator, irreducibility of the dynamics, and a Lyapunov condition.

We start with our hypoellipticity assumption 
(which could certainly be relaxed for particular applications, see for
instance~\cite{wu2001large}). It will be useful for proving regularity of the
Feynman--Kac semigroup in Lemma~\ref{lem:regularity}. We denote by $A^\dag$ the adjoint of a (closed) operator~$A$ considered on~$L^2(dx)$.

\begin{assumption}[Hypoellipticity]
  \label{as:regularity}
  The functions~$b$ and~$\sigma$ in~\eqref{eq:SDE} belong to~$\mathscr{S}^d$ and
  $\mathscr{S}^{d\times m}$, respectively, and the generator~$\Lc$ defined
  in~\eqref{eq:gen} satisfies the hypoelliptic Hörmander condition. More precisely,
  $\Lc$ can be written as
  \begin{equation}
    \label{eq:Lhypo}
  \Lc = \sum_{i=1}^d A_i^\dag A_i + A_0,
  \end{equation}
  where $(A_i)_{i=0}^d$ are first order differential operators with coefficients
  belonging to~$\mathscr{S}$ such that the family
\[
\big\{A_i\big\}_{i=1}^d\bigcup\big\{ [A_i,A_j]\big\}_{i,j=0}^d\bigcup
\big\{ [[A_i,A_j],A_k]\big\}_{i,j,k=0}^d\hdots
\]
 spans $\R^d$ at any $x\in\X$ for a finite number of commutators~$n_x\in\N$.
\end{assumption}

This assumption is natural  in practical situations,
as illustrated in the applications of Section~\ref{sec:applications} covering elliptic and
hypoelliptic diffusions, see~\cite{hormander1967hypoelliptic,eckmann2003spectral,bellet2006ergodic}
for details. Note that excluding the operator~$A_0$ from the first family means that, if~$\Lc$
satisfies Assumption~\ref{as:regularity}, $\partial_t + \Lc$ is hypoelliptic and the
transition kernel of~$(X_t)_{t\geq 0}$ has a smooth density for any $t>0$.

The regularity requirement comes together with a controllability condition (recall that~$\sigma$ takes values
in~$\R^{d\times m }$).
\begin{assumption}[Controllability]
  \label{as:control}
  For any $x,y\in\X$ and $T>0$, there exists a control $u \in C^0([0,T],\R^m)$
  such that the path~$\phi \in C^0([0,T],\X)$ defined as
  \begin{equation}
    \label{eq:path}
  \left\{
  \begin{aligned}
    \phi(0) & = x, \\
     \dot{\phi}(t) &= b(\phi(t)) + \sigma(\phi(t))u(t),
  \end{aligned}
  \right.
  \end{equation}
  is well-defined and satisfies $\phi(T) = y$.
\end{assumption}

Assumption~\ref{as:control} together with Assumption~\ref{as:regularity} implies that the
process is irreducible, \textit{i.e.} that the
transition density of~$(X_t)_{t\geq 0}$ is everywhere positive (by adapting the argument
of~\cite[Proposition~8.1]{bellet2006ergodic}), which will be used
in Lemma~\ref{lem:irreducibility}. Note that constructing a control~$u \in C^0([0,T],\X)$ may be
difficult in general~\cite{jurdjevic1997geometric}. However,
for the overdamped and underdamped Langevin dynamics we are interested in, building such a control
turns out to be guenuinely feasible,
see~\cite{mattingly2002ergodicity,talay2002stochastic,bellet2006ergodic,lelievre2016partial,lu2019geometric} and 
references therein. Let us mention that the above two assumptions are standard for proving
LDPs~\cite{varadhan1984large,wu2001large}.

A recurrent idea when studying Markov chain stability and large deviations on
an unbounded state space is to reduce the analysis to a compact set and to control the excursions
of the dynamics out of this set with a Lyapunov function~\cite{meyn2012markov,wu2001large}. 
Our Witten--Lyapunov condition for the dynamics reads as follows (for the terminology,
see Remark~\ref{rem:cramer} below).

\begin{assumption}[Witten--Lyapunov condition]
  \label{as:lyapunov}
  There exists a function $W:\X\to [1,+\infty)$ of class~$C^2(\X)$, with compact
    level sets and such that 
    \begin{equation}
      \label{eq:lyapunov}
    \Ps:= - \frac{\Lc W}{W}
    \end{equation}
    has compact level sets. Moreover, there exists a~$C^2(\X)$
    function~$\mathscr{W}:\X\to [1,+\infty)$ such that, for some constants $C_1 >0$, $C_2\in\R$,
  \begin{equation}
    \label{eq:restriction}
    \mathscr{W}^2 \leq C_1 W,\quad
    \Ps\sim -\frac{\Lc \mathscr{W}}{\mathscr{W}}, \quad
  - 2 \frac{\Lc \mathscr{W}}{\mathscr{W}} \leq - \frac{\Lc W}{W} + C_2.
  \end{equation}    
  In all what follows, we consider an arbitrary function~$\kap:\X\to[1;+\infty)$
    belonging to~$\mathscr{S}$ such that:
    \begin{itemize}
    \item  $\kap\ll \Psi$;
    \item  either (i) $\kap$ bounded, or (ii) $\kap$~has compact level sets and there exists $c \in \mathbb{R}$ such that
      \begin{equation}
        \label{eq:Lyapunov_kappaW}
        \Lc(\kappa W) \leq c \kappa W.
      \end{equation}
    \end{itemize}
  \end{assumption}
 
\begin{remark}
  Note that the condition $\mathscr{W}^2 \leq C_1 W$ implies in particular that $\mathscr{W} \ll W$.
  In addition, since~$\kap\ll \Ps$ and~$\Ps\sim - \frac{\Lc \mathscr{W}}{\mathscr{W}}$,
  it holds~$\kap\ll- \frac{\Lc \mathscr{W}}{\mathscr{W}}$.
  These facts will be frequently used in the proofs.  Moreover the
  conditions~\eqref{eq:restriction} are not restrictive for exponential-like
  Lyapunov function as shown in Proposition~\ref{prop:nonlin} below -- the idea being
  that~$\mathscr{W}$ can be set to~$\sqrt{W}$. The condition~\eqref{eq:Lyapunov_kappaW} also
  typically holds because~$W$ is chosen of exponential type while~$\kappa$ is a polynomial.
  In practice, the auxiliary function~$\mathscr{W}$
  is used to obtain some control in the proofs of Lemmas~\ref{lem:feynmankac}
  and~\ref{lem:regularity} (in particular to apply a Grönwall lemma).
  Assumption~\ref{as:lyapunov} could certainly be phrased differently, possibly
  with weaker conditions on the functions at stake.
\end{remark}

Although we stated Assumption~\ref{as:lyapunov} in order to fit standard conditions when considering
large deviations on unbounded state spaces~\cite{varadhan1984large,wu2001large}, in practice
it can be obtained from a non-linear Lyapunov condition in the spirit
of~\cite{kontoyiannis2005large} and~\cite[Condition 2.2]{dupuis2018large}. This is the purpose
of the next proposition, whose proof is postponed to Appendix~\ref{sec:nonlin}.

\begin{proposition}
  \label{prop:nonlin}
  Assume that there exists $V\in \mathscr{S}$ such that:
  \begin{itemize}
  \item $V$ has compact level sets;
  \item $|\sigma^T\nabla V|$ has compact level sets;
  \item for any~$\theta\in(0,1)$,
  \begin{equation}
    \label{eq:nonlin}
  -\Lc V - \frac{\theta}{2}|\sigma^T\nabla V|^2\sim  |\sigma^T\nabla V|^2,
  \end{equation}
  \end{itemize}
  Then Assumption~\ref{as:lyapunov} is satisfied with
  \[
  W(x) = \e^{\theta V(x)},\qquad \mathscr{W}(x) = \e^{\varepsilon V(x)},
  \]
  for $\theta\in(0,1)$ and $\varepsilon< \theta/2$ small enough. In this case it holds
  \[
  \Ps \sim |\sigma^T\nabla V|^2.
  \]
  Moreover, condition~\eqref{eq:Lyapunov_kappaW} holds true for any function
  $\kap:\X\to[1,+\infty)$ of class~$\mathscr{S}$ such that either (i)~$\kap$ is bounded
    or (ii)~$\kap$ has
    compact level sets, satisfies $\kap\ll \Psi$ and there exists~$C \geq 0$ with
    \begin{equation}
      \label{eq:condkappa}
     \Lc \kap \leq C \kap,\qquad |\sigma^T \nabla \log \,\kap|\leq C.
    \end{equation}
\end{proposition}

Note that~\eqref{eq:nonlin} means that the term~$-\Lc V$ coming from the dynamics must compensate
the quadratic loss proportional to~$|\sigma^T\nabla V|^2$. We also mention that the
condition~\eqref{eq:condkappa} is not restrictive in general since it is typically satisfied by polynomial-like functions~$\kap$.

A first consequence of Assumptions~\ref{as:regularity} to~\ref{as:lyapunov} is the ergodicity of the dynamics, whatever the initial distribution for~$X_0$. 

\begin{proposition}
  \label{prop:ergodicity}
  Under Assumptions~\ref{as:regularity},~\ref{as:control}
  and~\ref{as:lyapunov},~\eqref{eq:SDE} has a global strong solution, and the
  process~$(X_t^{(x)})_{t\geq 0}$ admits a unique invariant probability measure
  $\mu\in \mathcal{P}_W(\X)$. This
  measure has a positive $\Cinfty(\X)$-density with respect
  to the Lebesgue measure: there exists $\rho^\mu\in\Cinfty(\X)$ with $\rho^{\mu}>0$ such
  that $\mu(dy)=\rho^{\mu}(y)\,dy$. Moreover, the dynamics is ergodic with respect to~$\mu$:
  there exist $C,c>0$ such that
  \[
  \forall\, t\geq 0,\quad \forall\, \varphi \in \Binfty_W(\X),\quad 
  \big\| P_t\varphi - \mu(\varphi) \big\|_{\Binfty_W} \leq C \e^{- c t}
  \|\varphi - \mu(\varphi)\|_{\Binfty_W}.
  \]
  Equivalently,
  \[
  \forall\, t \geq 0, \quad 
  \forall\, \nu \in \mathcal{P}_W(\X), \quad
  \mathrm{d}_W(\nu P_t, \mu) \leq C \e^{- c t}\mathrm{d}_W(\nu,\mu).
  \]
\end{proposition}

\begin{proof}
  The existence of a unique local strong solution is standard when Assumption~\ref{as:regularity}
  holds, see~\cite[Chapter~IX, Exercise~2.10]{revuz2013continuous}. Assumption~\ref{as:lyapunov}
  then implies the existence of $a>0$, $b\in \R$ such that
  \[
  \Lc W \leq - a W + b,
  \]
  and global existence can be deduced from the above Lyapunov inequality~\cite{bellet2006ergodic}.
  The end of the proof is a direct application of~\cite[Theorem 8.9]{bellet2006ergodic}
  since Assumption~\ref{as:control} together with Assumption~\ref{as:regularity} ensures
  irreducibility.
\end{proof}

We can now present the large deviations principle associated with the
empirical measure of the process~$(X_t^{(x)})_{t\geq 0}$ with respect to its invariant
measure~$\mu$. Recall that the empirical measure of the process is defined by
\begin{equation}
  \label{eq:Lt}
L_t^{(x)} := \frac{1}{t} \int_0^t \delta_{X_s^{(x)}}\, ds,
\end{equation}
where $\delta_y$ denotes the Dirac mass at $y\in\X$. 
When one considers large deviations principles for empirical averages of
the form~\eqref{eq:Lt}, the topology on probability measures has to be specified. 
As mentioned in the introduction, most of the LDPs are stated in topologies
associated with bounded measurable functions (resp. continuous bounded), the so-called strong
topology or $\tau$-topology (resp. weak topology). We now prove that, in our setting,
a LDP holds in the  $\tau^\kap$-topology defined in Section~\ref{sec:setting}, for any
function~$\kap$ satisfying Assumption~\ref{as:lyapunov}.
The proof of Theorem~\ref{theo:LDP} is presented in Section~\ref{sec:proofLDP}. We
recall that a rate function is said to be good if its level sets are compact.

\begin{theorem}
  \label{theo:LDP}
  Suppose that Assumptions~\ref{as:regularity},~\ref{as:control}
  and~\ref{as:lyapunov} hold true, and
  consider a function~$\kap$ as in Assumption~\ref{as:lyapunov} and $x\in\X$ fixed. Then, the functional
  \begin{equation}
    \label{eq:SCGF}
    f\in\Binfty_\kap(\X)\mapsto \lambda(f) := \underset{t\to + \infty}{\lim}\,
    \frac{1}{t} \log \E\left[ \e^{\int_0^t f(X_s^{(x)})\, ds} \right] 
  \end{equation}
  does not depend on~$x$, is well-defined, convex and finite, and~$(L_t^{(x)})_{t\geq 0}$ satisfies a LDP in the
  $\tau^\kap$-topology with the good rate function defined by:
  \begin{equation}
    \label{eq:Ifenchel}
    \forall\, \nu\in \mathcal{P}(\X), \quad
    I(\nu) = \left\{
    \begin{aligned}
      &\underset{f \in \Binfty_\kap}{\sup} \ \big\{ \nu(f) - \lambda (f) \big\}, \quad \mbox{if}\  
     \nu\in \mathcal{P}_\kap(\X)\ \mbox{and}\ \nu \ll \mu,
    \\ &+ \infty,  \quad \mbox{otherwise}. 
    \end{aligned}
    \right.
    \end{equation}
  More precisely, for any $\tau^\kap$-measurable set $\Gamma \subset \mathcal{P}(\X)$ and any $x \in \X$, it holds
  \begin{equation}
    \label{eq:LDP}
    - \underset{\nu\in\mathring{\Gamma}}{\inf}\, I(\nu)
    \leq \
    \underset{t\to +\infty}{\liminfb}\ \frac{1}{t}\, \log\, \proba \left( L_t^{(x)} \in \Gamma \right)
    \ \leq \
    \underset{t\to +\infty}{\limsupb}\ \frac{1}{t}\, \log\, \proba \left( L_t^{(x)} \in \Gamma \right)
    \ \leq 
    - \underset{\nu\in\bar{\Gamma}}{\inf}\, I(\nu),
  \end{equation}
  where the interior and closure of~$\Gamma$ are taken with respect to the $\tau^\kap$-topology.
  Finally, for any~$\nu\in\mathcal{P}(\X)$, it holds $I(\nu)=0$ if and only if $\nu=\mu$; and,
  for any sequence $(t_n)_{ n\geq 1}$ such that $t_n/\log(n) \to +\infty$ as $n \to +\infty$, it holds
  \begin{equation}
    \label{eq:almostsure}
  L_{t_n}^{(x)} \xrightarrow[n\to+\infty]{} \mu  
  \end{equation}
  almost surely in the $\tau^\kap$-topology. 
\end{theorem}

Our conclusion is in essence close to that of~\cite{kontoyiannis2005large}, but the conditions
to reach it seem more natural to us
and correspond to usual conditions for proving large deviations principles in an unbounded state
space, see~\cite{wu2001large,dupuis2018large} and~\cite[Section~9]{varadhan1984large}.
In particular, they allow to derive the duality representation~\eqref{eq:Ifenchel}, and
we do not need to consider non-linear operators. Our strategy (presented in
Section~\ref{sec:proofLDP}) relies on the Gärtner--Ellis
theorem~\cite{gartner1977large,ellis1984large,ellis2007entropy,dembo2010large}, for which
the existence of the free-energy~\eqref{eq:SCGF} is a key element. The originality
of our work is to make use of the local martingale~\eqref{eq:martingaleintro} introduced by
Wu~\cite{wu2001large} in order to solve the spectral problem associated with the
Feynman--Kac operator, which proves the existence of the limit in~\eqref{eq:SCGF}. This directly provides
the LDP in the $\tau^\kap$-topology by duality.
However, there may be cases in which a LDP holds although the conditions of the Gärtner--Ellis
theorem are not satisfied, for instance in the framework of the Sanov theorem~\cite{wang2010sanov},
so our conditions may not be necessary.

Let us also mention that, in addition to~\eqref{eq:almostsure}, we also show for completeness
in the proof of Theorem~\ref{theo:LDP} that~$(L_t^{(x)})_{t\geq 0}$ almost surely spends
a time of finite Lebesgue measure outside any $\tau^\kap$-open set around~$\mu$.

Another advantage of our approach is to characterize precisely the set of functions
for which a LDP holds from the standard condition on~$\Ps$ defined
in~\eqref{eq:lyapunov}, like in~\cite{donsker1975asymptoticI,varadhan1984large}. This condition
is also used in~\cite[Corollary~2.3]{wu2001large} for proving a level~1 LDP for Langevin
dynamics. We present
below a clear connection with a spectral gap condition for the Witten--Schrödinger operator in
the reversible case. The comparison with Cram\'{e}r's condition for independent
variables highlights the effect of correlations on fluctuations.

\begin{remark}[Reversible processes, Witten Laplacian and Cram\'{e}r's condition]
  \label{rem:cramer}
  Consider the following reversible diffusion
  \[
  dX_t = - \nabla V(X_t)\,dt +\sqrt{2}\,dB_t,
  \]
  where $V:\X\to \R $ is a smooth potential with compact level sets. The generator of
  this dynamics is $\Lc = - \nabla V\cdot \nabla + \Delta$ and its invariant probability measure
  reads~$\mu(dx)=Z^{-1}\,\e^{-V(x)}dx$, where we assume that
  \[
  Z= \intX \e^{-V(x)}\,dx<+\infty.
  \]
  Define
  \[
  W_{\theta}(x)=\e^{\theta V(x)},
  \]
  for some~$\theta\in(0,1)$. This is a standard choice for obtaining
  compactness of the evolution operator~\cite[Section~8]{bellet2006ergodic}, and optimal
  control representations of rate functions~\cite{dupuis2018large}, see also
  Proposition~\ref{prop:nonlin}. An easy computation shows that
  \begin{equation}
    \label{eq:revW}
  \Ps_\theta=-\frac{\Lc W_{\theta}}{W_{\theta}}= \theta ( 1 - \theta) |\nabla V|^2 - \theta\Delta V.
  \end{equation}
  However, we also know~\cite{witten1982supersymmetry} that the
  generator~$\Lc$ considered on~$L^2(\mu)$ is unitarily equivalent to the operator
  \[
  \widetilde{\Lc}:=\e^{-\frac{V}{2}} \Lc\big( \e^{\frac{V}{2}}\cdot\big),
  \]
  defined on~$L^2(dx)$ (a procedure also called
  symmetrization~\cite[Section~4.3]{touchette2018introduction}), which is actually the opposite of
  the Witten Laplacian~\cite{witten1982supersymmetry,helffer2006semi}:
  \begin{equation}
    \label{eq:schrodinger}
    \widetilde{\Lc} = \Delta - \frac{1}{4}|\nabla V|^2 +\frac{1}{2} \Delta V =
    -\left( -\Delta + \Psi_{\frac{1}{2}}\right).
  \end{equation}
  In this case, the condition for~\eqref{eq:revW} to have compact
  level sets when $\theta = 1/2$ is actually equivalent to a confinement condition (or spectral gap
  condition~\cite{helffer2013spectral}) for the Witten--Schrödinger operator~$\widetilde{\Lc}$
  defined in~\eqref{eq:schrodinger}. In that sense,
  Assumption~\ref{as:lyapunov} is a natural generalization of a spectral gap condition for the
  Witten Laplacian in the case of possibly non-reversible dynamics. This is why we call
  Assumption~\ref{as:lyapunov} a Witten--Lyapunov condition.

  We now compare this Witten--Lyapunov condition to Cram\'{e}r's exponential moment condition
  in the case of independent variables of law~$\mu$. Consider a smooth potential~$V(x)$ which behaves as~$|x|^q$
  for $q>1$ outside a ball~$B(0,r)$ centered on the origin. Assumption~\ref{as:lyapunov} is thus satisfied by application
  of Proposition~\ref{prop:nonlin}. The standard
  Cram\'{e}r condition in the case of independent variables~$(X_i)_{i\geq 0}$ states that the empirical measure
  \[
  \frac{1}{n}\sum_{i=1}^n \delta_{X_i}
  \]
  satisfies a large deviations principle in the~$\tau^\kap$-topology if and only
  if~\cite[Theorem~1.1]{wang2010sanov}:
  \[
  \forall\,\theta\in\R,\quad  \intX \e^{\theta \kap}d\mu<+\infty.
  \]
  For $\mu(dx)= Z^{-1} \e^{-V(x)}dx$, a
  sufficient condition for the above condition to hold is
  to choose a smooth function $\kap$ behaving as $1 + |x|^\alpha$ with
  $0\leq \alpha <q$. On the other hand, the Witten--Lyapunov potential~\eqref{eq:revW}
  reads in this case
  \[
  \forall x \in \X \backslash B(0,r), \qquad \Ps_{\theta} (x) = \theta(1 - \theta) q^2 |x|^{2(q-1)} - \theta q(q+d-2)|x|^{q-2},
  \]
  so that we may choose $\kap(x)$ behaving as $1 + |x|^{\alpha}$ for $0\leq \alpha < 2(q-1)$. When comparing the
  two conditions, we obtain the following different situations depending on~$q$:
  \begin{itemize}
  \item $q>2$ (super-Gaussian case): $2(q-1)>q$, the Witten--Lyapunov condition is
    \emph{less restrictive} than Cram\'{e}r's condition;
  \item $q=2$ (Gaussian case): $2(q-1)=q$, the two conditions are equivalent;
  \item $q\in(1,2)$ (sub-Gaussian case): $2(q-1) < q$, the Witten--Lyapunov condition
    is \emph{more restrictive} than Cram\'{e}r's condition.
  \end{itemize}
  This simple example shows that considering a correlated system
  instead of independent variables has a non-trivial effect on the stability of the
  system. Depending on the confinement potential, the Witten--Lyapunov condition
  for~\eqref{eq:revW} to have compact level sets can be more or less restrictive
  than Cram\'{e}r's condition for independent variables distributed according to the invariant
  measure~$\mu$. Finally, we remark that for $q\in(1,3/2)$, the process is \emph{heavy-tailed} in
  the sense that $2(q-1) < 1$ and the observable $f(x)=x$ (assuming $d=1$) does not satisfy a LDP.
  In other words, the average position of the process defined by
  \[
  \frac{1}{t}\int_0^t X_s\,ds
  \]
  cannot be shown to satisfy a large deviations principle at speed~$t$ with our arguments.

  We finally mention that, in the case where the observable~$f$ grows faster at infinity than the
  potential~$\Ps$, it seems possible to derive a level~1 large deviations principle at a speed
  smaller than~$t$. We refer to~\cite{nickelsen2018anomalous} for a recent account dealing with
  the case of an Ornstein--Uhlenbeck process, and to~\cite{bordenave2014large,augeri2017heavy}
  for related issues.
\end{remark}

We close this section with a practical corollary of Theorem~\ref{theo:LDP} which
generalizes the level~1 LDP proved in~\cite[Corollary~2.3]{wu2001large}.

\begin{corollary}[Level~1 large deviations principle]
  \label{cor:contract}
  Suppose that Assumptions~\ref{as:regularity},~\ref{as:control} and~\ref{as:lyapunov} hold
  true and consider a function~$f\in\Binfty_\kap(\X)$. Fix $x \in \X$. Then, the function
  \begin{equation}
    \label{eq:lambdaloc}
  \theta\in\R\mapsto  \lambda_f(\theta):= \lim_{t\to +\infty}\frac{1}{t}\log
  \E\left[ \e^{\theta \int_0^t f(X_s^{(x)})\,ds}\right]
  \end{equation}
  is well-defined and differentiable, and does not depend on~$x$.
  Moreover,~$L_t^{(x)}(f)$ satisfies a large deviations principle in~$\R$ at speed~$t$
  with good rate function given by
  \begin{equation}
    \label{eq:contract}
  \forall\,a\in \R,\quad I_f(a) = \inf\big\{ I(\nu), \ \nu\in\mathcal{P}(\X),\ \nu(f)=a \big\},
  \end{equation}
  where~$I$ is defined in~\eqref{eq:Ifenchel}. Finally, it holds
  \begin{equation}
    \label{eq:Ifenchelloc}
  I_f(a) = \sup_{\theta \in \R}\ \big\{\theta a - \lambda_f(\theta )\big\}.
  \end{equation}
\end{corollary}

Corollary~\ref{cor:contract} is useful for practical applications, since~\eqref{eq:Ifenchelloc}
is a natural way to estimate the rate function~$I_f$ associated with an observable~$f$,
see for instance~\cite{giardina2006direct,rousset2006control,tailleur2008simulation,chetrite2015nonequilibrium,ferre2018adaptive}.

\begin{proof}
  For $f\in\Binfty_\kap(\X)$, the application $L_t^{(x)}\in \mathcal{P}_\kap(\X) \mapsto L_t^{(x)}(f) \in\R$
  is continuous in the $\tau^\kap$-topology~\cite[Lemma~3.3.8]{deuschel2001large}.
  Therefore,~$L_t^{(x)}(f)$ obeys a large deviations
  principle in~$\R$ by the contraction principle~\cite[Theorem~4.2.1]{dembo2010large},
  with good rate function given by~\eqref{eq:contract}. Moreover, one can redo the
  proofs leading to Theorem~\ref{theo:LDP} and show that~$\lambda_f$ defined
  in~\eqref{eq:lambdaloc} is smooth and well-defined on~$\R$. This implies that a LDP with good rate
  function~\eqref{eq:Ifenchelloc} holds through the Gärtner--Ellis theorem applied in~$\R$. Since
  the rate function is unique, the expressions~\eqref{eq:contract} and~\eqref{eq:Ifenchelloc}
  coincide.
\end{proof}

\section{Decomposition of the rate function}
\label{sec:Donsker}

Our goal in this section is to rewrite~$I$ in various ways, which is useful for theoretical
understanding and practical purposes. In Section~\ref{sec:DV}, we first show
an extension of the standard Donsker--Varadhan formulation for~$I$. This result is 
obtained by making use of the spectral analysis of the operator~$P_t^f$
for $f\in\Binfty_\kap(\X)$, which is presented in Section~\ref{sec:proofLDP}. We then apply this
result to obtain a variational representation for the principal eigenvalue~$\e^{t \lambda(f)}$ of~$P_t^f$.
Next, in Section~\ref{sec:irrever}, we split the expression
of the rate function according to the symmetric and antisymmetric parts of the dynamics,
extending the work~\cite{bodineau2008large} to general diffusions.  Such a
decomposition will prove useful in Section~\ref{sec:applications} to compare the entropy of
overdamped and underdamped Langevin dynamics. Most of the proofs of this section are
postponed to Section~\ref{sec:proofDonsker}.

\subsection{Donsker--Varadhan variational formula}
\label{sec:DV}

We start with the variational representation of the entropy.
Our proof, which can be found in Section~\ref{sec:IDV}, is an adaptation
of~\cite[Lemma 4.2.35]{deuschel2001large} relying on the Feynman--Kac semigroup and its spectral
elements. In order to state the result, we need to make sense of~$\Lc u$ for functions $u \in B^\infty_W(\X)$.
It turns out that the appropriate notion to this end
is the extended domain~$\mathcal{D}(\Lc)$ of the generator~$\Lc$ considered as an operator
on~$B^\infty_{W}(\X)$, defined in the following way: a function 
$\varphi \in\Binfty_W(\X)$ belongs to~$\mathcal{D}(\Lc)$ if and only if there
exists a measurable function $\phi:\X\to \R$ such that, for any $x\in\X$,
\begin{equation}
  \label{eq:integrability}
\int_0^t P_s |\phi|(x)\,ds < +\infty,
\end{equation}
and
\begin{equation}
  \label{eq:Dextended}
  P_t\varphi = \varphi +\int_0^t P_s \phi\,ds.
\end{equation}
In this case we write $\phi = \Lc \varphi$ (with some abuse of notation in view of 
the definition of~$\Lc$ as a differential operator in~\eqref{eq:gen}, but of course the
expressions coincide when~$\varphi$ is a smooth test function with compact support).

When the~$\tau$-topology is considered, such extended domains were already considered
for instance in~\cite{Wu00,wu2001large,kontoyiannis2005large}, see
also~\cite[Chapter~I, Definition~14.15]{davisbook}. For the unbounded functions we consider, one should
think of~$\phi = \Lc u$ as an element of~$\Binfty_{\kap W}(\X)$ (see the proof of
Lemma~\ref{lem:pcpal_eig_L_f} below, as well as the comments following Proposition~\ref{prop:IDV}). The integrability
condition~\eqref{eq:integrability} is reasonable in this context since~$(P_t)_{t\geq 0}$ is a well defined
semigroup on~$B^\infty_{\kap W}(\X)$ in view of the Lyapunov condition~\eqref{eq:Lyapunov_kappaW}.

We can now present the main result of this section.
\begin{proposition}
\label{prop:IDV}
  The rate function defined in~\eqref{eq:Ifenchel} admits the following representation:
  \begin{equation}
    \label{eq:IDV}
    \forall\, \nu\in \mathcal{P}(\X),\quad
    I(\nu) = \sup \left\{ - \intX \frac{ \Lc u}{u} \, d\nu, \ u\in\mathcal{D}^+(\Lc)\right\}, 
  \end{equation}
  where
  \begin{equation}
    \label{eq:Dplus}
    \mathcal{D}^+(\Lc) = \Big\{u\in  \mathcal{D}(\Lc)\cap C^0(\X) \ \Big|\ 
     u > 0,\, -\frac{\Lc u}{u}\in\Binfty_\kap(\X) \Big\}.
  \end{equation}
  In particular, the functional
  defined in~\eqref{eq:IDV} is equal to~$+\infty$ if~$\nu\notin\mathcal{P}_\kap(\X)$
  or~$\nu$ is not absolutely continuous with respect to~$\mu$.
\end{proposition}

This result is standard when~$\X$ is compact~\cite{donsker1975variational}, but
does not seem to be known for an unbounded space~$\X$ and for the~$\tau^\kap$-topology we consider.
In this situation the space~$\mathcal{D}^+(\Lc)$  has to be designed with some caution. Note that~$\mathcal{D}^+(\Lc)$
is not empty since it contains the functions of the
form $u=\e^{\psi}$ for $\psi\in\Cinfc(\X)$. Note also
that the last statement of Proposition~\ref{prop:IDV} is consistent with the Fenchel
definition~\eqref{eq:Ifenchel} of the rate function. In order to get some intuition on
the formula~\eqref{eq:IDV}, let us mention that the proof formally relies on replacing the maximum
over functions~$u\in\mathcal{D}^+(\Lc)$ by the supremum over eigenfunctions~$h_f$ satisfying
\[
(\Lc + f)h_f = \lambda(f)h_f,
\]
for $f\in\Binfty_\kap(\X)$. The above equation rewrites, since~$h_f>0$ (see Lemmas~\ref{lem:specLf} and~\ref{lem:pcpal_eig_L_f}),
\[
-\frac{\Lc h_f}{h_f} = f - \lambda(f).
\]
By integrating with respect to a measure~$\nu\in\mathcal{P}_\kap(\X)$ we find~\eqref{eq:IDV} on
the left hand side, and the Fenchel transform~\eqref{eq:Ifenchel} on the right hand side.
The functional spaces associated with~$f$ and~$h_f$ motivate the choice of~$\mathcal{D}^+(\Lc)$,
in particular the fact that $\Lc h_f = \lambda(f)h_f - fh_f \in B^\infty_{\kap W}(\X)$  (as the sum of an element
in~$B^\infty_W(\X)$ and the product of a function in~$B^\infty_W(\X)$ and another one in~$B^\infty_\kap(\X)$),
which allows to define~$\Lc h_f$ in the weak sense~\eqref{eq:Dextended}.


A natural consequence of Proposition~\ref{prop:IDV} is the following variational representation
for the cumulant function. The proof, postponed to Section~\ref{sec:proofcor}, relies
on the convexity of the cumulant function to invert the Fenchel transform~\eqref{eq:Ifenchel}.
\begin{corollary}
  \label{cor:variational}
  Suppose that Assumptions~\ref{as:regularity}, \ref{as:control} and~\ref{as:lyapunov}
  hold true, and consider
  $f\in\Binfty_\kap(\X)$. Then, 
  \begin{equation}
    \label{eq:variational}
  \lambda(f) = \sup_{\nu\in\mathcal{P}_\kap}\big\{ \nu(f) - I(\nu)\big\},
  \end{equation}
  where~$I$ is defined in~\eqref{eq:IDV}.
\end{corollary}

  Corollary~\ref{cor:variational} may seem anecdotal, but it provides a variational
  representation for the principal eigenvalue of non-symmetric diffusion operators, as
  pioneered by Donsker and Varadhan in their seminal paper~\cite{donsker1975variational} for a
  compact space~$\X$. To the best of our knowledge, this formula had not been shown in an unbounded
  setting, for which we need to introduce the ``generalized domain''~$\mathcal{D}^+(\Lc)$ defined
  in~\eqref{eq:Dplus}. However, our set of assumptions implies that~$\lambda(f)$ can be thought of as
  the largest eigenvalue of $\Lc+f$, and turns out to be isolated for any~$f$
  (because of the compactness of the resolvent provided by Lemma~\ref{lem:specLf}),
  whereas in~\cite{donsker1975variational},~\eqref{eq:variational} may be the supremum
  of the essential spectrum of the operator. This suggests that~\eqref{eq:variational} holds
  under weaker assumptions. A possible approach for generalizing our results may be to consider
  different methods for studying the long time behaviour of unnormalized semigroups,
  see for instance~\cite{champagnat2017general,bansaye2019non,champagnat2019practical},
  or to resort to more subtle spectral analysis
  tools~\cite{wu2000deviation,wu2004essential,gao2014bernstein,birrell2018uncertainty}.

\subsection{Entropy decomposition: symmetry and antisymmetry}
\label{sec:irrever}
Our goal is now to provide refined expressions for the rate function~$I$ in terms of
symmetric and antisymmetric parts of the dynamics, inspired in particular
by~\cite{bodineau2008large}. In the following, for any closed operator~$T$, we denote
by~$T^*$ its adjoint on~$L^2(\mu)$, where~$\mu$ is the invariant probability measure of the process,
as obtained in Proposition~\ref{prop:ergodicity}. Considering the generator~$\Lc$ of the
diffusion~\eqref{eq:SDE}, we can always decompose it into symmetric and antisymmetric parts
with respect to~$\mu$ through
\begin{equation}
  \label{eq:LSLA}
\Lc = \LS + \LA,\quad \LS = \frac{ \Lc + \Lc^*}{2}, \quad \LA = \frac{ \Lc - \Lc^*}{2}.
\end{equation}
It is important to note that~$\LA$ is a first order differential operator (and therefore
obeys the chain rule of first order differentiation). We assume here that the operators $\Lc,\LA,\LS$ admit $\Cinfc(\X)$ as a
common core (but the domains of these operators may be different).

The decomposition~\eqref{eq:LSLA} allows to separate the rate function~\eqref{eq:IDV} into two
parts. This is the purpose of the next key result, whose proof can be found in
Section~\ref{sec:Idecomp}. It is inspired by the computations
in~\cite[Proposition 2]{bodineau2008large}, which we simplify and generalize here through a
variational Witten transform and the use of the Sobolev spaces introduced in
Section~\ref{sec:setting}. The algebra of the proof also suggests to consider~$I(\nu)$ for
probability measures~$\nu$ of the form~$d\nu=\e^v\,d\mu$.

\begin{theorem}
  \label{theo:InormH}
  Suppose that Assumptions~\ref{as:regularity},~\ref{as:control} and~\ref{as:lyapunov} hold true,
  consider  a measure $\nu\in\mathcal{P}_\kap(\X)$ such that~$d\nu =\e^v\,d\mu$ with $v\in\H^{1}(\nu)$ and $\LA v \in\H^{-1}(\nu)$.
  Then, the rate function~$I$ defined
  in~\eqref{eq:IDV} admits the following decomposition:
  \begin{equation}
    \label{eq:Idecomp}
    I(\nu) = \IS(\nu) + \IA(\nu),
  \end{equation}
  where
  \begin{equation}
    \label{eq:Isym}
  \IS(\nu) =\frac{1}{4} |v|_{\H^1(\nu)}^2
  \end{equation}
  and
  \begin{equation}
    \label{eq:Iantisym}
  \IA(\nu) = \frac{1}{4} \big|\LA v\big|_{\H^{-1}(\nu)}^2.
  \end{equation}
\end{theorem}

Theorem~\ref{theo:InormH} expresses the rate function as the sum of dual norms
of the symmetric and antisymmetric parts of the dynamics. Note also that we consider a measure
of the form~$d\nu=\e^v \,d\mu$, that is the Radon--Nikodym derivative of~$\nu$ with respect
to~$\mu$ is positive. However, we believe that we can consider more general measures~$\nu$,
see Remark~\ref{rem:extendrho} in the proof. Since the measure~$\nu$
at hand appears both inside the norms and in the definition of the norms themselves, a
possibly clearer rewriting is the following:
\[
I(\nu) = \frac{1}{4} \left| \log \frac{d\nu}{d\mu}\right|_{\H^1(\nu)}^2
+ \frac{1}{4} \left|\LA \left(\log \frac{d\nu}{d\mu}\right) \right|_{\H^{-1}(\nu)}^2.
\]
Moreover,  the symmetric part of the rate function~\eqref{eq:Isym} can be written as a Fisher
information for the invariant measure~$\mu$, a standard result~\cite{gartner1977large}: denoting
by $\rho = d\nu/ d\mu$, it holds
\[
\IS(\nu) = 
\frac{1}{4} \intX \frac{\nabla \rho\cdot S \nabla \rho}{\rho}\, d\mu.
\]

The next corollary builds upon~\eqref{eq:Iantisym} by rewritting~$\IA$ using a Poisson equation,
which can be manipulated more easily. The proof can be found in Section~\ref{sec:Iantisym}.

\begin{corollary}
  \label{cor:Iantisym}
  Suppose that Assumptions~\ref{as:regularity},~\ref{as:control} and~\ref{as:lyapunov} hold true, and
  consider a measure~$\nu\in\mathcal{P}_\kap(\X)$ such that
  $d\nu = \e^v d\mu$ with $v\in\H^{1}(\nu)$ and $\LA v\in\H^{-1}(\nu)$.
  Then, the antisymmetric part of the rate function~\eqref{eq:Iantisym} reads
  \begin{equation}
    \label{eq:Ipoisson}
    \IA(\nu) =  \frac{1}{4} \intX \mathscr{C}(\psi_v, \psi_v)\, d\nu,
  \end{equation}
  where~$\psi_v$ is the unique solution in~$\H^1(\nu)$ to the Poisson equation
  \begin{equation}
    \label{eq:psiv}
  \widetilde{\nabla}(S  \nabla \psi_v) = \LA v,
  \end{equation}
  the symmetric matrix~$S$ being defined in~\eqref{eq:gen} and $\widetilde{\nabla}$ denoting
  the adjoint of the gradient operator in~$L^2(\nu)$. 
\end{corollary}

It has been known for a long
time~\cite{donsker1975variational} that the rate function of a reversible process is a Fisher
information as in~\eqref{eq:Isym}. The antisymmetric part of the rate function has been
less investigated, although an expression like~\eqref{eq:Ipoisson} already
appears in~\cite{gartner1977large} (see also~\cite{rey2015irreversible,bodineau2008large}).
However, our setting provides natural well-posedness conditions for both parts of the
rate function to be finite. Moreover, the
uniqueness of~$\psi_v$ is a consequence of the definition of~$\H^1(\nu)$ through equivalence
classes, see Section~\ref{sec:setting}.

Interestingly, the solution~$\psi_v$ of~\eqref{eq:psiv} can be
formally represented through~\cite{lelievre2016partial}
\[
\psi_v = \int_0^{+\infty} \e^{t \Lc_{\nu}} (\LA v) \,dt,
\]
where $\Lc_{\nu}= -  \widetilde{\nabla}(S  \nabla\, \cdot\,) $. 
The stochastic process~$(X_t^{\nu})_{t\geq 0}$ associated with~$\Lc_\nu$ is reversible with
respect to~$\nu$. Denoting by~$\e^{-V_{\nu}}$ the density of~$\nu$ with respect to the Lebesgue
measure,~$(X_t^{\nu})_{t\geq 0}$ is solution to the following SDE:
\[
dX_t^{\nu} = - S \nabla V_\nu(X_t^{\nu})\,dt + \nabla\cdot S(X_t^{\nu})\,dt
+ \sigma(X_t^{\nu})\,dB_t, \qquad X_0^\nu \sim \nu.
\]
Finally~\eqref{eq:Ipoisson} takes the form
\begin{equation}
  \label{eq:Icorr}
  \IA(\nu) = \frac{1}{4} \int_0^{+\infty} \E_{\nu}\Big[ \big(\LA v\big)(X_0^\nu)
    \big(\LA v\big)(X_t^\nu)\Big]dt.
\end{equation}
The antisymmetric part of the entropy is therefore the autocorrelation of~$\LA v$
along a reversible process that realizes the fluctuation corresponding to the measure~$\nu$.
From a mathematical point of view, it seems interesting to relate~\eqref{eq:Icorr} to the
so-called level 2.5 of large deviations~\cite{barato2015formal,chetrite2015variational}, since
this approach consists in considering joint fluctuations of the empirical measure and
the associated empirical current. In
this case, the large deviations function is explicit: this reflects the fact that a Markov process
is characterized \emph{entirely} by its density and current. Exploring further the connection
between~\eqref{eq:Icorr} and level 2.5 large deviations is an interesting direction for
future works.

\begin{remark}
  It is also possible to consider the adjoint~$\Lc^*$ not with respect to the invariant
  measure~$\mu$ (whose analytical expression may be unknown), but instead with respect
  to a reference measure~$\mu_{\mathrm{ref}}$ with a known analytical expression such that $\Lc^* = \Lc_1 - \Lc_2 + \xi$ for some
  measurable function~$\xi$ (with $\Lc = \Lc_1+\Lc_2$). This leads to an additionnal term  $-\intX \xi\,d\nu$ in the
  expression of the rate function~\eqref{eq:Idecomp}, as can be
  readily checked by a straightforward adaptation of the proof.
  The operators $\Lc_1$ and $\Lc_2$ are the counterparts
  of the symmetric and antisymmetric parts of the generator in this decomposition.
  A typical situation to apply this strategy is provided by systems subject to a small external nonequilibrium forcing,
  the reference measure usually being chosen as the invariant measure at equilibrium, in the absence of external forcing.
  Atom chains in contact with an inhomogeneous heat bath were studied with this approach
  in~\cite{bodineau2008large}, $\mu_{\mathrm{ref}}$ being the Gibbs measure associated with
  a fixed temperature profile. 
\end{remark}

\section{Applications}
\label{sec:applications}

\subsection{Overdamped Langevin dynamics}
\label{sec:overdamped}
In this section, we come back to the setting of Remark~\ref{rem:cramer} by considering
a diffusion process over $\X=\R^d$ subject to
\begin{equation}
  \label{eq:overdamped}
dX_t =b(X_t)\,dt + \sqrt{2}\, dB_t,
\end{equation}
where $b:\R^d\to\R^d$ is a smooth function and~$(B_t)_{t\geq 0}$ is a $d$-dimensional Brownian
motion. This corresponds to~\eqref{eq:SDE} with $\sigma=\sqrt{2}$, in which case the generator reads
\[
\Lc = b \cdot \nabla + \Delta.
\]
We will treat the reversible case where $b = - \nabla V$ for a smooth potential~$V$, and 
$b = - \nabla V + F$ for a smooth function~$F$ such that $\nabla\cdot( F\e^{-V})=0$.
In both cases, the invariant probability measure~$\mu$ of the process is (assuming
$\e^{-V} \in L^1(\X)$)
\begin{equation}
  \label{eq:gibbsoverdamped}
\mu(dx) = Z^{-1}\e^{-V(x)}dx,\quad Z=\intX \e^{-V}<+\infty.
\end{equation}
The dynamics~\eqref{eq:overdamped} is reversible (\textit{i.e.} $\Lc^*=\Lc$,
where~$\Lc^*$ denotes the adjoint of~$\Lc$ in~$L^2(\mu)$) if and only if $b=-\nabla V$.
We now give a standard condition on~$V$ under which the framework developped in
Sections~\ref{sec:LDP} and~\ref{sec:Donsker} applies.

\begin{assumption}
  \label{as:overdamped}
  The potential $V\in\mathscr{S}$ has compact level sets, satisfies
  $\e^{-V} \in L^1(\X)$ and, for any $\theta\in(0,1)$, it holds
  \begin{equation}
    \label{eq:wittenoverdamped}
  (1 - \theta)| \nabla V|^2 - \Delta V \xrightarrow[|x|\to +\infty]{} + \infty.
  \end{equation}
\end{assumption}
This assumption is satisfied for smooth potentials growing like~$|x|^{q}$ for
$q >1$ at infinity, and it also implies that the invariant probability measure~$\mu$
satisfies a Poincaré inequality~\cite{bakry2008simple}. Similar conditions are
derived in~\cite{kontoyiannis2005large} in the context of large deviations. The next
proposition is a direct application of Propositions~\ref{prop:nonlin} and~\ref{prop:ergodicity},
Theorem~\ref{theo:LDP} and Corollary~\ref{cor:Iantisym}.

\begin{proposition}
  \label{prop:overdamped}
  Under Assumption~\ref{as:overdamped}, the process~\eqref{eq:overdamped} with~$b=-\nabla V$
  admits the function
  \[
  W(x) = \e^{\theta V(x)}
  \]
  for any $\theta \in(0,1)$ as a Lyapunov function in the sense of Assumption~\ref{as:lyapunov}.
  For any fixed $\theta \in(0,1)$, there exist $C,c>0$ such that for any initial
  measure $\nu\in\mathcal{P}_W(\X)$,
  \[
  \mathrm{d}_W(\nu P_t,\mu)\leq C \,\e^{-c t} \mathrm{d}_W(\nu,\mu).
  \]
  Moreover, 
  \begin{equation}
    \label{eq:Psovd}
  \Ps = -\frac{\Lc W}{W} = \theta\big(  (1 - \theta)| \nabla V|^2 - \Delta V\big)
  \end{equation}
  has compact level sets and, for any $\kap:\X\to[1,+\infty)$ belonging to~$\mathscr{S}$,
    bounded  or with compact level sets and such that
  \[
  \frac{\Ps(x)}{\kap(x)}  \xrightarrow[|x|\to +\infty]{} + \infty,
  \]
  the empirical measure
  \[
  L_t:= \frac{1}{t}\int_0^t \delta_{X_t}\,ds
  \]
  satisfies a large deviations principle in the $\tau^\kap$-topology. The good rate function
  is defined by: for all~$\nu\in\mathcal{P}_\kap(\X)$ with~$d\nu = \rho\,d\mu = \e^v\,d\mu$,
  \begin{equation}
    \label{eq:Ioverdamped}
       I(\nu) = \frac{1}{4} \intX |\nabla v|^2\,d\nu =
   \frac{1}{4}\intX \frac{|\nabla \rho|^2}{\rho}\, d\mu, 
  \end{equation}
  and $I(\nu)=+\infty$ otherwise.
\end{proposition}

In this reversible example, we see that the rate function is only defined through its symmetric
part~\eqref{eq:Isym}, as shown in Theorem~\ref{theo:InormH}. We now consider a modification of
this dynamics when a divergence-free drift is added. The next proposition is an extension of
the examples proposed in~\cite{rey2015irreversible} to the unbounded state space case.

\begin{proposition}
  \label{prop:overdampedirr}
  Suppose that Assumption~\ref{as:overdamped} holds and consider the diffusion process solution to:
  \[
  dX_t = \big(- \nabla V(X_t) + F(X_t)\big)dt + \sqrt{2}\,dB_t,
  \]
  with~$F$ a smooth vector field such that $\nabla\cdot(F\e^{-V})= 0$ and
  \begin{equation}
    \label{eq:orthoF}
  \frac{F\cdot \nabla V}{\Ps} \xrightarrow[|x|\to +\infty]{} 0,
  \end{equation}
  where~$\Ps$ is defined in~\eqref{eq:Psovd}. Then, with the notation of Section~\ref{sec:irrever}
  it holds  $\LS = -\nabla V\cdot \nabla + \Delta$ and $\LA = F\cdot \nabla$.
  Moreover
  \begin{equation}
    \label{eq:wittennoneq}
  \Ps_F:= -\frac{(\Lc+F\cdot \nabla) W}{W} =
  \theta \big((1 - \theta)| \nabla V|^2 - \Delta V - F\cdot \nabla V\big) \sim \Ps,
  \end{equation}
  and~$(X_t)_{t\geq 0}$ satisfies a LDP in the $\tau^\kap$-topology for any
  function~$\kap$ belonging to~$\mathscr{S}$, bounded or with compact level sets and such that
  \[
  \frac{\Ps(x)}{\kap(x)}
  \xrightarrow[|x|\to +\infty]{} + \infty.
  \]
  The associated rate function~$\IF$ reads: for any~$\nu$ such that $d\nu=\e^{v}\,d\mu$
  with~$v\in\H^1(\nu)$ and $F\cdot \nabla v \in \H^{-1}(\nu)$,
  \[
  \IF(\nu) = \frac{1}{4}\intX |\nabla v|^2\, d\nu +
  \frac{1}{4} \intX |\nabla \psi_v|^2 d\nu,
  \]
  where~$\psi_v$ is the unique $\H^1(\nu)$-solution to
  \[
  -\Delta \psi_v + \nabla (V-v) \cdot \nabla \psi =  F\cdot \nabla v.
  \]
\end{proposition}

Proposition~\ref{prop:overdampedirr} shows that, in this simple case, the equilibrium and
nonequilibrium dynamics admit a LDP for the same class of functions but with different rate
functions, the irreversible dynamics producing more entropy.
It is therefore an extension of the case treated in~\cite[Theorem~2.2]{rey2015irreversible}.
As for this result, Proposition~\ref{prop:overdampedirr} can be used to design algorithms with
accelerated convergence to equilibrium, see
also~\cite{hwang2005accelerating,hwang2015variance,duncan2016variance}.
A setting in which Proposition~\ref{prop:overdampedirr} typically applies is when
$V(x)$ behaves as $|x|^q$ for some $q>1$ outside an open set centered on the origin,
and $F= A\nabla V$ with $A\in\R^{d\times d}$ such that $A^T = - A$
(see~\cite{rey2015irreversible}). The latter condition implies in particular
that $F\cdot\nabla V = 0$  so~\eqref{eq:orthoF} immediately holds.


\subsection{Underdamped Langevin dynamics}
\label{sec:langevin}
We now apply our framework to the underdamped Langevin dynamics. A first nice feature
of our results is that, compared to~\cite{wu2001large}, we obtain a stronger result with 
similar assumptions -- that is our LDP for the empirical measure holds for a finer topology than
the one associated with bounded measurable functions. Note however
that~\cite[Corollary~2.3]{wu2001large} obtains 
results similar to ours for a contraction of the rate function. In addition, 
Theorem~\ref{theo:InormH} and Corollary~\ref{cor:Iantisym} allow to obtain precise results on
the dependency of the rate function on the friction parameter~$\gamma$.

We start by describing the Langevin equation in Section~\ref{sec:langevinintro}, before
stating the large deviations principle in Section~\ref{sec:langevinLDP}. Finally
Section~\ref{sec:langevinasymp} provides asymptotics on the rate function depending on
the friction.

\subsubsection{Description of the dynamics}
\label{sec:langevinintro}
The dynamics is set on $\X=\R^d \times \R^d$, with
$(X_t)_{t\geq 0}=(q_t,p_t)_{t\geq 0}\in\R^d \times \R^d$ evolving as
\begin{equation}
  \label{eq:langevin}
  \left\{
\begin{aligned}
  d q_t & = p_t \,dt, \\
  d p_t & = -\nabla V(q_t)\, dt - \gamma p_t\, dt + \sqrt{2\gamma}\,dB_t,
\end{aligned}
\right.
\end{equation}
where $\gamma >0$ is a friction parameter, $V:\R^d \to \R$ is a smooth potential,
and~$(B_t)_{t \geq 0}$ is a $d$-dimensional Brownian motion. We could also consider
the easier case where the position space is bounded ($q\in\mathbb{T}^d$) but leave
this simple modification to the reader. The generator of the dynamics is
\begin{equation}
  \label{eq:Lgamma}
\Lc_{\gamma} = \Lham + \gamma \LFD,
\end{equation}
where
\[
\Lham= p\cdot \nabla_q - \nabla V \cdot \nabla_p,\quad
\LFD = -p\cdot\nabla_p + \Delta_p.
\]
The operator~$\Lc_\gamma$ leaves invariant the measure
\begin{equation}
  \label{eq:mulangevin}
  \muref(dx)=\mu(dq\,dp)=\bar{\mu}(dq)\omega(dp),\quad
  \bar{\mu}(dq)=Z_q^{-1}\e^{ - V(q)} dq, \quad
  \omega(dp) = (2\pi)^{-d/2} \e^{- \frac{p^2}{2}} dp.
\end{equation}
The invariant measure~\eqref{eq:mulangevin} can be written
\begin{equation}
  \label{eq:mulangevinH}
\mu(dq\,dp) = Z^{-1} \e^{-H(q,p)}\,dq\,dp,
\end{equation}
where
\begin{equation}
  \label{eq:hamiltonian}
  H(q,p) = V(q) +\frac{p^2}{2}
\end{equation}
is the Hamiltonian of the system, and we assume that the normalization constant~$Z$
in~\eqref{eq:mulangevinH} is finite (which is indeed the case when $\e^{-V}\in L^1(\mu)$).
In~\eqref{eq:Lgamma},
the Liouville operator~$ \Lham$  corresponding to the Hamiltonian part of the dynamics
is antisymmetric in~$L^2(\mu)$. On the other hand, the fluctuation-dissipation part with
generator~$\LFD$ is symmetric in~$L^2(\mu)$, so that~$\LA = \Lham$ and~$\LS = \gamma\LFD$ with the
notation of Section~\ref{sec:irrever}.

Before turning to the LDP associated with the Langevin dynamics~\eqref{eq:langevin}, we give
some intuition on the behaviour of the process as~$\gamma$ varies. First, it is clear that in
the small~$\gamma$ limit,~\eqref{eq:langevin} becomes the Hamiltonian dynamics
\[
  \left\{
\begin{aligned}
  d q_t & = p_t \,dt, \\
  d p_t & = -\nabla V(q_t)\, dt.
\end{aligned}
\right.
\]
To be more precise, we introduce the process $(Q_t^\gamma,P_t^\gamma)=(q_{t/\gamma},p_{t/\gamma})$
where~$(q_t,p_t)_{t\geq0}$ is solution to~\eqref{eq:langevin}. It can then be shown that,
in the limit $\gamma\to 0$, the Hamiltonian $H(Q_t^\gamma,P_t^\gamma)$ converges
to an effective diffusion on a
graph~\cite{freidlin1994random,freidlin1998random,freidlin2004random,hairer2008ballistic}.
In particular the relevant time scale in the underdamped limit is~$\gamma^{-1}t$.

On the other hand, in the limit $\gamma\to +\infty$
and under an appropriate time rescaling, we recover the overdamped dynamics studied in
Section~\ref{sec:overdamped}. To see this, we integrate the second line in~\eqref{eq:langevin}
to obtain
\[
p_t - p_0 = -\int_0^t \nabla V (q_s)\,ds - \gamma (q_t - q_0) + \sqrt{2\gamma}B_t.
\]
By introducing now~$Q_t^\gamma = q_{\gamma t}$ and $P_t^\gamma = p_{\gamma t}$, the
latter equality becomes
\[
Q_t^\gamma - Q_0^\gamma = \frac{P_0^\gamma - P_t^\gamma}{\gamma} - \int_0^t \nabla V(Q_s^\gamma)\,ds
+ \sqrt{2}B_t.
\]
When $\gamma\to +\infty$, we observe that~$Q_t^\gamma$ converges formally towards the solution
of~\eqref{eq:overdamped}, see~\cite[Section~6.5]{pavliotis2014stochastic}. The relevant time
scale in the overdamped limit is therefore~$\gamma t$. These remarks
will be of interest below when studying the rate function associated with the
dynamics~\eqref{eq:langevin}.

\subsubsection{Large deviations}
\label{sec:langevinLDP}
In order to obtain a large deviations principle for~\eqref{eq:langevin},
let us make the following classical assumption on the growth of the
potential~\cite{wu2001large,mattingly2002ergodicity,kopec2015weak,lelievre2016partial}.

\begin{assumption}
  \label{as:langevin}
  The potential $V\in\mathscr{S}$ has compact level sets, satisfies $\e^{-V}\in L^1(\X)$
  and there exist $c_V>0$, $C_V\in\R$ such that
  \[
  q\cdot\nabla V(q) \geq c_V |q|^2 - C_V.
  \]
\end{assumption}

We can now find a Lyapunov function for~\eqref{eq:langevin} by following
\textit{e.g.}~\cite{wu2001large,talay2002stochastic,mattingly2002ergodicity}, as made precise in
Appendix~\ref{sec:lyapunovlangevin}. Recall that the Hamiltonian~$H$ is defined
in~\eqref{eq:hamiltonian}.

\begin{lemma}
  \label{lem:lyapunovlangevin}
  Suppose that $(X_t)_{t\geq 0}=(q_t,p_t)_{t\geq 0}$ solves~\eqref{eq:langevin} where~$V$ satisfies
  Assumption~\ref{as:langevin}. Then for any $\gamma>0$ and $\theta\in(0,1)$, there exists
  $\varepsilon >0$ such that
  \begin{equation}
    \label{eq:lyapunovlangevin}
    W(q,p)= \e^{\theta H(q,p) + \varepsilon q\cdot p }
  \end{equation}
  is a Lyapunov function in the sense of Assumption~\ref{as:lyapunov}. More precisely,
  for any~$\gamma>0$ and $\theta\in(0,1)$, there exist~$\varepsilon>0$ and
  $a,b,C>0$ such that
  \[
  -\frac{\Lc_\gamma W}{W}\geq a |q|^2 + b |p|^2 - C.
  \]
\end{lemma}

The Lyapunov function~\eqref{eq:lyapunovlangevin} can be adapted in cases where~$V$
has singularities, see~\cite{herzog2017ergodicity,lu2019geometric}.
We can now deduce our main theorem on the Langevin dynamics since Assumptions~\ref{as:regularity}
and~\ref{as:control} are readily satisfied, see for instance~\cite{mattingly2002ergodicity}.

\begin{theorem}
  \label{theo:LDPlangevin}
  Assume that $(X_t)_{t\geq 0}=(q_t,p_t)_{t\geq 0}$ solves~\eqref{eq:langevin} where~$V$ satisfies
  Assumption~\ref{as:langevin}, and consider a smooth function~$\kappa$ with $\kap(q,p) = 1 +|q|^\alpha + |p|^\beta$ for $|q|+|p|\geq 1$ and $\alpha \in [0,2)$, $\beta\in[0,2)$.
  Then~$(X_t)_{t\geq 0}$ is ergodic with respect to the measure~$\mu$
  in the sense of Proposition~\ref{prop:ergodicity}, with Lyapunov
  function defined in~\eqref{eq:lyapunovlangevin}. Moreover, the empirical
  measure
  \[
  L_t := \frac{1}{t}\int_0^t \delta_{(q_s,p_s)}\,ds
  \]
  satisfies a LDP in the $\tau^\kap$-topology. Finally, for any $\nu\in\mathcal{P}_\kap(\X)$
  such that $d\nu=\e^v\,d\mu$ with $v\in\H^1(\nu)$ and $\Lham v\in \H^{-1}(\nu)$,
  the rate function reads
  \begin{equation}
    \label{eq:Ilangevin}
    I_{\gamma}(\nu) =  \frac{\gamma}{4} \intX  |\nabla_p v|^2\, d\nu
  + \frac{1}{4\gamma}\intX |\nabla_p \psi|^2 \,d\nu,
  \end{equation}
  where $\psi$ is the unique solution in~$\H^1(\nu)$ to the Poisson problem:
  \begin{equation}
    \label{eq:poissonlangevin}
   - \Delta_p \psi + (p - \nabla_p v)\cdot\nabla_p\psi = \Lham v.
  \end{equation}
\end{theorem}

The proof of Theorem~\ref{theo:LDPlangevin} is a direct application of the results
of Sections~\ref{sec:LDP} and~\ref{sec:Donsker}. For the expression of the rate function,
we use~\eqref{eq:psiv} and~\eqref{eq:Lgamma} together with the fact that in this case, the
matrix~$S$ defined in Section~\ref{sec:setting} reads
\[
S = \gamma
\begin{pmatrix}
  0 & 0
  \\
  0 & \mathrm{I}_{d\times d}
\end{pmatrix}\in\R^{2d\times 2 d}.
\]
While~$\kap$ can be chosen independently of the friction~$\gamma$,
it is interesting to note the dependency of the rate function~\eqref{eq:Ilangevin}
with respect to this parameter. We discuss more precisely the scaling of the rate function
with respect to~$\gamma$ in the next section, depending on the form of~$\nu$.

\subsubsection{Low and large friction asymptotics of the rate function}
\label{sec:langevinasymp}
The next corollary shows how the decomposition~\eqref{eq:Ilangevin} allows to
identify the most likely fluctuations in the overdamped and underdamped limits. By this we
mean that, when $\gamma \to 0$ or $\gamma \to +\infty$, most fluctuations become
exponentially rare in~$\gamma$ or~$1/\gamma$, but some of them are associated with rate functions that vanish
as $\gamma\to 0$ and $\gamma\to +\infty$. The expression of these typical fluctuations is
motivated by the discussion on the overdamped and underdamped limits in
Section~\ref{sec:langevinintro}, from which the scalings of the rate function appear natural.
Recall the definition of the marginal in position~$\bar{\mu}$ in~\eqref{eq:mulangevin}.

\begin{corollary}
  \label{cor:position}
  Suppose that the assumptions of Theorem~\ref{theo:LDPlangevin} hold true.
\begin{itemize}
  \item \emph{Overdamped limit~$\gamma \to +\infty$:} Consider a measure
    $\nu\in\mathcal{P}_\kap(\X)$ with $d\nu=\e^v\,d\mu$ equilibrated in the velocity variable,
    \textit{i.e.} such that
  $v(q,p)=v(q)$ with $v\in\H^1(\nu)$ and $p\cdot \nabla_q v \in \H^{-1}(\nu)$. Then, for
  any $\gamma >0$,
  \begin{equation}
    \label{eq:Igamma}
    I_{\gamma}(\nu) =  \frac{1}{4\gamma}\int_{\R^d} |\nabla v(q)|^2\,\bar{\nu}(dq),
  \end{equation}
  where $\bar{\nu} = \e^v \bar{\mu}$.

\item \emph{Hamiltonian limit~$\gamma \to 0$:} Consider a Hamiltonian
  fluctuation, \textit{i.e.} $d\nu=\e^v\,d\mu$ with
  $v(q,p)= g(H(q,p))\in\H^1(\nu)$ for $g\in C^1(\R)$, where~$H$ is defined
  in~\eqref{eq:hamiltonian}. Then, for any $\gamma >0$,
  \begin{equation}
    \label{eq:Igammaham}
    I_{\gamma}(\nu) =\frac{\gamma}{4} \intX \big|p g'\big(H(q,p)\big)\big|^2 \nu(dq\,dp).
  \end{equation}
  \end{itemize}
\end{corollary}

The proof is an immediate consequence of~\eqref{eq:Ilangevin}.
\begin{proof}
  Consider first the case where $d\nu = \e^v \,d\mu$ with $v(q,p)=v(q)$. We have
  \[
  \frac{\gamma}{4} \intX  |\nabla_p v|^2\, d\nu = 0.
  \]
  Next,~\eqref{eq:poissonlangevin} becomes
  \[
   -\big(\Delta_p - p\cdot\nabla_p\big)\psi(q,p) =  p\cdot\nabla_q v(q).  
  \]
  The solution to this equation is~$\psi(q,p)= -p\cdot \nabla_q v(q)$ which indeed belongs
  to~$\H^1(\nu)$ since $\Lham v\in \H^{-1}(\nu)$ (in fact we may add to~$\psi$ any function
  depending on~$q$ only but the solutions would be equivalent by definition of the
  space~$\H^1(\nu)$ in Section~\ref{sec:setting}). Plugging this solution
  into~\eqref{eq:Ilangevin} leads to~\eqref{eq:Igamma}.

  Assume now that $v(q,p)= g(H(q,p))$ belongs to~$\H^1(\nu)$ with $g\in C^1(\R)$. It holds
  \[
  \Lham v(q,p) = g'(H(q,p))\Lham H(q,p) = 0.
  \]
  As a result, the solution~$\psi$ to~\eqref{eq:poissonlangevin} is $\psi = 0$
  (again, up to a function of~$q$ only), from which~\eqref{eq:Igammaham} follows
  since~$v\in\H^1(\nu)$.
\end{proof}

Corollary~\ref{cor:position} characterizes the dominant fluctuations in the small and
large friction regimes. In the overdamped limit~$\gamma\to +\infty$ the dominant fluctuations
are in position only, and the rate function is actually that of the limiting overdamped
dynamics~\eqref{eq:Ioverdamped} up to a time rescaling in~$t\mapsto \gamma t$, which is coherent with
the discussion on the overdamped limit in Section~\ref{sec:langevinintro}.
On the other hand, in the Hamiltonian limit~$\gamma\to 0$, the dominant fluctuations
are Hamiltonian, with the inverse time rescaling $t\mapsto \gamma^{-1}t$. This is consistent with
the small temperature limit of Hamiltonian systems~\cite{freidlin1998random}.

Although Corollary~\ref{cor:position} provides interesting information, its structure is quite
rigid. For instance, in the overdamped limit, we consider only position-dependent perturbations,
which is not realistic. We now refine the asymptotics by considering the next order
correction in~$\gamma$ for the perturbation in both regimes, which shows the robustness of the
analysis. In the result stated below, we consider a family of probability measures~$\nu_\gamma$ indexed by~$\gamma>0$,
and simply denote by~$\nu$ the probability measure~$\nu_0$.

\begin{corollary}
  \label{cor:position1}
  Suppose that the assumptions of Theorem~\ref{theo:LDPlangevin} hold true.
  \begin{itemize}
  \item \emph{Overdamped limit~$\gamma \to +\infty$:}    
  Consider the measure
  $\nu_{\gamma}\in\mathcal{P}_\kap(\X)$ defined by~$\nu_\gamma = \e^{v_\gamma}d\mu$ with
  $v_\gamma(q,p)=v(q) + \gamma^{-1} \tilde{v}(q,p)$ where $\Lham v\in \H^{-1}(\nu)$,
  and $\tilde{v}\in \H^1(\nu)$ is bounded and satisfies
  $\nabla_q v \cdot \nabla_p \tilde{v}\in\H^{-1}(\nu)$
  and $\Lham \tilde{v} \in\H^{-1}(\nu)$. Then
  \begin{equation}
    \label{eq:Isymlangevin}
\forall\,\gamma\geq 1, \quad  I_\gamma(\nu_\gamma) = \frac{1}{4\gamma}
\left[ \intX |\nabla_p \tilde{v}|^2\,d\nu
  + \int_{\R^d} |\nabla_q v|^2\,d\bar{\nu}\right]  +\mathrm{O}\left( \frac{1}{\gamma^2} \right),
  \end{equation}
  where $\bar{\nu}=\e^v \bar{\mu}$.

\item \emph{Hamiltonian limit~$\gamma \to 0$:}  
  Consider~$\nu_\gamma= \e^{v_\gamma}d\mu$
  with~$v_\gamma(q,p) = g(H(q,p)) + \gamma \tilde{v}(q,p)$,
  where~$g\in C^1(\R)$, $g(H)\in\H^1(\nu)$, and $\tilde{v}\in \H^1(\nu)$ is bounded and
  satisfies $\Lham \tilde{v}\in \H^{-1}(\nu)$. Then
  \begin{equation}
    \label{eq:Iantisymlangevin}
    \forall\,\gamma\leq 1,\quad
    I_\gamma(\nu_\gamma) = \frac{\gamma}{4}\left[\intX \big|pg'\big(H(q,p)\big)\big|^2 \nu(dq\,dp)
  +\intX |\nabla_p \tilde{\psi}|^2\,d\nu\right] +\mathrm{O}\big(\gamma^2\big),
  \end{equation}
  where~$\tilde{\psi}$ is the unique solution in~$\H^1(\nu)$ to
  \begin{equation}
  \label{eq:psiexpandlangevin}
  - \Delta_p\tilde{\psi}  -  \big(1- g'(H(q,p))\big)p\cdot \nabla_p\tilde{\psi} = \Lham \tilde{v}.
  \end{equation}
\end{itemize}
\end{corollary}

We believe that it is also instructive to mention the relation between the rate
function~\eqref{eq:Ilangevin} and the asymptotic variance of the Langevin dynamics.
Indeed, when considering small perturbations of the invariant measure,
Corollary~\ref{cor:position1} shows that
\begin{equation}
  \label{eq:Igammascale}
I_\gamma \sim \mathrm{min}\left( \gamma,\frac{1}{\gamma}\right).
\end{equation}
On the other hand, the resolvent estimates
in~\cite[Section~2.1]{leimkuhler2016computation}
and~\cite{GS16,hairer2008ballistic,iacobucci2019convergence}
show that the asymptotic variance~$\sigma_\gamma^2$ scales like
\begin{equation}
  \label{eq:sigmagammascale}
\sigma_\gamma^2 \sim \mathrm{max}\left( \gamma,\frac{1}{\gamma}\right).
\end{equation}
Since we expect the asymptotic variance to be the inverse of the rate function around
the invariant measure~\cite{den2008large,rey2015irreversible},
the scalings~\eqref{eq:Igammascale} and~\eqref{eq:sigmagammascale} are consistent.
However, as~\eqref{eq:Ilangevin} suggests, this scaling is no longer true for general
fluctuations. We now present the proof of Corollary~\ref{cor:position1}.
\begin{proof}
  We first consider the overdamped limit $\gamma\to +\infty$. Since~$\tilde{v}$
  is bounded we have, for any~$\gamma\geq 1$ and~$\psi\in\H^1(\nu_\gamma)$,
  \begin{equation}
    \label{eq:equivnorms}
    \e^{\frac{\inf \tilde{v}}{\gamma}}|\psi|_{\H^1(\nu)}^2
  \leq    |\psi|_{\H^1(\nu_\gamma)}^2\leq  \e^{\frac{\sup \tilde{v}}{\gamma}}|\psi|_{\H^1(\nu)}^2.
  \end{equation}
  Thus, the norms~$\H^1(\nu_\gamma)$ and~$\H^1(\nu)$ are equivalent for any fixed~$\gamma\geq 1$,
  and the functions of~$\H^1(\nu_\gamma)$ and~$\H^1(\nu)$ coincide
  (we repeatedly use this fact below, and we will use a similar argument
  when $\gamma \leq 1$). A similar conclusion holds for the corresponding
  dual norms. This consequence of the boundedness of~$\tilde{v}$ makes the analysis simpler.

  Recall that we consider~$v_\gamma(q,p) = v(q) + \gamma^{-1}\tilde{v}(q,p)$ in the overdamped
  limit. The symmetric part of the rate function is easily computed since~$v$ only depends on the
  position variable, namely
  \[
  \IS(\nu_\gamma)= \frac{\gamma}{4}\intX\big|\nabla_p (v + \gamma^{-1}\tilde{v})\big|^2\,
  \e^{v+\frac{\tilde{v}}{\gamma}}
  d\mu = \frac{1}{4\gamma} \intX |\nabla_p\tilde{v}|^2\,d\nu
  + \mathrm{O}\left(\frac{1}{\gamma^2}\right),
  \]
  where we used that~$\tilde{v}$ belongs to~$\H^1(\nu)$ and is bounded to expand the exponential.
  For the antisymmetric part, by~\eqref{eq:poissonlangevin}, we have to consider the
  solution~$\psig\in\H^1(\nu_\gamma)$ to
  \[
  -\Delta_p\psig + \left(p-\frac{1}{\gamma}\nabla_p \tilde{v}\right)\cdot \nabla_p\psig
  =  \Lham v_\gamma.
  \]
  Corollary~\ref{cor:position} suggests that at leading order in~$\gamma$ it holds
  $\psig = \psi + \mathrm{O}(\gamma^{-1})$ where $\psi(q,p) = p\cdot \nabla v(q)$.
  In order to make this idea more precise we compute
  \[
  \left(-\Delta_p + \left(p-\frac{1}{\gamma}\nabla_p \tilde{v}\right)\cdot
  \nabla_p\right)(\psig - \psi)
  = \frac{1}{\gamma}\big(\Lham \tilde{v} + \nabla_q v\cdot \nabla_p \tilde{v}\big).
  \]
  In what follows, we denote by $u=\Lham \tilde{v} + \nabla_q v\cdot \nabla_p \tilde{v}$ the right
  hand side of the above equation.
  Since~$\nabla_q v\cdot \nabla_p \tilde{v}\in\H^{-1}(\nu_\gamma)$
  and $\Lham \tilde{v}\in \H^{-1}(\nu_\gamma)$ by assumption, it holds~$u\in\H^{-1}(\nu_\gamma)$.
  Thus, multiplying by $\psig - \psi$ and integrating with respect to~$\nu_\gamma$ we obtain
  \[
  \intX \big|\nabla_p(\psig - \psi)\big|^2\,d\nu_\gamma = -\frac{1}{\gamma}
  \intX (\psig - \psi)u\,d\nu_\gamma.
  \]
  Using the duality between~$\H^1(\nu_\gamma)$ and~$\H^{-1}(\nu_\gamma)$
  (see~\cite[Section~2.2 Claim~F]{komorowski2012fluctuations}) and~\eqref{eq:equivnorms} we find
  \[
  \forall\,\gamma \geq 1,\quad
  |\psig - \psi|_{\H^1(\nu)}\leq \frac{C}{\gamma}
  | u |_{\H^{-1}(\nu)},
  \]
  where~$C$ is some constant independent of~$\gamma$.
  This shows that~$\psig = \psi + \gamma^{-1}\tilde{\psi}_\gamma$ with
  $|\tilde{\psi}_\gamma|_{\H^1(\nu)}\leq C'$ for a constant~$C'>0$
  and all~$\gamma\geq 1$. Plugging
  this estimate into~\eqref{eq:Ilangevin} and using that~$\nabla_p \psi = \nabla_q v$, we
  obtain the second term on the right hand side of~\eqref{eq:Isymlangevin}.

  The arguments to prove the limit~$\gamma\to 0$ follow a similar path, so we only sketch
  the proof. First, the boundedness of~$\tilde{v}$ allows again to compare the Sobolev norms
  associated with~$\nu$ and~$\nu_\gamma$ for any $\gamma\leq 1$ (by writting the
  counterpart of~\eqref{eq:equivnorms} in this regime).
  The first term on the right hand side of~\eqref{eq:Iantisymlangevin} is easily
  obtained as in Corollary~\ref{cor:position} using that~$g(H)\in \H^1(\nu)$ and~$\tilde{v}$
  is bounded. Concerning the antisymmetric part,~\eqref{eq:poissonlangevin} now reads
  \[
  \big(-\Delta_p + ( p - \nabla_p v_{\gamma})\cdot \nabla_p\big) \psig = \gamma \Lham \tilde{v},
  \]
  since~$\Lham g(H(q,p))=0$.
  Because of the scaling in~$\gamma$ on the right hand side of the above equation, 
  the solution~$\psig$ can be expanded as~$\psig = \gamma \tilde{\psi} + \mathrm{O}(\gamma^2)$
  in~$\H^1(\nu)$, where~$\tilde{\psi}$ is solution to
  \[
  -\Delta_p \tilde{\psi} + \big( 1 - g'(H(q,p))\big)p\cdot \nabla_p \tilde{\psi} = \Lham \tilde{v}.
  \]
  This reasoning can be made rigorous by a precise asymptotic analysis as above.
  Plugging this expansion into~\eqref{eq:Ilangevin}
  provides the second term on the right hand side of~\eqref{eq:Iantisymlangevin}.
\end{proof}

\section{Conclusion and perspectives}
\label{sec:discussion}

The goal of this paper was twofold. Our first aim was to provide, given a diffusion process,
a precise class of unbounded functions for which a large deviations principle holds.
This question is answered in Section~\ref{sec:LDP} were we prove a LDP for the empirical
measure in a topology associated with unbounded functions, in relation with
a Witten--Lyapunov condition. In particular, a comparison with Cram\'{e}r's condition for
independent variables shows the effect of correlations on the stability of the SDE at hand.
These results extend in several
directions and refine results from previous works~\cite{wu2001large,kontoyiannis2005large}.
However, the necessity of our Lyapunov condition for a LDP to hold
is still an open problem -- whereas the necessity of a similar condition
is known for the Sanov theorem~\cite{wang2010sanov}. Our second concern was to
provide finer expressions of the rate function governing the LDP, in particular in order to
study Langevin dynamics which appear for instance in molecular simulation. We answer to this
question in two ways in Section~\ref{sec:Donsker}.
We first provide an alternative variational formula for the rate function in
Section~\ref{sec:DV}, which gives as
a by-product a very general representation formula for the principal eigenvalue of second order
differential operators, without symmetry assumption. This extends the important work
of Donsker and Varadhan~\cite{donsker1975variational} in an unbounded setting. In
Section~\ref{sec:irrever}, we show a general decomposition of the rate function into symmetric
and antisymmetric parts of the dynamics based on the computations in~\cite{bodineau2008large}.
Interestingly, the proof of the result relies on a Witten-like transform in the above mentioned
variational representation of the rate function.
These results allow us to describe precisely the rate function
of an irreversible overdamped Langevin dynamics in Section~\ref{sec:overdamped}, revisiting
results from~\cite{rey2015irreversible} in an unbounded setting. More interestingly
we provide in Section~\ref{sec:langevin}, for Langevin dynamics, 
asymptotics of the rate function for the overdamped and the underdamped limits. We thus
characterize the most likely fluctuations in both regimes with a natural physical interpretation.
Considering piecewise deterministic
processes~\cite{bierkens2018piecewise,durmus2018geometric,durmus2018piecewise}
(which lack regularity) instead of the Langevin dynamics is also an interesting problem.

We would like to mention several interesting directions for
future works. A first natural issue is to rephrase our results in the optimal control
framework developed \textit{e.g.} in~\cite{boue1998variational,dupuis2011weak,dupuis2018large}.
This is particularly interesting for numerical purposes, since the optimal control representation
can be learned on the fly with stochastic approximation
methods~\cite{borkar2003peformance,basu2008learning,benveniste2012adaptive,ferre2018adaptive}.
We believe that such results can be obtained by harvesting the contraction principle provided
by Corollary~\ref{cor:contract}.

On a more theoretical ground, dual Sobolev norms have recently attracted attention in
the optimal control community due to the so-called optimal matching
problem, see for instance~\cite{ledoux2017optimal,ledouxfluctuation} and references therein.
With these works in mind, the dual Sobolev norm in the antisymmetric part of the rate function
described in Section~\ref{sec:irrever} could be interpreted as an infinitesimal transport cost
related to the antisymmetric part of the dynamics, which is an alluring interpretation of
irreversibility. Note that the relations between optimal transport and large deviations theory
have a fruitful history, see \textit{e.g.}~\cite{gozlan2007large}.

It has been known for some time in the physics literature that the empirical density
of a diffusion may not contain enough information to describe its fluctuations
in an irreversible regime. It is actually more relevant to consider the fluctuations of both
the empirical density and current, a procedure sometimes called level~2.5 large
deviations~\cite{chetrite2015variational,barato2015formal}. This framework
can be used to provide a clear description of the rate function of irreversible dynamics.
As shown in~\cite{barato2015formal}, such large deviations results can be derived by
Krein--Rutman arguments like those used in the present paper. Therefore, we believe that
our results can be extended to prove level~2.5 large deviations principles and characterize
precisely the class of admissible currents.

Finally, it is important to understand the behaviour of observables which are not covered
by our analysis. It has been recently shown~\cite{nickelsen2018anomalous} in the case
of the Ornstein--Uhlenbeck process that observables growing too fast at infinity with
respect to the confinement are characterized by a \emph{heavy tail} behaviour. This leads to
a level~1 large deviations principle at an anomalous speed with a localization
in time of the fluctuation, and the Krein--Rutman
strategy developped in the present paper does not apply. We therefore believe
there are several interesting open questions in this direction.

\subsection*{Acknowledgments}
The authors warmfully thank Hugo Touchette for reading an early version of the manuscript
as well as the first preprint, and providing useful comments; as well as the referees, whose suggestions helped us
making more precise various aspects of this work. The authors are grateful to
Ofer Zeitouni for an interesting discussion about scalings in large deviations theory, as well as
to Jianfeng Lu for pointing out the work~\cite{bodineau2008large}. We also thank
Julien Reygner for general discussions on
large deviations. The PhD of Grégoire Ferré was supported by the Labex Bézout ANR-10-LABX-58-01.
The work of Gabriel Stoltz was funded in part by the Agence Nationale de la Recherche, under
grant ANR-14-CE23-0012 (COSMOS), and by the European Research Council under the European
Union’s Seventh Framework Programme (FP/2007-2013)/ERC Grant Agreement number 614492.
We also benefited from the scientific environment of the Laboratoire International
Associé between the Centre National de la Recherche Scientifique and the University of Illinois at
Urbana-Champaign.
\section{Proofs}
\label{sec:proofs}

In all the proofs below, for conciseness,
we write $\E_x,\mathbb{P}_x$, etc, with some abuse of notation, to indicate that the expectations we consider
are taken with respect to all realizations of the dynamics~\eqref{eq:SDE} started
from~$X_0=x$; and do not indicate explicitly the dependence of~$X_t$ on~$x$, in
contrast to the convention used in Section~\ref{sec:LDP}. 

\subsection{Proof of the large deviations principle}
\label{sec:proofLDP}
As mentioned after Theorem~\ref{theo:LDP}, our proof relies on the Gärtner--Ellis
theorem~\cite{dembo2010large}, for which we need several preliminary results. The key object
is the functional
\[
f\in\Binfty_\kap(\X)\ \mapsto \ \lambda(f):=\underset{t\to + \infty}{\lim}\,
    \frac{1}{t} \log \E_x\left[ \e^{\int_0^t f(X_s)\, ds} \right].
\]
Roughly speaking, the Gärtner--Ellis theorem (Theorem~\ref{theo:GE} in Appendix~\ref{sec:tools})
states that if this functional is finite
and Gateau-differentiable over~$\Binfty_\kap(\X)$ and~$(L_t)_{t\geq 0}$ defined
in~\eqref{eq:Lt} is exponentially tight for the $\tau^\kap$-topology, then~$(L_t)_{t\geq 0}$
satisfies a LDP in the dual space of~$\Binfty_\kap(\X)$. A reminder of this
theorem and some elements of analysis are given in Appendix~\ref{sec:tools}.

However, studying the range of functions~$f$ for which the functional~$\lambda$ is finite and
Gateau-differentiable is not an easy task. Formally, our strategy is to prove that~$r(f)$, the
element of the spectrum of the operator~$\Lc +f$ with the largest modulus, is a real eigenvalue
for any function $f\in\Binfty_\kap(\X)$, and to show that it is actually equal to the cumulant
function~$\lambda(f)$ defined in~\eqref{eq:SCGF}.
This amounts to showing the well-posedness and regularity of a family of spectral problems.
For this, we use several ideas from~\cite{ferre2018more}, which shows that under Lyapunov and
irreducibility conditions, the eigenvalue problem to which~$\lambda$ is associated is well
defined. In order to avoid technical difficulties related to unbounded operators, we study
the semigroup~$(P_t^f)_{t\geq 0}$ rather than its generator~$\Lc +f$, see Remark~\ref{rem:domain}
below for more details.
The seminal paper by Gärtner~\cite[Section~3]{gartner1977large}
provides useful technical tools, as well as~\cite{ellis1984large,wu2001large}.

In all of this section, we suppose that Assumptions~\ref{as:regularity},~\ref{as:control}
and~\ref{as:lyapunov} hold true and consider a function $\kap:\X\to [1,+\infty)$ of
  class~$\mathscr{S}$ as in
  Assumption~\ref{as:lyapunov}, \textit{i.e.} such that $\kap\ll \Ps$ and either~$\kap$
  is bounded or has compact level sets and satisfies~\eqref{eq:Lyapunov_kappaW}. We
repeatedly use that $\kap\ll - \frac{\Lc \mathscr{W}}{\mathscr{W}}$ in view
of~\eqref{eq:restriction}. We start with important properties of key martingales that 
appear regularly in the proofs of the required technical results.

\begin{lemma}
  \label{lem:martingales}
  If $(X_t)_{t\geq 0}$ is a solution to~\eqref{eq:SDE}, then the stochastic processes defined by
  \begin{equation}
    \label{eq:martingales}
    M_t = W(X_t)\, \e^{-\int_0^t \frac{\Lc W}{W}(X_s)\,ds}     \quad \mathrm{and} \quad
    \Mm_t = \mathscr{W}(X_t)\, \e^{-\int_0^t \frac{\Lc \mathscr{W}}{\mathscr{W}}(X_s)\,ds}
  \end{equation}
  are continuous non-negative local martingales, hence supermartingales. Moreover, it
  holds almost surely
  \begin{equation}
    \label{eq:martingalebound}
  \mathscr{M}_t^2 \leq C_1\,\e^{t C_2 } M_t,
  \end{equation}
  where $C_1>0$ and $C_2\in\R$ are the constants from Assumption~\ref{as:lyapunov}.
\end{lemma}

\begin{proof}
  First, Itô formula gives
  \[
  dM_t =   \e^{-\int_0^t \frac{\Lc W}{W}(X_s)\,ds}\nabla W(X_t)\cdot\sigma(X_t) \,dB_t.
  \]
  Since~$W$ is $C^2(\X)$ and~$\sigma$ is continuous,~$M_t$ is a continuous local
  martingale~\cite{karatzas2012brownian}. Since it is non-negative, it is a supermartingale
  by Fatou's lemma, and the same conclusion holds for~$\Mm_t$.
  On the other hand,~\eqref{eq:restriction} shows that 
  \[
  \mathscr{M}_t^2 = \mathscr{W}(X_t)^2\, \e^{\int_0^t - 2 \frac{\Lc \mathscr{W}}{\mathscr{W}}
    (X_s)\,ds}\leq C_1W(X_t)\, \exp\left[\int_0^t \left(- \frac{\Lc W}{W} 
    (X_s)+C_2\right)ds\right]\leq C_1\, \e^{C_2t}M_t,
  \]
  which concludes the proof.
\end{proof}

The use of the martingale~$M_t$ is inspired by~\cite{wu2001large} where it is considered to
control return times to compact sets. Here, it allows to define the Feynman--Kac semigroup
associated with the dynamics~$(X_t)_{t\geq 0}$ with weight function $f\in\Binfty_\kap(\X)$.

\begin{lemma}
  \label{lem:feynmankac}
  Fix $f\in\Binfty_\kap(\X)$. For any $t\geq 0$, the Feynman--Kac operator 
  \begin{equation}
    \label{eq:feynmankac}
    \forall\, \varphi \in\Binfty_W(\X),
    \qquad 
    \big(P_t^f\varphi\big)(x) := \E_x\left[ \varphi(X_t)\,\e^{\int_0^t f(X_s)\,ds}\right],
  \end{equation}
  is well defined. Moreover, $(P_t^f)_{t\geq 0}$ is
  a semigroup of bounded operators on~$\Binfty_W(\X)$.
  Finally, for any $t>0$ and any $a>0$, there exist
  $c_{a,t}\geq 0$ and a compact subset $K_{a,t}\subset\X$ such that
  \begin{equation}
    \label{eq:compactness}
    \forall\, x\in\X,\quad
    \big(P_t^f W\big)(x) \leq \e^{-a t} W(x) + c_{a,t}\ind_{K_{a,t}}(x).
  \end{equation}
\end{lemma}

\begin{proof}
  We first show that for any~$f\in\Binfty_\kap(\X)$,~$(P_t^f)_{t\geq 0}$ is a semigroup of
  bounded operators on~$\Binfty_W(\X)$, before turning to the proof of~\eqref{eq:compactness}.
  For a fixed~$f\in\Binfty_\kap(\X)$, since $\kap \ll \Ps$, there exists~$c>0$ such that, for any $t>0$,
  \[
  \left(P_t^f W\right)(x) = \E_x\left[W(X_t)\,\e^{\int_0^t f(X_s)\,ds}\right]
  \leq \e^{ct}\E_x\left[W(X_t)\,\e^{- \int_0^t\frac{\Lc W}{W} (X_s)\,ds}\right].
  \]
  Using Lemma~\ref{lem:martingales}, the supermartingale property leads to
  \[
    \left( P_t^f W \right)(x) \leq \e^{ct}\E_x\left[M_t\right] \leq \e^{ct}W(x).
  \]
  Therefore, for all $\varphi\in\Binfty_W(\X)$, 
  \[
  \forall\,x\in\X,\quad
  \left|P_t^f \varphi(x)\right|\leq P_t^f |\varphi|(x)\leq \|\varphi\|_{\Binfty_W} \left(P_t^fW\right)(x), 
  \]
  and hence
  \[
  \left\|P_t^f \varphi\right\|_{\Binfty_W}\leq \e^{c t}\|\varphi\|_{\Binfty_W}.
  \]
  As a result~$(P_t^f)_{t\geq 0}$ is a semigroup of bounded operators over~$\Binfty_W(\X)$.

  We next prove~\eqref{eq:compactness} for a
  fixed $f\in\Binfty_\kap(\X)$, which we assume non-zero without loss of generality. Note that
  \[
  \frac{\Lc W}{W} + f \leq - \Ps + \| f \|_{\Binfty_\kap} \kap.
  \]
  Since~$\Ps$ has compact level sets and $\kap\ll \Ps$, for any~$a >0$ there exists a
  compact set $K_{a}\subset \X$ and a constant~$b_{0,a}$ such that
  \[
  \frac{\Lc W}{W} + f \leq -(a+\alpha) + b_{0,a} \ind_{K_a},
  \]
  where $\alpha>0$ is a constant to be chosen later on.
  This implies that
  \[
  (\Lc + f) W \leq - (a+\alpha) W + b_a \ind_{K_a},
  \]
  with $b_a = b_{0,a}\sup_{K_a} W < +\infty$ since $W\in C^2(\X)$. Therefore (by some
  standard approximation arguments relying on stopping times,
  as discussed for instance in~\cite{bellet2006ergodic})
  \begin{equation}
    \label{eq:dtPtf}
    \begin{aligned}
    \frac{d}{dt}\Big( \e^{(a+\alpha) t} P_t^f W\Big) & =\e^{(a+\alpha) t}
    P_t^f\big( (a+\alpha) W +(\Lc+f) W\big)\\ &
  \leq b_a\, \e^{(a+\alpha)t} \, P_t^f\ind_{K_a}\leq b_a\,\e^{(a+\alpha)t} \, P_t^f\ind.
  \end{aligned}
    \end{equation}
  We can now bound the right hand side of the above equation with a technique
  similar to the one used in~\cite[Section~2.3]{ferre2018more}. Indeed, for any $x\in\X$,
  \begin{equation}
    \label{eq:majofwW}
    \big(P_t^f\ind\big)(x) = \E_x\left[ \e^{\int_0^t f(X_s)\,ds}\right]
     \leq \E_x\left[ \e^{ \|f\|_{\Binfty_\kap} \int_0^t \kap(X_s)\,ds}\right].
  \end{equation}
  Since~$\kap \ll -\frac{\Lc \mathscr{W}}{\mathscr{W}}$, there exists 
  a constant~$c\geq 0$ depending on~$f$ such that
  \[
  \kap \leq \frac{1}{\|f\|_{\Binfty_\kap}}\left( -\frac{\Lc \mathscr{W}}{\mathscr{W}}
  + c\right).
  \]
  Plugging this estimate into~\eqref{eq:majofwW} and using that~$\mathscr{W}\geq 1$
  leads to
  \[
  \big(P_t^f\ind\big)(x) \leq \e^{ct}\,
  \E_x\left[\mathscr{W}(X_t)\, \e^{ \int_0^t -\frac{\Lc \mathscr{W}}{\mathscr{W}}(X_s)\,ds}\right]
  =\e^{ct}\,\E_x[\Mm_t]\leq \e^{ct} \mathscr{W}(x),
  \]
  where the last bound is due to Lemma~\ref{lem:martingales}.

  Using this estimate to bound the right hand side of~\eqref{eq:dtPtf}, we end up with
  \[
  \frac{d}{dt}\Big( \e^{(a+\alpha) t} P_t^f W\Big) 
  \leq b_a\, \e^{(a+\alpha+c)t}\,\mathscr{W}.
  \]
  Integrating with respect to time leads to
  \[
  \big( P_t^f W\big) (x) \leq \e^{-(a+\alpha)t}W(x) +\tilde{b}_a\mathscr{W}(x),\qquad
  \tilde{b}_a= \frac{b_a}{a+\alpha+c}\e^{ct}.
  \]
  Since~$\mathscr{W}\ll W$, there exists a compact set~$K_{a,t}\subset\X$
  such that $\tilde{b}_a\mathscr{W} \leq \e^{-(a+\alpha) t} W$ outside~$K_{a,t}$, so that we have
  \[
  \forall\,x\in\X,\quad
  \big( P_t^f W\big)(x) \leq 2\, \e^{-(a+\alpha)t}W(x) +\left(\tilde{b}_a\sup_{K_{a,t}} \mathscr{W}\right)
  \ind_{K_{a,t}}(x).
  \]
  We can now assume that we chose from the begining $\alpha>\log(2)/t$ (recall that~$t$
  is fixed). Setting $c_{a,t} = \tilde{b}_a\sup_{K_{a,t}} \mathscr{W}$, this leads to
  \[
  \forall\,x\in\X,\quad
  \big(P_t^f W\big)(x) \leq  \e^{-at}W(x) + c_{a,t} \ind_{K_{a,t}}(x),
  \]
  which proves~\eqref{eq:compactness}.
\end{proof}

Lemma~\ref{lem:feynmankac} proves crucial to obtain the compactness of the evolution
operator~$P_t^f$, as noted in~\cite{ferre2018more} (a result inspired
by~\cite[Theorem~8.9]{bellet2006ergodic}). Note however that~$(P_t^f)_{t\geq 0}$
is  a priori not a strongly continuous semigroup on~$B^\infty_W(\X)$, see the discussion in~\cite[Proposition~B13]{Wu00}
and Remark~\ref{rem:domain} below for more details.

Another key ingredient is the
regularization property of the evolution. The following bound on the Feynman--Kac semigroup
depending on the weight function~$f$ is one element in this direction.
\begin{lemma}
  \label{lem:diffbound}
  Suppose that Assumptions~\ref{as:regularity},~\ref{as:control}
  and~\ref{as:lyapunov} hold true, and fix $f,g\in\Binfty_\kap(\X)$.
  Then, for any~$t>0$, any $\varphi\in\Binfty_W(\X)$ and any $ x \in \X$, it holds
  \begin{equation}
    \label{eq:diffbound}
    \begin{aligned}
      & \left| \big(P_t^f\varphi\big)(x) - \big(P_t^g\varphi\big)(x) \right| \\
      & \qquad \leq \| \varphi \|_{\Binfty_W}
      \,   \E_x\left[  W(X_t)\left( \int_0^t | f(X_s) - g(X_s)|\, ds\right)
        \e^{ (\|f\|_{\Binfty_\kap} + \|g\|_{\Binfty_\kap}) \int_0^t \kap(X_s)\,ds} \right].
    \end{aligned}
  \end{equation}
\end{lemma}
\begin{proof}
  Using the inequality~$| \e^a - \e^b|\leq | a - b|\, \e^{|a| +|b|}$ for~$a,b\in\R$,
  we have, for~$x\in \X$,
  \[
  \begin{aligned}
    & \left| \big(P_t^f\varphi\big)(x) - \big(P_t^g \varphi\big) (x) \right| \leq \E_x \left[ |\varphi(X_t)|\, \left|
    \,\e^{\int_0^t f(X_s)\, ds} - \e^{\int_0^t g(X_s)\, ds} \right| \right]
  \\ & \qquad
  \leq \| \varphi \|_{\Binfty_W} \E_x \left[ W (X_t) \left| \int_0^t
    f(X_s)\,ds -\int_0^t g(X_s)\, ds \right| \, \e^{ \int_0^t |f(X_s)|\, ds +
      \int_0^t |g(X_s)|\, ds}\right],
  \\ & \qquad \leq \| \varphi \|_{\Binfty_W} \E_x \left[ W (X_t)\left(  \int_0^t |
    f(X_s) - g(X_s)| \, ds \right) \e^{ (\|f\|_{\Binfty_\kap} + \|g\|_{\Binfty_\kap}) \int_0^t
      \kap(X_s)\,ds} \right],
  \end{aligned}
  \]
  which is the desired conclusion.
\end{proof}

We can now use Lemma~\ref{lem:diffbound} to show an important regularization property of
the Feynman--Kac semigroup. 
\begin{lemma}
  \label{lem:regularity}
  For any~$f\in\Binfty_\kap(\X)$, $\varphi\in\Binfty_W(\X)$, any $t>0$ and any compact
  $K\subset\X$, the function~$P_t^f(\varphi\ind_K)$ is continuous.
\end{lemma}

Let us insist on the fact that the statement of Lemma~\ref{lem:regularity} is a consequence of
Hörmander's theorem~\cite[Theorem~4.1]{eckmann2003spectral}
when~$f$ has polynomial growth and is smooth. However, the result is
more difficult to obtain when~$f$ is irregular. Note for instance that we cannot rely on the continuity property
proved in Section~\ref{sec:proofcor} below since the space of smooth functions with compact support is not dense
in~$B^\infty_W(\X)$. The idea
of the proof is to use the local martingales introduced in Lemma~\ref{lem:martingales}
to show that the regularization property of Hörmander's theorem is preserved when~$f$ is irregular
but does not grow too fast. 

\begin{proof}
  We use Assumption~\ref{as:regularity}
  to revisit~\cite[pages~34-35]{gartner1977large} in an unbounded setting and
  with a hypoelliptic flavour. First, we note that for $f\in\Cinfc(\X)$, the result is a direct
  application of Assumption~\ref{as:regularity} combined with Hörmander's theorem, since the
  evolution operator~$P_t^f$ can be shown to be an integral operator with a transition probability which admits a
  density~$p^f(t,x,y)$ belonging to~$C^\infty((0,+\infty)\times \X \times \X)$ (see for instance~\cite{IK74}
  for~$f=0$, which can easily be extended to~$f\in\Cinfc(\X)$ with the hypoelliptic result of~\cite[Theorem~4.1]{eckmann2003spectral}). In particular,~$P_t^f(\varphi\ind_K)$ is continuous.  


  We now use an approximation argument inspired by~\cite[Section~3]{gartner1977large}
  for a generic function $f\in\Binfty_\kap(\X)$. Consider a sequence~$(f_n)_{n\in\N}$ of functions
  belonging to~$\Cinfc(\X)$ with $\|f_n\|_{\Binfty_\kap}\leq \|f\|_{\Binfty_\kap}$
  for any $n\in\N$, and such
  that~$f_n \to f$ almost everywhere as $n\to +\infty$ (such a sequence exists by Lusin's theorem,
  see~\cite[Chapter~2]{rudin2006real}). By modifying the proof of Lemma~\ref{lem:diffbound}, and since
  $\|f_n\|_{\Binfty_\kap}\leq \|f\|_{\Binfty_\kap}$, we have for any $\varphi\in\Binfty_W(\X)$, $n\in\N$ and $x\in\X$,
  \begin{equation}
    \label{eq:intineqdiff}
    \begin{aligned}
      & \Big| P_t^f\big(\varphi\ind_K\big)(x) - P_t^{f_n}\big(\varphi\ind_K\big)(x) \Big| \\
      & \qquad \leq  \| \varphi \|_{\Binfty_W}
      \,   \E_x\left[\ind_K(X_t)  W(X_t)\left( \int_0^t | f(X_s) - f_n(X_s)|\, ds\right)
        \e^{ \delta \int_0^t \kap(X_s)\,ds} \right],
    \end{aligned}
  \end{equation}
  with $\delta = 2 \|f\|_{\Binfty_\kap}$.

  Our goal is now to show that~$P_t^{f_n}(\varphi\ind_K)$ converges uniformly over any
  compact~$K'$ to~$P_t^f(\varphi\ind_K)$, by proving that the right hand side
  of~\eqref{eq:intineqdiff} goes uniformly to~$0$ over~$K'$. This will conclude the proof since
  a uniform limit of continuous functions is continuous.

  We introduce to this end the events
  \begin{equation}
    \label{eq:event}
    \forall\, m\geq 1,\quad \mathscr{E}_m = \left\{\frac{1}{t} \int_0^t \Ps(X_s)\,ds
    \leq m \right\},
  \end{equation}
  and fix a compact set $K'\subset\X$.
  The right hand side of~\eqref{eq:intineqdiff} can then be split into two terms
  \[
  \begin{aligned}
    (A) = \E_x \left[\ind_K(X_t)\ind_{\mathscr{E}_m^c}  W(X_t)
      \left( \int_0^t | f(X_s) - f_n(X_s)|\, ds\right)
    \e^{ \delta \int_0^t \kap(X_s)\,ds} \right],
    \\
    (B) = \E_x \left[\ind_K(X_t)\ind_{\mathscr{E}_m}  W(X_t)\left(
      \int_0^t | f(X_s) - f_n(X_s)|\, ds\right)
      \e^{ \delta \int_0^t \kap(X_s)\,ds} \right],
  \end{aligned}
  \]
  for which we show convergence to~$0$, uniformly for $x\in K'$, starting with~$(A)$.
  Since~$\kap\ll -\Lc \mathscr{W}/\mathscr{W}$, there exists $c>0$ such that
  \[
  2\delta \kap \leq - \frac{\Lc \mathscr{W}}{\mathscr{W}} + c.
  \]
  Moreover,
  $\|f_n\|_{\Binfty_\kap}\leq \|f\|_{\Binfty_\kap}$,
  $a\leq \e^a$, and $\mathscr{W}\geq 1$, so that
  \[
  (A) \leq \e^{ct} \left( \sup_K W \right) \,\E_x \left[\ind_K(X_t)\ind_{\mathscr{E}_m^c}  \mathscr{W}(X_t)\,
    \e^{ \int_0^t  -\frac{\Lc \mathscr{W}}{\mathscr{W}}(X_s)\,ds} \right].
  \]
  By definition of~$\mathscr{M}_t$ in~\eqref{eq:martingales} we have
  \[
  (A) \leq \e^{ct} \left(\sup_K W\right) \,\E_x \left[\ind_{\mathscr{E}_m^c} \mathscr{M}_t\right].
  \]
  The Cauchy--Schwarz inequality then shows that
  \[
  (A)\leq \e^{ct} \left(\sup_K W\right)  \sqrt{\E_x[\mathscr{M}_t^2] }\left( \proba_x \left( 
   \int_0^t \Ps(X_s)\,ds > m t \right) \right)^{\frac{1}{2}}.
  \]
  By~\eqref{eq:martingalebound} it holds
  $\sqrt{\E_x[\mathscr{M}_t^2] }\leq \sqrt{C_1}\,\e^{C_2 t/2} \sqrt{W(x)}$.
  Next, by Tchebychev's inequality and since $W\geq 1$,
  \[
    \begin{aligned}
      \proba_x \left(\int_0^t \Ps(X_s)\,ds > m t\right) & \leq \e^{-mt} \E_x\left[\e^{\int_0^t
          \Ps (X_s)\,ds}\right] \\
      & \leq \e^{-mt} \E_x\left[W(X_t)\,\e^{-\int_0^t
          \frac{\Lc W}{W}(X_s)\,ds}\right] \leq  \e^{-mt} W(x).
    \end{aligned}
  \]
  As a result, we obtain, for $x\in K'$,
  \[
  (A) \leq  \e^{-\frac{mt}{2}} \left(\sup_K W\right)\left( \sup_{K'}W\right)
  \sqrt{C_1}\,\e^{ct + C_2 t/2}.
  \]
  Therefore, for any $\varepsilon>0$, we can choose $m\geq 0$ such that $(A)\leq \varepsilon$.

  Let us now control~$(B)$, introducing $g_n = |f - f_n|$. Since~$\kap\ll\Ps$, it holds
  for some $c'\geq 0$,
  \[
  \delta \kap \leq \Ps + c'.
  \]
  Using the definition~\eqref{eq:event} we have 
  \[
  \begin{aligned}
    (B) & \leq  \e^{(m+c')t} \left(\sup_K W \right)
    \E_x\left[ \ind_{\mathscr{E}_m} \int_0^t g_n(X_s)\,ds
      \right]
    \\ & \leq  \underbrace{ \e^{(m+c')t} \left(\sup_K W\right)\E_x\left[ \ind_{\mathscr{E}_m}
        \int_0^t g_n(X_s)\ind_{ B_R^c}(X_s)\,ds \right]}_{(B')} 
    + \underbrace{ \e^{(m+c')t}\left( \sup_K W\right)
    \E_x\left[ \ind_{\mathscr{E}_m} \int_0^t g_n(X_s)
      \ind_{B_R}(X_s)\,ds \right]}_{(B'')},
    \end{aligned}
  \]
  where~$B_R$ is the ball of center~$0$ and radius~$R>0$ to be chosen. Let us first bound~$(B')$, which
  retains only the parts of the trajectories performing excursions out of~$B_R$.
  Using $\kap\ll \Ps$, for $\varepsilon>0$ and $m\geq 0$ as fixed above,
  there exist $R>0$, $C_R >0$ such that
  \[
  \kap \leq \varepsilon\frac{\e^{-(m+c')t}}{2t m \left(\sup_K W\right) \| f\|_{\Binfty_\kap}} \Ps
  + C_R \ind_{B_R}.
  \]
  We fix $R>0$ and $C_R>0$ such that the above inequality holds true.
  Using again $g_n\leq 2\|f\|_{\Binfty_\kap}\kap$, we are led to
  \[
  \begin{aligned}
    (B') & \leq 2\e^{(m+c')t}\left( \sup_K W\right) \| f\|_{\Binfty_\kap}
    \E_x\left[ \ind_{\mathscr{E}_m} \int_0^t \kap(X_s)\ind_{B_R^c}(X_s)\,ds\right]
    \\ & \leq 
    \E_x\left[ \ind_{\mathscr{E}_m} \int_0^t \frac{\varepsilon}{tm}
      \Ps(X_s)\ind_{B_R^c}(X_s)\,ds\right]
     \leq  \frac{\varepsilon}{tm}
    \E_x\left[ \ind_{\mathscr{E}_m} \int_0^t \Ps(X_s)\,ds\right]
     \leq \varepsilon,
  \end{aligned}
  \]
  where the last line follows from the definition~\eqref{eq:event} of $\mathscr{E}_m$.
  Therefore, once~$m$ is fixed, there exists $R>0$  such that for any $n\geq 1$ and  $x\in K'$,
  it holds $(B')\leq \varepsilon$. It remains to control~$(B'')$ in order to obtain the uniform
  convergence to zero of~\eqref{eq:intineqdiff} over~$K'$ as $n\to +\infty$. In fact,
  \[
    \begin{aligned}
  (B'') & \leq \e^{(m+c')t}\left( \sup_K W\right) \int_0^t \E_x\left[  g_n(X_s)
    \ind_{ B_R}(X_s) \right]ds \\
  & = \e^{(m+c')t} \left(\sup_K W\right) \int_0^t P_s (g_n
  \ind_{ B_R})(x)\,ds,
  \end{aligned}
  \]
  where~$(P_s)_{s\geq 0}$ is the evolution semigroup defined in~\eqref{eq:Pt}. Since
  $(\ind_{B_R}g_n)_{n\geq1}$ is a sequence of bounded functions converging almost everywhere
  to zero and the transition kernel~$P_s$ has a smooth density for $s>0$, it follows that
  $(P_s (g_n\ind_{ B_R}))_{n\geq 1}$ goes uniformly to zero over compact sets for any $s>0$
  as $n\to+\infty$, see
  \textit{e.g.}~\cite{gartner1977large,bellet2006ergodic}. Moreover, it can be shown that
  \[
  \left| \int_0^{\eta}P_s (g_n\ind_{ B_R})\,ds \right| \leq \eta \left\| g_n \ind_{B_R} \right\|_{B^\infty} \leq 2\eta \left\| f \ind_{B_R} \right\|_{B^\infty},
  \]
  which goes to zero when $\eta\to 0$, uniformly in $x\in K'$ and $n\in\N$. 
  Therefore, for~$\varepsilon >0$, $R>0$ and $m\geq 0$ fixed as above, and choosing
  \[
    \eta = \varepsilon \, \frac{\e^{-(m+c')t}}{2 \left\| f \ind_{B_R} \right\|_{B^\infty} \sup_K W},
  \]
  there exists $n'\in\N$ such that for all $n\geq n'$ and $x\in K'$,
  \begin{equation}
    \label{eq:deltan}
  0\leq  \int_0^t P_s (g_n\ind_{ B_R})(x)\,ds =
  \int_0^{\eta} P_s (g_n\ind_{ B_R})(x)\,ds +
  \int_{\eta}^t P_s (g_n\ind_{ B_R})(x)\,ds \leq  \varepsilon\frac{\e^{-(m+c')t}}{\sup_K W}.
  \end{equation}
  Then, for any $n\geq n'$, $x\in K'$, it holds
  \[
  (B'')\leq \varepsilon.
  \]

  Let us summarize the various approximations: for any $\varepsilon >0$, we first fix
  $m\geq 0$ so that $(A)\leq \varepsilon$. Then, we choose $R>0$ large enough so that
  $(B')\leq \varepsilon$. Finally, we take~$\eta$ small enough and~$n$ large enough
  in~\eqref{eq:deltan} so that $(B'')\leq \varepsilon$ for $n\geq n'$. As a result, for
  any~$\varepsilon>0$
  there is $n'\geq 0$ such that for $n\geq n'$ and $x\in K'$, it holds $(A)+(B)\leq 3\varepsilon$.

  In conclusion, the right hand side of~\eqref{eq:intineqdiff} goes to zero uniformly
  as $n\to+\infty$ over any
  compact set~$K'$. Therefore~$P_t^{f_n}(\varphi\ind_K)$ is continuous
  and converges uniformly over~$K'$ to~$P_t^f(\varphi\ind_K)$, which is therefore continuous
  over~$K'$. Since the compact $K'\subset\X$ is arbitrary,~$P_t^f(\varphi\ind_K)$ is continuous
  over~$\X$, which concludes the proof.
\end{proof}

Before presenting the main result concerning the spectral properties of the operator~$P_t^f$
and its consequences on the definition of the cumulant function~$\lambda(f)$, we need
the following ``irreducibility'' lemma, which relies on Assumption~\ref{as:control}.

\begin{lemma}
  \label{lem:irreducibility}
  For any time $t>0$, $x\in\X$ and any Borel set $A\subset\X$ with non-empty interior, it
  holds that
  \begin{equation}
    \label{eq:irreducibility}
    \left(P_t^f\ind_A\right)(x) >0.
  \end{equation}
\end{lemma}

\begin{proof}
  Take $x\in\X$ and $y\in\mathring{A}$ (which is possible since $A$ has non-empty interior). By
  Assumption~\ref{as:control}, there exists a $C^1$-path $(\phi_s)_{s\in[0,t]}$
  solving~\eqref{eq:path} such that $\phi_0 = x$ and $\phi_t = y$.
  We can then use the proof of the Stroock--Varadhan support theorem,
  see~\cite[Theorem~6.1]{bellet2006ergodic} for an overview. In particular,
  Assumption~\ref{as:control} implies that~\cite[Eq.~(5.5)]{stroock1972support}
  is satisfied. Therefore,~\cite[Eq.~(5.3)]{stroock1972support} ensures that,
  for any $\varepsilon>0$,
  \begin{equation}
    \label{eq:phidef}
    \proba_x\left( \sup_{0\leq s\leq t} | X_s - \phi_s|\leq
  \varepsilon \right) >0.
  \end{equation}
  Moreover, since  $\phi_t=y\in\mathring{A}$ and upon reducing $\varepsilon>0$
  we may assume that $B(y,\varepsilon) \subset A$,
  where $B(y,\varepsilon)$ denotes the ball of center~$y$ and radius $\varepsilon>0$.
  Recalling that $f\in\Binfty_\kap(\X)$, we then obtain
  \begin{equation}
    \label{eq:intmino}
  \begin{aligned}
  \big(P_t^f\ind_A\big)(x) & = \E_x\left[ \ind_{\{X_t\in A\} }\, \e^{\int_0^t f(X_u)\, du}\right]
   \geq \E_x\left[ \ind_{\{ \sup_{0\leq s\leq t } | X_s - \phi_s|\leq \varepsilon\}} \,
    \e^{- \|f\|_{\Binfty_\kap}\int_0^t  \kap(X_u)\, du}\right]
  \\ & \geq \mathrm{exp}\Big(- t \|f\|_{\Binfty_\kap} \sup_{\mathrm{S_{\phi,\varepsilon}}} \kap\Big)
\proba_x\left( \sup_{0\leq s\leq t} | X_s - \phi_s|\leq
  \varepsilon \right),
  \end{aligned}
  \end{equation}
  where we denote by $S_{\phi,\varepsilon}$ the $\varepsilon$-tube around the
  path $(\phi_s)_{s\in[0,t]}$,
  namely
  \[
    S_{\phi,\varepsilon} =\left\{ x\in \X \, \big|\, \exists\, s\in[0,t]\ \mathrm{with}\
      |\phi_s - x|\leq \varepsilon\right\}.
  \]
  Since $S_{\phi,\varepsilon}$ is a bounded set and $\kap$ is continuous over $\X$, it holds
  \[
  \sup_\mathrm{S_{\phi,\varepsilon}} \kap < +\infty.
  \]
  The combination of~\eqref{eq:phidef} and~\eqref{eq:intmino} leads to the desired
  result~\eqref{eq:irreducibility}.
\end{proof}

At this stage, we follow the spectral analysis path developed in~\cite{ferre2018more}.
However, we have to prove that the assumptions used in~\cite{ferre2018more} are fulfilled in
our context. In particular the irreducibility is granted by Lemma~\ref{lem:irreducibility}.

\begin{lemma}
  \label{lem:specLf}
  For any $f\in\Binfty_\kap(\X)$ and any $t>0$, the operator $P_t^f$ considered over~$\Binfty_W(\X)$
  has a real largest eigenvalue~$\mathrm{e}^{t r(f)}$ with eigenspace of dimension one, and an associated
  continuous eigenvector
  $h_f \in \Binfty_W(\X)$ such that $h_f(x) > 0$ for any $x\in\X$.
  Moreover,~$h_f$ is the only positive eigenvector of~$P_t^f$ (up to multiplication by
  a positive constant). Finally,~$r(f)$ is equal to the cumulant function defined
  in~\eqref{eq:SCGF}:
  \begin{equation}
    \label{eq:SCGFspec}
  r(f) =\lambda(f).
  \end{equation}
\end{lemma}

The result of Lemma~\ref{lem:specLf} is twofold: it entails the well-posedness of the principal
eigenproblem associated with~$P_t^f$ for any $f\in\Binfty_\kap(\X)$ and $t>0$, and then identifies
this principal eigenvalue with the free energy function~\eqref{eq:SCGF}.
Another consequence of this lemma is that~$h_f$ is in fact the principal eigenvector of~$\Lc+f$,
see Lemma~\ref{lem:pcpal_eig_L_f} below for a more precise statement.

\begin{proof}
  We follow the general strategy of~\cite{ferre2018more} and split the proof into several steps.

  \paragraph{Step 1: Compactness of the evolution operator.}
  We first show that, for given $t >0$
  and $f\in\Binfty_\kap(\X)$, the operator~$P_t^f$ defined in Lemma~\ref{lem:feynmankac} is
  compact when considered on~$\Binfty_W(\X)$. For any compact set~$K\subset\X$ we have the decomposition
  \begin{equation}
    \label{eq:decomp}
    P_t^f = P_{t/3}^f\ind_K P_{t/3}^f \ind_K P_{t/3}^f
  + P_{t/3}^f\ind_{K^c} P_{2t/3}^f 
  + P_{t/3}^f\ind_K P_{t/3}^f \ind_{K^c} P_{t/3}^f.
  \end{equation}
  We first consider the compact sets~$K_a$ from~\eqref{eq:compactness} for $a>0$ and time~$t/3$ (omitting
  the dependence on~$t$ in the notation since the time is fixed here) and note
  that~$\ind_{K_a^c}P_{t/3}^f$ converges to~$0$ in operator norm as $a\to +\infty$. Indeed, for
  any~$\varphi\in\Binfty_W(\X)$,~\eqref{eq:compactness} leads to
  \begin{equation}
    \label{eq:boundoperator}
  \left\| \ind_{K_a^c} P_{t/3}^f\varphi \right\|_{\Binfty_W} \leq \| \varphi\|_{\Binfty_W}\,\e^{-a t/3},
  \end{equation}
  so that $\left\| \ind_{K_a^c} P_t^f \right\|_{\mathcal{B}(\Binfty_W)}\to 0$ when $a\to +\infty$.

  We next show that $ P_{t/3}^f\ind_K P_{t/3}^f \ind_K$ is compact over~$\Binfty_W(\X)$
  for any compact set $K\subset\X$.
  Consider a sequence~$(\varphi_k)_{k\in\N}$ bounded in~$\Binfty_W(\X)$.
  Following the first step of the proof of~\cite[Lemma~2]{ferre2018more} and using our 
  strong Feller result, Lemma~\ref{lem:regularity}, we see
  that~$P_{t/3}^f\ind_K$ is a strong Feller operator, so
  $P_{t/3}^f\ind_K P_{t/3}^f \ind_K$ is ultra-Feller (see~\cite[Lemma~6]{ferre2018more}).
  This means that the operator $P_{t/3}^f\ind_K P_{t/3}^f \ind_K$ is continuous
  in total variation norm, so that the family
  $(P_{t/3}^f\ind_K P_{t/3}^f \ind_K\varphi_k)_{k\in\N}$ is uniformly equicontinuous.
  We used here that since $\varphi\in\Binfty_W(\X)$ and~$W$ is continuous, it holds
  $\ind_K \varphi \in\Binfty(\X)$.  The sequence
  $(P_{t/3}^f\ind_K P_{t/3}^f \ind_K\varphi_k)_{k\in\N}$ therefore converges
  in~$\Binfty(\X)$ up to extraction by the Ascoli
  theorem~\cite[Theorem~11.28]{rudin2006real}, and in~$\Binfty_W(\X)$ since $W\geq 1$. Therefore,
  the operator $P_{t/3}^f\ind_K P_{t/3}^f \ind_K$
  sends a bounded sequence into a convergent one (up to extraction), so it is compact
  in~$\Binfty_W(\X)$~\cite{reed1980functional}. The decomposition~\eqref{eq:decomp} and the
  bound~\eqref{eq:boundoperator} then show that~$P_t^f$ is the limit in operator norm of the compact
  operators~$P_{t/3}^f\ind_{K_a} P_{t/3}^f \ind_{K_a} P_{t/3}^f$ as $a\to +\infty$, so it is
  compact in~$\Binfty_W(\X)$ (see \textit{e.g.}~\cite[Theorem~VI.12]{reed1980functional}).

  \paragraph{Step 2: Existence of the principal eigenvalue.}
  We can now use the Krein--Rutman theorem on the (closed) total cone
  $\mathbb{K}_W = \{ \varphi \in \Binfty_W\,|\, \varphi \geq 0\}$
  (see~\cite{deimling2010nonlinear,ferre2018more} for definitions). For $t>0$, it is clear
  that~$P_t^f$ leaves this cone invariant. We next show that~$P_t^f$ has a non-zero
  spectral radius
  \[
  R_{t}(f)
    = \lim_{n\to +\infty} \big\| (P_{t}^f)^n\big\|_{\mathcal{B}(\Binfty_W)}^{\frac{1}{n}}.
  \]
  To this end, fix a compact set~$K$ with non-empty interior. We have shown
  in Lemma~\ref{lem:irreducibility} that
  \[
  \forall\,x\in K,\quad \big(P_{t}^f\ind_K\big) (x) >0.
  \]
  Since~$P_{t}^f\ind_K$ is continuous by Lemma~\ref{lem:regularity}, this shows that
  \begin{equation}
    \label{eq:alphaK}
  \alpha_K:=\min_{x\in K} \big(P_{t}^f\ind_K\big)(x) >0.
  \end{equation}
  Therefore, for any $x\in K$,
  \[
  \begin{aligned}
    \left[\left(P_t^f\right)^2\ind_K\right](x) & =
    \E_x\left[ (P_t^f\ind_K)(X_t)\,\e^{\int_0^t f(X_s)\,ds}\right]
    \geq \E_x\left[ \ind_{ K}(X_t)(P_t^f\ind_K)(X_t)\,\e^{\int_0^t f(X_s)\,ds}\right]
  \\ & \geq  \alpha_K \E_x\left[ \ind_{K}(X_t)\,\e^{\int_0^t f(X_s)\,ds}\right]
    = \alpha_K \big(P_t^f\ind_K\big)(x)
     \geq \alpha_K^2,&
  \end{aligned}
  \]
  so that $\ind_K(x)\left((P_t^f)^2\ind_K\right)(x) \geq \alpha_K^2 \ind_K(x)$ for $x\in \X$.
  Iterating the procedure for any $n\geq 1$ we get
  \[
  \left\| (P_{t}^f)^n\right\|_{\mathcal{B}(\Binfty_W)}
  \geq \left\| \ind_K (P_t^f)^n\ind_K\right\|_{\Binfty_W}
  \geq \alpha_K^n \left\|\ind_K\right\|_{\Binfty_W} = \frac{\alpha_K^n}{\inf_K W}.
  \]
  As a result, since $1\leq\inf_K W<+\infty$, we obtain in the
  large~$n$ limit the following lower bound for the spectral radius:
  \[
  R_{t}(f)  \geq \alpha_K >0,
  \]
  which shows that $R_t(f)$ is positive. Since~$P_{t}^f$ is
  compact,~\cite[Theorem~19.2]{deimling2010nonlinear} ensures that~$R_{t}(f)$
  is a real eigenvalue of~$P_{t}^f$ with associated eigenvector $h_f\in \mathbb{K}_W$ (in
  particular, $h_f\geq 0$). Using the semigroup property of~$P_t^f$ and standard
  arguments (see~\cite[Theorem~2.4]{pazy2012semigroups}), we can show
  that there exists~$r(f) \in \mathbb{R}$ such that $R_{t}(f)=\e^{r(f)t}$ and 
  \begin{equation}
    \label{eq:Ptfeigen}
    P_t^f h_f = \e^{r(f)t} h_f.
  \end{equation}

  \paragraph{Step 3: Properties of $h_f$.}
  For the remainder of the proof, we write for simplicity $r:=r(f)$ and $h:=h_f$ (the function~$f$
  being fixed). We show here that~$h$ is continuous and positive. For any
  compact~$K\subset \X$ and $t>0$,~\eqref{eq:Ptfeigen} leads to
  \[
  \begin{aligned}
  \left|P_t^f(\ind_K h) - \e^{r t} h\right| &= \left|P_t^f(\ind_K h) - P_t^f h\right|  =
  \left|P_t^f(\ind_{K^c} h)\right| 
   =  \left|P_t^f\big(\ind_{K^c} \e^{- r t} P_t^f h\big)\right|
  \\ & \leq \e^{- r t}
  \| h\|_{\Binfty_W} \|P_t^f\|_{\mathcal{B}(\Binfty_W)} \left\|\ind_{K^c} P_t^f W\right\|_{\Binfty_W} |W|.
  \end{aligned}
  \]
  Using Lemma~\ref{lem:feynmankac} we obtain that, for any~$a>0$, there exists a compact
  set~$K_a$ such that
  \[
  \big\|\e^{-r t}P_t^f(\ind_{K_a} h) -  h\big\|_{\Binfty_W}\leq C \e^{-a t}\quad
  \mathrm{with}
  \quad C =  \e^{- 2 r t} 
  \| h\|_{\Binfty_W} \|P_t^f\|_{\mathcal{B}(\Binfty_W)},
  \]
  so that~$h$ is continuous as the uniform limit of continuous functions (since~$P_t^f(\ind_{K_a} h)$
  is continuous by Lemma~\ref{lem:regularity}). Finally, since $h\geq 0$ and~$h$ is
  not identically equal to~$0$, there exists $x_0\in\X$ such that $h(x_0) >0$. Moreover~$h$ is
  continuous, so there is $\varepsilon>0$ for which $h>0$ on $B(x_0,\varepsilon)$.
  By~\eqref{eq:Ptfeigen} it holds, for any $x\in\X$,
  \[
  \e^{rt}h(x) = (P_t^fh)(x) \geq P_t^f \big(h\ind_{B(x_0,\varepsilon)}\big)(x) 
  \geq \left(\inf_{B(x_0,\varepsilon)}h\right) \big(P_t^f\ind_{B(x_0,\varepsilon)}\big)(x). 
  \]
  Since $h>0$ on $B(x_0,\varepsilon)$ and~$h$ is continuous,
  $\inf_{B(x_0,\varepsilon)}h>0$. Moreover $(P_t^f\ind_{B(x_0,\varepsilon)})(x) >0$ for
  any~$x\in\X$ by Lemma~\ref{lem:irreducibility}, so the previous lower bound shows 
  that $h(x) >0$ for all $x\in\X$.
  
  \paragraph{Step 4: Properties of eigenspaces and eigenfunctions.}
  We now show that the eigenspace associated with~$h$ is of dimension one, 
  and that any other eigenvector vanishes somewhere in~$\X$.
  For this, we introduce the so called $h$-transform~\cite{kontoyiannis2005large,rousset2006control,chetrite2015nonequilibrium,ferre2018more}. A key element here
  is the fact that~$h(x)>0$ for all~$x\in\X$, which allows to define the following Markov
  operator, for an arbitrary time~$t>0$:
  \begin{equation}
    \label{eq:Q_h}
  Q_h\varphi = \e^{-r t}h^{-1}P_t^f (h\varphi),
  \end{equation}
  where~$h$ and~$h^{-1}$ refer here to the multiplication operators by the functions~$h$
  and~$h^{-1}$ respectively.
  We now prove that~$Q_h$ is ergodic by first
  noting that~$Q_h$ admits~$Wh^{-1}$ as a Lyapunov function
  (using~\eqref{eq:compactness} and the normalization $\|h\|_{\Binfty_W}=1$ which implies
  that $Wh^{-1}\geq 1$). Using Assumption~\ref{as:lyapunov}, we can also show that~$Wh^{-1}$ has
  compact level sets, see~\cite[Appendix~E]{ferre2018more} for details.

  Moreover, we can prove that~$Q_h$ satisfies a minorization condition on any compact set.
  For this, we first use that $P_t^f \geq P_t^{-\|f\|_{\Binfty_\kappa} \kappa}$. Then, for any $t>0$ and $\alpha \geq 0$,
  the operator $P_t^{-\alpha \kappa}$ has
  a smooth transition density by hypoellipticity (because~$\kappa$ and the coefficients
  of~$\Lc$ belong to the class~$\mathscr{S}$, see~\cite[Theorem~4.1]{eckmann2003spectral}),
  which is
  positive in view of Lemma~\ref{lem:irreducibility} by an argument similar to the one sketched after
  Assumption~\ref{as:control} (see for instance the proof of~\cite[Proposition~8.1]{bellet2006ergodic}).
  Therefore, for any $K\subset\X$ compact with non-empty interior, and denoting by~$\eta_K$ the uniform
  Lebesgue measure on~$K$, there is~$a_K>0$ such that,
  for any measurable set~$A\subset\X$,
  \[
  \forall\,x\in K,\quad  \big(P_t^f\ind_A\big)(x) \geq \left(P_t^{-\|f\|_{\Binfty_\kappa} \kappa}\ind_A\right)(x)
  \geq a_K\eta_K(A).
  \]
  Since~$h$ is continuous, this implies that, for any measurable~$\varphi \geq 0$,
  \[
  \forall\,x\in K,\quad 
  \big(Q_h\varphi\big)(x) \geq \frac{|K|a_K\min_K h}{\max_K h}\frac{\eta_K(\varphi)}{|K|},
  \]
  where both the minimum and maximum above are finite and non-zero (recall that~$|K|>0$
  is the Lebesgue measure of~$K$).
  This shows that~$Q_h$ satisfies a minorization condition~\cite{hairer2011yet} over any
  compact set.
  
  Therefore, the Markovian dynamics with kernel~$Q_h$ admits a unique invariant probability
  measure~$\mu_h$, with respect to which it is ergodic in~$\Binfty_{Wh^{-1}}(\X)$. By this we mean
  that (in view of~\cite[Theorem~1.2]{hairer2011yet}) there exist $\bar{\alpha} >0$ and $C>0$
  such that for any $\varphi\in\Binfty_{Wh^{-1}}(\X)$, 
  \begin{equation}
    \label{eq:hergodic}
    \forall\,n\geq 1,\quad
    \big\| (Q_h)^n\varphi - \mu_h(\varphi)\big\|_{\Binfty_{Wh^{-1}}}\leq C \e^{-\bar{\alpha} n} 
  \|\varphi - \mu_h(\varphi)\|_{\Binfty_{Wh^{-1}}},
  \end{equation}
  and it holds $\mu_h(W/h)<+\infty$.

  We can now use this ergodic behaviour to show that the eigenspace associated with~$r$
  has dimension one and that~$P_t^f$ cannot have another positive eigenvector with norm~$1$
  in~$\Binfty_W(\X)$.
  Indeed, if there were another eigenvector $\tilde{h}\in\Binfty_W(\X)$ associated with~$r$,
  then the fact that $\tilde{h}/h \in \Binfty_{Wh^{-1}}(\X)$ together with~\eqref{eq:hergodic}
  ensure that
  \[
  (Q_h)^n\left(\frac{\tilde{h}}{h}\right) = \frac{\tilde{h}}{h}\xrightarrow[n\to +\infty]{}
  \mu_h\left(\frac{\tilde{h}}{h}\right).
  \]
  This shows that~$h$ and~$\tilde{h}$ would be proportional, and answers the claim that
  the eigenspace associated with~$r$ has
  dimension~$1$. Assume now that there is another real eigenvalue $\tilde{r}<r$ with
  real eigenvector $\tilde{h} \in\Binfty_W(\X)$ such that $\tilde{h}(x) >0$ for all $x\in\X$.
  Noting again that $\tilde{h}/h \in  \Binfty_{Wh^{-1}}(\X)$ and since
  $\tilde{h} >0$,~\eqref{eq:hergodic} shows that, for any $x\in\X$,
  \begin{equation}
    \label{eq:zeroergo}
  (Q_h)^n\left(\frac{\tilde{h}}{h} \right)(x) \xrightarrow[n\to +\infty]{}
  \mu_h\left(\frac{\tilde{h}}{h}\right) > 0.
  \end{equation}
  However it now holds, for any $x\in\X$,
  \[
  (Q_h)^n\left(\frac{\tilde{h}}{h} \right)(x) = \e^{(\tilde{r}-r)tn} \frac{\tilde{h}}{h}(x)
   \xrightarrow[n\to +\infty]{} 0,
  \]
  where we used that $h>0$ and $\tilde{r}<r$. Combining the two
  equations above shows that
  \[
  \mu_h\left(\frac{\tilde{h}}{h}\right)=0,
  \]
  which contradicts~\eqref{eq:zeroergo}. As a
  result, there cannot be another eigenvalue with a positive eigenvector.

  \paragraph{Step 5: The principal eigenvalue is the cumulant function.}
  Proving~\eqref{eq:SCGFspec} now follows by a simple rewriting. For $x\in\X$ and
  $t_0 >0$ fixed, it holds, for any $n\in\N^*$,
  \[
  \E_x\left[ \e^{\int_0^{nt_0} f(X_s)\, ds}\right] =\left[ (P_{t_0}^f)^n\ind\right](x)
  = \e^{rnt_0}\big[h (Q_h)^n h^{-1}\big](x),
  \]
  so that
  \[
  \frac{1}{nt_0}\log\E_x\left[ \e^{\int_0^{nt_0} f(X_s)\, ds}\right] =
  \frac{1}{nt_0} \log\left[\e^{ rnt_0 } h (Q_h)^n h^{-1}(x)\right]
  = r + \frac{1}{nt_0} \log\left[ h (Q_h)^n h^{-1}(x)\right].
  \]
  By~\eqref{eq:hergodic} (since $h^{-1}\in\Binfty_{Wh^{-1}}(\X)$), we see that
  $\left(h (Q_h)^n h^{-1}\right)(x)$ converges to~$\mu_h(h^{-1}) h(x)$ (with~$x$ fixed), so that
  \[
  r(f)=\lim_{n\to+\infty}\ \frac{1}{nt_0}\log\E_x\left[ \e^{\int_0^{nt_0} f(X_s)\, ds}\right].
  \]
  We have chosen to work with an arbitrary time $t_0>0$ for convenience, so a priori
  the above limit depends on~$t_0$. To conclude the proof, it remains to show that the limit actually does not depend on
  the specific choice of~$t_0$ and that
  \[
  r(f)=\lim_{t\to+\infty}\ \frac{1}{t}\log\E_x\left[ \e^{\int_0^{t} f(X_s)\, ds}\right].
  \]
  This extension from~$t_0>0$ fixed to any~$t>0$ follows by standard arguments not reproduced here
  (see \textit{e.g.}~\cite{herzog2017ergodicity,ferre2018more}).  
\end{proof}

An important ingredient for the lower bound of the LDP is the Gateau-differentiability of the
cumulant functional, which we prove below.

\begin{lemma}
  \label{lem:miam}
  The functional
  \begin{equation}
    \label{eq:cumulantgateau}
  f\in\Binfty_\kap(\X)\mapsto \lambda(f) = \underset{t\to + \infty}{\lim}\,
  \frac{1}{t} \log \E_x\left[ \e^{\int_0^t f(X_s)\, ds} \right]
  \end{equation}
  is convex and Gateau-differentiable.
\end{lemma}

\begin{proof}
  The convexity of~$\lambda$ is a standard consequence of Hölder's inequality.
  Concerning Gateau-differentiability,
  we follow the strategy of~\cite[Section~3]{gartner1977large} for a compact state space,
  relying on results of Kato~\cite{kato2013perturbation}. For this, we interpret
  the cumulant function~\eqref{eq:cumulantgateau} as the largest eigenvalue of the tilted
  generator,~$r(f)$, as shown in Lemma~\ref{lem:specLf}. More precisely,
  for $f,g\in\Binfty_\kap(\X)$ and~$\alpha\in\R$, $\lambda(f+\alpha g)$
  is associated with the largest eigenvalue of the operator~$P_t^{f+\alpha g}$ 
  in~$\Binfty_W(\X)$ through
  \[
  P_t^{f+\alpha g} h_{f+\alpha g}= \e^{t \lambda (f+\alpha g)} h_{f+\alpha g},
  \]
  so that derivability in~$\alpha$ can be shown through the differentiability of the
  spectrum of a bounded operator. We thus show that the operator-valued function~$\alpha \mapsto P_t^{f+\alpha g}$
  is differentiable in operator norm.

  To this end, we fix $C>0$, and prove that for 
  $|\alpha|\leq C$, there exists $K \in \mathbb{R}_+$ such that 
  \begin{equation}
    \label{eq:taylorP}
    \left\| P_t^{f+\alpha g} - P_t^f - \alpha Q_t^{f,g}\right\|_{\mathcal{B}(\Binfty_W(\X))} \leq K\alpha^2, 
  \end{equation}
  where 
  \[
    Q_t^{f,g}:\varphi\in\Binfty_W(\X)\mapsto \E_x\left[ \varphi(X_t)\left(\int_0^t
        g(X_s)\,ds\right) \, \e^{\int_0^t f(X_s)\, ds}\right].
  \]
  Note that the operator $Q_t^{f,g}$ is bounded on~$\Binfty_W(\X)$ by the same
  martingale estimate used to prove Lemma~\ref{lem:diffbound}. In order to prove~\eqref{eq:taylorP}, we use the identity
  \[
    0 \leq \left|\mathrm{e}^b - 1 - b \right| \leq \frac{b^2}{2} \mathrm{e}^{|b|} 
  \]
  to obtain, for any~$\varphi\in\Binfty_W(\X)$ and $x\in\X$,
  \[
    \begin{aligned}
      & \left| \left[\left(P_t^{f+\alpha g} - P_t^f - \alpha Q_t^{f,g}\right)\varphi\right](x) \right| \\
      & \qquad \leq  \| \varphi\|_{\Binfty_W} \E_x\left[ W(X_t) \, \e^{\int_0^t f(X_s)\, ds} \left|\e^{\alpha\int_0^t g(X_s)\,ds}-1-\alpha\int_0^t g(X_s)\,ds\right|\right] \\
      & \qquad \leq  \frac{\alpha^2}{2} \| \varphi\|_{\Binfty_W} \E_x\left[ W(X_t) \, \e^{\left(\|f\|_{\Binfty_\kappa}+\alpha\|g\|_{\Binfty_\kappa}\right)\int_0^t \kappa(X_s)\, ds} \left( \int_0^t g(X_s)\,ds \right)^2\right] \\
      & \qquad \leq  \alpha^2 \| \varphi\|_{\Binfty_W} \E_x\left[ W(X_t) \, \e^{\left(\|f\|_{\Binfty_\kappa}+(1+\alpha)\|g\|_{\Binfty_\kappa}\right)\int_0^t \kappa(X_s)\, ds}\right], \\
    \end{aligned}
  \]
  where we used the inequality $z^2/2 \leq \e^z$ for $z\geq0$
  in the last line. By manipulations similar to the one used to prove
  Lemma~\ref{lem:diffbound}, we can bound the latter expectation by~$\e^{ct}W(x)$ for
  some constant $c>0$, which leads to~\eqref{eq:taylorP} with $K = \e^{ct}$.
  
  Equation~\eqref{eq:taylorP} shows that $\alpha\mapsto P_t^{f+\alpha g}$ is differentiable in operator
  norm, and that
  \[
    \frac{d}{d\alpha}\Big|_{\alpha = 0}P_t^{f+\alpha g} = Q_t^{f,g}.
  \]
  Thus, the principal eigenvalue $\lambda(f +\alpha g)$,
  which is always isolated, is differentiable,
  see~\cite[Chapter~II, Theorem~5.4]{kato2013perturbation}
  and~\cite[Chapter~IV, Theorem~3.5]{kato2013perturbation}. This concludes the proof
  of Gateau-differentiability.
\end{proof}

\begin{remark}
  By pursuing further the Taylor expansion~\eqref{eq:taylorP} in the proof of Lemma~\ref{lem:miam},
  we can actually show that, for any $f,g\in\Binfty_\kap(\X)$, the function
  \[
  \alpha \in\C \mapsto \lambda (f + \alpha g)
  \]
  is analytic (this analyticity was
  already proven in~\cite{kontoyiannis2005large} using a different argument that can be simplified
  with our tools). This relies on the simple inequality~$a^n/n! \leq \e^a$ for any~$a\geq 0$,
  together with the series expansion of the exponential and martingale estimates as in
  the proof of Lemma~\ref{lem:miam}. Indeed, our proof, based on martingales, shows that for
  any~$t>0$, the function
  \[
  \alpha \mapsto \frac{1}{t}\log \E_x \left[ \e^{\int_0^t \big(f(X_s)+\alpha g(X_s)
        \big)\,ds} \right]
  \]
  is analytic. Moreover, it is finite on~$\R$ and converges pointwise to a finite valued
  function as $t\to +\infty$, as shown in Lemma~\ref{lem:specLf}. Therefore, the convergence
  holds uniformly on any compact as~$t\to +\infty$ (see~\cite[Theorem~VI.3.3]{ellis2007entropy}).
  Since a locally uniform limit of analytic functions is analytic
  (see~\cite[Theorem~10.28]{rudin2006real}), the function $\alpha \mapsto \lambda(f + \alpha g)$ is analytic.
\end{remark}

The last step before proving the large deviations principle itself is an exponential
tightness result, see~\cite[Section~1.2]{dembo2010large}.
At this stage, the finiteness of~$\lambda(f)$ together with the Gateau-differentiability
of $f\in\Binfty_\kap(\X)\mapsto\lambda(f)$
already provides the upper bound over compact sets and the lower bound in~\eqref{eq:LDP}.
In order to extend the upper bound to all closed sets, we prove exponential tightness in
the~$\tau^\kap$-topology, see Appendix~\ref{sec:tools} for some definitions (this
exponential tightness is not explicitely stated in~\cite{kontoyiannis2005large}).
\begin{lemma}
  \label{lem:tightness}
  The family of probability measures $t\mapsto \proba_x ( L_t \in\cdot\, )$
  over~$\mathcal{P}(\X)$ is exponentially tight in the $\tau^\kap$-topology. 
\end{lemma}

\begin{proof}
  We adapt the strategy of~\cite[Corollary~2.3]{wu2001large} and~\cite[Section~2.2]{wang2010sanov}
  by introducing the family of sets
  \[
  \Gamma_{N} = \big\{ \nu\in\mathcal{P}(\X)\, \big| \, \nu(\Ps) \leq N\big\},
  \quad N>0.
  \]
  For $N>0$, the sets~$\Gamma_N$ are subsets of $\mathcal{P}_\kap(\X)$ since $\kap\ll \Ps$.
  We show that they are actually precompact in the $\tau^\kap$-topology.

  Let us first show that~$\Gamma_{N}$ is precompact in the usual weak topology
  for any~$N>0$.  Consider for this the compact
  sets~$K_\beta =\{ x\in\X \, |\, \Ps(x)\leq \beta\}\subset \X$ for $\beta>0$
  (recall that~$\Ps$ has compact level sets). Then, for any~$\nu\in \Gamma_{N}$, we have
  \[
  \beta \nu(K_\beta^c) +  \nu( \Ps \ind_{K_\beta}) \leq  \nu( \Ps \ind_{K_\beta^c})
  +  \nu( \Ps \ind_{K_\beta})=\nu(\Ps) \leq N.
  \]
  This shows that for any~$\beta >0$ and any~$\nu\in\Gamma_{N}$,
  \[
  \nu(K_\beta^c) \leq \frac{N}{\beta},
  \]
  hence (upon choosing~$\beta$ sufficiently large) for any $N>0$ the family of
  measures~$\Gamma_{N}$ is tight, so it is precompact
  for the weak topology by the Prohorov theorem~\cite{billingsley2013convergence}. Now,
  if~$\kap$ is bounded,~$\Gamma_N$ is tight for the $\tau^\kap$-topology and the theorem is shown,
  so we may assume that~$\kap$ has compact level sets (see Assumption~\ref{as:lyapunov}).
  For proving compactness in our weighted topology, we show that~$\kap$ is
  uniformly integrable over~$\Gamma_{N}$ in order to
  use~\cite[Theorem~7.12]{villani2003topics}. Since $\kap\ll \Ps$, the set
  \[
  A_n=\left\{ x\in\X\, \left| \, \frac{\Ps(x)}{\kap(x)}\leq n\right.\right\}
  \]
  is compact for any~$n\geq 1$. Moreover, since we assume~$\kap$ to be continuous with compact
  level sets, for any $n\geq 1$ there exists $m_n \geq n$ such that
  \[
  \left\{ \frac{\Ps}{\kap} \leq n\right\} \subset \{ \kap\leq m_n\},
  \]
  with $m_n\to +\infty$ when $n\to +\infty$.
  Therefore, for any $\nu\in\Gamma_{N}$ and $n\geq 1$,
  \[
  \int_{\{\kap > m_n\}}\kap \,d\nu\leq
  \int_{A_n^c} \kap \,d\nu = \frac{1}{n}  \int_{A_n^c}n \kap \,d\nu\leq
  \frac{1}{n} \int \Ps\,d\nu = \frac{1}{n}\nu(\Ps)
  \leq \frac{N}{n}.
  \]
  Taking the supremum over $\nu\in\Gamma_N$  in the above equation
  and recalling that $m_n\to +\infty$ when $n\to +\infty$ we obtain
  \begin{equation}
    \label{eq:equiint}
  \lim_{m\to +\infty}\, \sup_{\nu\in\Gamma_N}\,\int_{\{\kap>m\}} \kap\,d\nu = 0.
  \end{equation}
  We can then conclude that~$\Gamma_N$ is precompact for the $\tau^\kap$-topology. Consider indeed
  a sequence $(\nu_n)_{n\in\N}\subset\Gamma_N$. By Prohorov's theorem,~$(\nu_n)_{n\in\N}$ has
  a subsequence weakly converging towards a measure~$\nu$, \textit{i.e.}
  $\nu_n(\varphi)\to\nu(\varphi)$
  for any~$\varphi\in\Cb(\X)$. Then, by~\cite[Theorem~7.12]{villani2003topics},~\eqref{eq:equiint}
  ensures that~$\nu\in\mathcal{P}_\kap(\X)$ and for any
  $f\in\Binfty_\kap(\X)$, $\nu_n(f)\to\nu(f)$ as $n\to +\infty$. In other
  words,~$\Gamma_N$ is precompact for the $\tau^\kap$-topology.

  We can now prove the $\tau^\kap$-exponential tightness of the empirical
  distribution~$(L_t)_{t\geq 0}$ in~$\mathcal{P}(\X)$. Indeed, for any~$N,t>0$,
  Tchebytchev's inequality leads to
  \[
  \proba_x\big( L_t \in \Gamma_{N}^c \big)  = \proba_x
  \left( \int_0^t \Psi(X_s)\, ds >  N t \right)
   \leq \e^{- N t} \E_x \left[ \e^{\int_0^t  \Ps(X_s)\, ds}\right]
   =\e^{-  N t} \left(P_t^{ \Ps}\ind\right)(x).
  \]
  Renormalizing at log scale leads to
  \begin{equation}
    \label{eq:neartight}
  \underset{t\to +\infty}{\limsupb}\ \frac{1}{t} \log\, \proba_x\big( L_t \in \Gamma_{N}^c \big)
  \leq -  N +\underset{t\to +\infty}{\limsupb}\ \frac{1}{t} \log \left[\left(P_t^{\Ps}\ind\right)(x)\right].
  \end{equation}
  The right hand side of the above quantity may look infinite since~$\Ps$ grows faster
  than~$\kap$. However, using again the martingale~$M_t$ defined in Lemma~\ref{lem:martingales}
  we obtain, for any~$t>0$,
  \[
  \E_x\left[\e^{\int_0^t\Ps(X_s)\,ds}\right]\leq
  \E_x\left[W(X_t)\,\e^{-\int_0^t\frac{\Lc W}{W}(X_s)\,ds}\right] = \E_x[M_t]\leq W(x).
  \]
  Thus it holds
  \[
  \underset{t\to +\infty}{\limsupb}\ \frac{1}{t} \log \left[\left(P_t^{\Ps}\ind\right)(x)\right]
  \leq \underset{t\to +\infty}{\limsupb}\ \frac{1}{t} \log
  \E_x\left[W(X_t)\,\e^{\int_0^t\Ps(X_s)\,ds}\right]\leq \underset{t\to +\infty}{\limsupb}\ \frac{1}{t}
  \log W(x) = 0.
  \]
  As a result,~\eqref{eq:neartight} becomes
  \[
    \underset{t\to +\infty}{\limsupb}\ \frac{1}{t} \log\, \proba_x\big( L_t \in \Gamma_{N}^c \big)
  \leq -  N.
  \]
  Since~$\Gamma_N$ is precompact in the~$\tau^\kap$-topology
  for any $N>0$, and~$N$ can be chosen arbitrarily large, this proves the exponential tightness
  of the family of empirical distributions in the $\tau^\kap$-topology.
\end{proof}

We are now in position to prove Theorem~\ref{theo:LDP}. 

\begin{proof}[Proof of Theorem~\ref{theo:LDP}]
  The previous lemmas make it possible to apply the Gärtner--Ellis
  theorem (recalled in Appendix~\ref{sec:tools}). The function~$\Lambda$ in Theorem~\ref{theo:GE} of Appendix~\ref{sec:tools}
  is the cumulant function
  \[
  \lambda : f\in \Binfty_\kap(\X)\mapsto \lim_{t\to +\infty} \frac{1}{t}\log \E_x
  \left[ \e^{\int_0^t f(X_s)\, ds}\right].
  \]
  The topological dual
  of $(\mathcal{M}_\kap(\X),\tau^\kap)$ is~$\Binfty_\kap(\X)$, where~$\mathcal{M}_\kap(\X)$ is
  the set of measures over~$\X$ integrating~$\kap$
  (see~\cite{rudin2006real,kontoyiannis2005large} and~\cite[Lemma~3.3.8]{deuschel2001large} for
  details). We have proved that~$\lambda$ is well defined,
  Gateau-differentiable, and that the family of measures
    \[
    t\mapsto    \pi_t(\,\cdot\, ) := \mathbb{P}_x\left( L_t \in \cdot\ \right),
    \]
    is exponentially tight in the $\tau^\kap$-topology. Therefore,~$(\pi_t)_{t\geq 0}$ satisfies
    a large deviations principle in the $\tau^\kap$-topology with good rate
    function given by
    \begin{equation}
      \label{eq:Ifproof}
      \forall\,\nu\in \mathcal{M}(\X), \quad
      I(\nu) = \sup_{f\in\Binfty_\kap} \big\{ \nu(f) - \lambda(f)\big\}.
    \end{equation}
    Note first that $I(\nu) \geq 0$. We next observe that $I(\nu)=+\infty$ if~$\nu$ is not
    normalized to~1 (take~$f$ to be constant in the supremum~\eqref{eq:Ifproof}), so we may
    consider~$I$ over~$\mathcal{P}(\X)$.
    Moreover, choosing $f=\kap$ in~\eqref{eq:Ifproof} and noting that $\lambda(\kap)<+\infty$ by Lemma~\ref{lem:specLf},
    we get $I(\nu)=+\infty$ if $\nu\notin\mathcal{P}_\kap(\X)$. If~$\nu$ is not absolutely
    continuous with respect to~$\mu$, there exists a measurable set~$A\subset\X$ such that $\mu(A) = 0$ and
    $\nu(A)>0$. Since~$\mu$ has
    a positive density with respect to the Lebesgue measure, this means that~$A$ has zero Lebesgue measure. Consider then
    $f_a = a\ind_A\in\Binfty_\kap(\X)$ for $a \in\R$. Since~$A$ has zero Lebegue measure
    and~$(X_t)_{t\geq 0}$ has a smooth density for all $t>0$ (as a consequence of
    Assumption~\ref{as:regularity}) it holds, for all $t>0$, 
    \[
    \E_x\big[f_a(X_t)\big] = a \proba_x \big( X_t \in A\big)=0.
    \]
    Therefore, the process
    \[
    Z_t = \int_0^t f_a(X_s)\,ds,
    \]
    satisfies $\E_x[Z_t]=0$ for all~$t>0$. Since $Z_t\geq 0$, it holds $Z_t=0$ almost
    surely, for any $t>0$. As a consequence we obtain
    \[
    \forall\,t>0,\quad \frac{1}{t}\log \E_x\left[ \e^{\int_0^t f_a(X_s)\,ds}\right]
    = \frac{1}{t}\log \E_x\left[ \e^{Z_t}\right]
    =0.
    \]
    This shows that $\lambda(f_a)=0$, so that from~\eqref{eq:Ifproof} we obtain
    \[
    I(\nu)\geq a \nu(A),
    \]
    with~$\nu(A)>0$. By letting $a\to +\infty$  we are led to $I(\nu)=+\infty$.
    
    Finally, we show that $I(\nu)=0$ if and only if $\nu=\mu$, and that~$(L_{t_n})_{n\geq 0}$ converges
    almost surely to~$\mu$ in the $\tau^\kap$-topology for any sequence~$(t_n)_{n\geq 0}$
    such that $t_n/\log(n)\to +\infty$ (see~\cite[Appendix~B]{dembo2010large}
    for the definition of this almost-sure convergence). Define 
    \[
    \mathscr{I}=\left\{\nu\in\mathcal{P}(\X)\,\left|\,
    I(\nu)=\inf_{\mathcal{P}(\X)}\,I\right. \right\}.
    \]
    Since~$I$ has compact level sets (because it is a good rate function, see
    Theorem~\ref{theo:GE}),~$\mathscr{I}$ is a non-empty closed subset of~$\mathcal{P}(\X)$ for the $\tau^\kap$-topology. Moreover, in order for the
    LDP upper bound to make sense, it holds $\inf_{\mathcal{P}(\X)}\,I = 0$.
    If~$\mathscr{I}_\delta$ denotes an open neighborhood of~$\mathscr{I}$,
    the lower semicontinuity of~$I$ implies that 
    \[
    \inf_{\mathscr{I}_\delta^c}\, I >0.
    \]
    Therefore, by the large deviations upper bound we have, for any $t\geq 0$,
    \begin{equation}
      \label{eq:expoineq}
        \proba_x\big( L_t \notin \mathscr{I}_\delta \big)
        = \proba_x\big( L_t \in \mathscr{I}_\delta^c\big)\leq C \,
        \exp\left(- t \inf_{\mathscr{I}_\delta^c}\,I\right),
        \end{equation}
        for some constant~$C>0$. Consider now a sequence $(t_n)_{n \geq 1}$ such that $t_n/\log(n) \to +\infty$ as $n \to +\infty$.
        In particular, there exists $n_\star \in \mathbb{N}$ such that $t_n \inf_{\mathscr{I}_\delta^c} I \geq 2 \log(n)$ for $n \geq n_\star$, which implies 
        \[
      \sum_{n \geq 0}  \proba_x\big( L_{t_n} \notin \mathscr{I}_\delta \big)
    \leq n_\star + C \sum_{n \geq n_\star} \frac{1}{n^2} < +\infty.
    \]
    This shows that~$(L_{t_n})_{n\geq 0}$ converges almost surely to~$\mathscr{I}$
    in the $\tau^\kap$-topology, by the Borel-Cantelli lemma (and by definition of convergence in a topological
    space~\cite[Appendix~B]{dembo2010large}). However, we know by Proposition~\ref{prop:ergodicity} that the
    only possible limit for~$(L_{t_n})_{n\geq 0}$ is~$\mu$, hence~$\mathscr{I}=\{\mu\}$
    and~$(L_{t_n})_{n\geq 0}$ almost surely converges to~$\mu$.
    
    We finally show for completeness that~$(L_t)_{t\geq 0}$ almost surely spends a finite Lebesgue
    time outside~$\mathscr{I}_\delta$.
    For this we introduce the random
        subset of~$\R_+$ of times $t\geq 0$ for which~$L_t$ does not belong
        to~$\mathscr{I}_\delta$,
        namely~$T =\{t\geq 0\,|\, L_t \notin \mathscr{I}_\delta\}$. Since
        \[
        \proba_x\big( L_t \notin \mathscr{I}_\delta \big) = \E_x[ \ind_{\{L_t\notin
          \mathscr{I}_\delta\} }],
        \]
        we have, by Fubini's theorem, for any~$t>0$,
        \[
        \int_0^t \proba_x\big( L_s \notin \mathscr{I}_\delta \big)ds =
        \E_x\left[ \int_0^t \ind_{\{L_s\notin
            \mathscr{I}_\delta\} }ds\right]
        = \E_x \big[ | T\cap[0,t] |\big].
        \]
        By using~\eqref{eq:expoineq} and the dominated convergence theorem, we obtain
        \[
        \E_x\big[|T|\big] = 
        \int_0^{+\infty} \proba_x\big( L_t \notin \mathscr{I}_\delta \big)dt
        <+\infty.
        \]
        As a result, $|T|<+\infty$ almost surely. 
        This means that,
        for any neighborhood~$\mathscr{I}_\delta$ of~$\mathscr{I}$ in the $\tau^\kap$-topology,
        the empirical measure~$(L_t)_{t\geq 0}$
        almost surely spends a finite Lebesgue measure time outside~$\mathscr{I}_\delta$,
        and this concludes the proof.
\end{proof}

\subsection{Proofs of Section~\ref{sec:Donsker}}
\label{sec:proofDonsker}

We start by providing a preliminary technical result in Section~\ref{sec:preliminary_tech_h_f}, which shows that
the eigenvectors~$h_f$ considered in Lemma~\ref{lem:specLf}
belong to the generalized domain~$\mathcal{D}^+(\Lc)$ defined in~\eqref{eq:Dplus}.
We then turn to the proofs of Proposition~\ref{prop:IDV} (see Section~\ref{sec:IDV}) and Corollary~\ref{cor:variational}
(see Section~\ref{sec:proofcor}).

\subsubsection{A preliminary technical result}
\label{sec:preliminary_tech_h_f}

\begin{lemma}
  \label{lem:pcpal_eig_L_f}
  Fix~$f \in B^\infty_\kap(\X)$. The function~$h_f \in B^\infty_W(\X)$ defined in Lemma~\ref{lem:specLf}
  belongs to~$\mathcal{D}^+(\Lc)$ and satisfies 
  \begin{equation}
     \label{eq:eigenproblem_ratio}
    -\frac{\Lc h_f}{h_f} = f - \lambda(f)\in\Binfty_\kap(\X).
  \end{equation}
\end{lemma}

\begin{proof}
  We already know by Lemma~\ref{lem:specLf} that $h_f \in C^0(\X)$ and $h_f > 0$. It suffices therefore to show that $h_f \in\mathcal{D}(\Lc)$ and to obtain the representation~\eqref{eq:eigenproblem_ratio} for~$\Lc h_f$. We combine to this end elements from~\cite[Theorem~4.2.25]{deuschel2001large} and~\cite[Proposition~B13]{Wu00}.

  We start by noting that, since~$h_f$ is an eigenvector
  of the operator~$P_t^f$ with eigenvalue~$\e^{\lambda(f) t}$, it holds
  \begin{equation}
    \label{eq:h_f_invariance_under_Pt}
    h_f(x) = \e^{-\lambda(f) t}\left(P_t^f h_f\right)(x) = \E_x\left( h_f(X_t) \, \e^{\int_0^t [f(X_s)-\lambda(f)]\,ds}\right).
  \end{equation}
  Therefore,
  \begin{equation}
    \label{eq:intsteplambda}
    \begin{aligned}
    \left(P_t h_f\right)(x)-h_f(x) & = \E_x\left[ \left(1-\e^{\int_0^t [f(X_s)-\lambda(f)]\,ds}\right)h_f(X_t)\right]
    \\ & = -\int_0^t \E_x\left[\big(f(X_s)-\lambda(f)\big)\e^{\int_s^t [f(X_\theta)-\lambda(f)]\,d\theta} \, h_f(X_t)\right],
    \end{aligned}
  \end{equation}
  where the last equality comes from Fubini's theorem and
  \[
    \Phi(t)-\Phi(0) = \int_0^t \Phi'(s)\,ds, \qquad \Phi(s) = \e^{\int_s^t [f(X_\theta)-\lambda(f)]\,d\theta}.
  \]
  Note that we can indeed apply Fubini's theorem since there exist $K,c>0$ such that
  \[
    \begin{aligned}
      & \left| \big(f(X_s)-\lambda(f)\big)\e^{\int_s^t [f(X_\theta)-\lambda(f)]\,d\theta} \, h_f(X_t)\right| \\
      & \qquad \qquad \leq K \left(\lambda(f)+\|f\|_{B^\infty_\kap}\right)\left\|h_f\right\|_{B^\infty_W} \kappa(X_s) W(X_t) \e^{c \int_s^t \kappa(X_\theta)\,d\theta},  
    \end{aligned}
  \]
  and (since we are integrating nonnegative functions)
  \[
  \begin{aligned}
    \int_0^t \E_x\left[ \kappa(X_s) W(X_t) \e^{c \int_s^t \kappa(X_\theta)\,d\theta} \right]ds & =
    \E_x\left[W(X_t) \int_0^t\kappa(X_s) \, \e^{c \int_s^t \kappa(X_\theta)\,d\theta} \, ds\right]
    \\ & \leq \frac1c \E_x\left[W(X_t) \, \e^{c \int_0^t \kappa(X_\theta)\,d\theta} \right],
    \end{aligned}
    \]
  where the last expression is finite by manipulations similar to the ones performed in the proof of Lemma~\ref{lem:martingales}. 
  
  We can next use~\eqref{eq:h_f_invariance_under_Pt} at initial time~$s \in [0,t]$ together with a conditioning argument to write
  \[
  \begin{aligned}
    \E_x\left[\big(f(X_s)-\lambda(f)\big)\e^{\int_s^t [f(X_\theta)-\lambda(f)]\,d\theta} \, h_f(X_t)\right]
    & = \E_x\left[\big(f(X_s)-\lambda(f)\big) \big(P_{t-s}^{f - \lambda(f)}h_f\big)(X_s)\right]
    \\ & = \E_x\left[\big(f(X_s)-\lambda(f)\big) h_f(X_s)\right].
\end{aligned}
    \]
  This finally shows that~\eqref{eq:intsteplambda} becomes
  \[
    P_t h_f-h_f = \int_0^t P_s\big[(\lambda(f) - f) h_f \big]ds.
  \]
  Since $(\lambda(f)-f)h_f$ is in $B^\infty_{\kap W}(\X)$ (as the product of functions in~$B^\infty_W(\X)$
  and~$B^\infty_\kap(\X)$) and $(P_t)_{t\geq 0}$ is a semigroup of bounded operators on~$\Binfty_{\kap W}(\X)$
  by~\eqref{eq:Lyapunov_kappaW}, it holds
  \[
  \int_0^t P_s |(\lambda(f)-f)h_f|\,ds <+\infty,
  \]
  so that~\eqref{eq:integrability} is satisfied.
  As a result, $h_f \in \mathcal{D}(\Lc)$ and $\Lc h_f = (\lambda(f) - f)h_f$ in
  the weak sense defined by~\eqref{eq:Dextended}.
\end{proof}

\begin{remark}
\label{rem:domain}
  It is actually possible to make more general statements about the domains of the generators
  of~$(P_t^f)_{t\geq 0}$ for $f\in\Binfty_{\kap}(\X)$, similarly to~\cite{Wu00,wu2001large}.
  For this, one considers the (closed) subset of functions~$\varphi\in\Binfty_W(\X)$
  for which $P_t^f\varphi \to \varphi$ in~$\Binfty_W(\X)$ when $t\to 0$,
  see~\cite[Exercice~1.16]{revuz2013continuous}. We can then define a generator~$\Lc_f$
  with domain~$D(\Lc_f)$
  for this semigroup. By manipulations similar to those of Lemma~\ref{lem:pcpal_eig_L_f},
  we can show that~$D(\Lc_f)\subset\mathcal{D}(\Lc)$ when
  we define~$\mathcal{D}(\Lc)$ as in~\eqref{eq:Dextended}. In this case we obtain the
  representation $\Lc_f = \Lc - f$ which could be expected. This procedure allows to define
  a common domain for the operators~$\Lc_f$ with $f\in\Binfty_{\kap}(\X)$.

  Here we bypass the approach sketched above because, for the proof of Proposition~\ref{prop:IDV} given
  below, we can restrict our attention to the eigenvectors~$h_f$ for~$f\in\Binfty_{\kap}(\X)$.
  In this case, it is clear that $P_t^f h_f \to h_f$ in $\Binfty_W(\X)$ when $t\to 0$,
  and we have the simple representation formula $\Lc h_f = (\lambda(f) - f) h_f$, which can be seen as
  a reformulation of the eigenvalue equation $(\Lc + f )h_f = \lambda(f)h_f$.
\end{remark}

\subsubsection{Proof of Proposition~\ref{prop:IDV}}
\label{sec:IDV}
For the proof, which is partly inspired by~\cite[Lemma~4.1.36]{deuschel2001large}, we denote
by~$\IF$ the rate function given by the Fenchel transform in~\eqref{eq:Ifenchel} and~$\IV$ for the
Varadhan functional on the right hand side of~\eqref{eq:IDV}. We repeatedly use the results
of Lemmas~\ref{lem:specLf} and~\ref{lem:pcpal_eig_L_f}.

We first show that~$\IV(\nu)=+\infty$ if $\nu$ is not absolutely continuous with respect
to~$\mu$ or does not belong to~$\mathcal{P}_\kap(\X)$. Assume first that~$\nu\ll\mu$ does not hold:
there exists a set $A\subset\X$ such that $\nu(A) >0$ and $\mu(A) = 0$.
For any $a\in\R$ we introduce~$f_a =a\ind_A$ and denote by~$h_a$ the eigenvector associated
with the principal eigenvalue~$\mathrm{e}^{t \lambda(f_a)}$ of~$P_t^{f_a}$ for some~$t>0$.
Recall that~$h_a \in \mathcal{D}^+(\Lc)$ by
Lemma~\ref{lem:pcpal_eig_L_f}. As shown in the proof of Theorem~\ref{theo:LDP}, it
holds~$\lambda(f_a)=0$, so that~\eqref{eq:eigenproblem_ratio} can be rewritten as 
\[
- \frac{\Lc h_a}{h_a} = a \ind_A.
\]
Therefore, 
\[
\IV(\nu)\geq \intX- \frac{\Lc h_a}{h_a}\,d\nu = a \nu(A) >0.
\]
By letting~$a\to +\infty$, we conclude that $\IV(\nu)=+\infty$ when~$\nu$ is not
absolutely continuous with respect to~$\mu$.
Next, if $\nu\notin\mathcal{P}_\kap(\X)$, since $\kap\geq 1$ it holds $\nu(\kap)=+\infty$.
We may then choose $f=\kap\in\Binfty_\kap(\X)$. By Lemma~\ref{lem:pcpal_eig_L_f}, the principal
eigenvector~$h_\kap$ belongs to~$\mathcal{D}^+(\Lc)$ with $\lambda(\kap)<+\infty$,
so we have
\[
\IV(\nu)\geq \intX- \frac{\Lc h_\kap}{h_\kap}\,d\nu = \intX \kap\,d\nu - \lambda(\kap) = +\infty,
\]
\textit{i.e.} $\IV(\nu)=+\infty$ if $\nu\notin\mathcal{P}_\kap(\X)$. This shows that
$\IF(\nu) = \IV(\nu)$ when $\nu$ is not absolutely continuous with respect to~$\mu$
or $\nu\notin\mathcal{P}_\kap(\X)$. We next show that $\IF = \IV$ when~$\nu\ll\mu$
and $\nu\in\mathcal{P}_\kap(\X)$, which we assume until the end of the proof.

Let us first show that $\IF\geq \IV$. For this, we consider $u\in\mathcal{D}^+(\Lc)$ and introduce
\[
f_u = - \frac{\Lc u}{u}.
\]
Because of the definition~\eqref{eq:Dplus} of~$\mathcal{D}^+(\Lc)$, we know that
$f_u\in\Binfty_\kap(\X)$. We can then write, since $\nu\in\mathcal{P}_\kap(\X)$,
\begin{equation}
  \label{eq:IFpart}
  \IF(\nu)\geq \nu(f_u) - \lambda(f_u).
\end{equation}
We now show that $\lambda(f_u)\leq 0$. By computations similar to the ones in the proof
of Lemma~\ref{lem:martingales}, and using the continuity of~$u\in\mathcal{D}^+(\Lc)$
(see also~\cite[Corollary~2.2]{wu2001large}), we obtain by the local martingale property that
\begin{equation}
  \label{eq:Ptfu_bound}
0 \leq P_t^{f_u} u \leq u.
\end{equation}
Therefore, recalling the definition~\eqref{eq:Q_h} of the $h$-transformed evolution operator
with a time~$t>0$ fixed (with $r(f_u)=\lambda(f_u)$ in view of Lemma~\ref{lem:specLf}),
and denoting by~$h_u>0$ the eigenvector associated with~$f_u$ in
Lemma~\ref{lem:specLf},~\eqref{eq:Ptfu_bound} becomes
\[
\e^{-nt \lambda(f_u)} \frac{u}{h_u} \geq Q_{h_u}^n \left(\frac{u}{h_u}\right) \xrightarrow[n\to+\infty]{} \int_\X \frac{u}{h_u} \, d\mu_{h_u},
\]
where the limit $n \to +\infty$ follows from~\eqref{eq:hergodic} (noting that $u/h_u \in B^\infty_{W h_u^{-1}}(\X)$). The latter limit is positive since $u/h_u$ is continuous and positive, which implies that $\lambda(f_u) \leq 0$. Therefore,~\eqref{eq:IFpart} leads to
\[
\IF(\nu)\geq \nu(f_u) - \lambda(f_u) \geq \nu(f_u) = \int_X - \frac{\Lc u}{u}\, d\nu.
\]
Since $u\in\mathcal{D}^+(\Lc)$ is arbitrary, taking the supremum shows
that $\IF(\nu)\geq \IV(\nu)$ for any
$\nu\in \mathcal{P}_\kap(\X)$ with $\nu\ll\mu$.

We finally turn to the inequality $\IF\leq \IV$. Consider for any arbitrary $f\in\Binfty_\kap(\X)$
the eigenvector $h_f\in\Binfty_W(\X)$ defined in Lemma~\ref{lem:specLf}. By
Lemma~\ref{lem:pcpal_eig_L_f}, this eigenvector belongs
to~$\mathcal{D}^+(\Lc)$ and satisfies $\Lc h_f = (\lambda(f) - f)h_f$. Thus, since $\nu\in\mathcal{P}_\kap(\X)$,
we have
\[
\IV(\nu) \geq \intX - \frac{\Lc h_f}{h_f}\,d\nu= \nu(f) - \lambda(f).
\]
Given that, in the above equation,~$f$ is an arbitrary function belonging
to $\Binfty_\kap(\X)$, taking the supremum leads to
\[
\IV(\nu) \geq \underset{f\in\Binfty_\kap}{\sup}\ \big\{\nu(f) - \lambda(f)\big\}.
\]
This finally shows that $\IF(\nu) = \IV(\nu)$ for all~$\nu\in\mathcal{P}_\kap(\X)$ with $\nu\ll\mu$
and concludes the proof.


\subsubsection{Proof of Corollary~\ref{cor:variational}}
\label{sec:proofcor}
Since~$I$ is the Fenchel transform of~$\lambda$, the result follows if we can show
that the application~$\lambda$ defined on~$\Binfty_\kap(\X)$ is stable by bi-Fenchel conjugacy.
The convexity and finiteness of~$\lambda$ show that a (necessary and)
sufficient condition for~$\lambda$ to be bi-Fenchel stable is for the functional
$f\mapsto \lambda(f)$ to be
lower-semicontinuous (see~\cite[Theorem~2.22]{barbu2012convexity}). We show below
that it is actually continuous: for any sequence~$(f_n)_{n\geq 0}$ in~$\Binfty_\kap(\X)$ such
that $\|f_n - f\|_{\Binfty_\kap}\to 0$ for some $f\in\Binfty_\kap(\X)$, it holds
$\lambda(f_n)\to\lambda(f)$ as $n\to+\infty$. We shall use for this a stability result
from~\cite{chatelin1981spectral}.

Consider a sequence~$(f_n)_{n\geq 0}$ converging to~$f$ in~$\Binfty_\kap(\X)$. Using
Lemma~\ref{lem:diffbound},  for any $\varphi\in\Binfty_W(\X)$, $t>0$, $x\in\X$ and $n\in\N$,
it holds (using again the inequality $a\leq \e^a$ for $a\geq 0$)
\[
  \begin{aligned}
    &   \left| \big(P_t^f \varphi\big)(x) -
      \big(P_t^{f_n} \varphi\big)(x)\right| \\
    & \qquad \qquad \leq \|\varphi\|_{\Binfty_W}
  \E_x\left[  W(X_t)\left(\int_0^t |f(X_s) - f_n(X_s)|\,ds\right)
    \e^{(\| f\|_{\Binfty_\kap}+\| f_n\|_{\Binfty_\kap})\int_0^t\kap(X_s)\,ds}\right]
  \\ &  \qquad \qquad \leq \|\varphi\|_{\Binfty_W} \|f - f_n\|_{\Binfty_\kap}
  \E_x\left[  W(X_t)
    \e^{2(\| f\|_{\Binfty_\kap}+\| f_n\|_{\Binfty_\kap})\int_0^t\kap(X_s)\,ds}\right]
  \\ &  \qquad \qquad \leq C \|\varphi\|_{\Binfty_W} \|f - f_n\|_{\Binfty_\kap}
  \E_x\left[  M_t\right]
  \\ &  \qquad \qquad \leq C \|\varphi\|_{\Binfty_W} \|f - f_n\|_{\Binfty_\kap}W(x),
\end{aligned}
\]
for some constant $C>0$ depending on~$t>0$, $\|f\|_{\Binfty_\kap}$ and $\sup_{n \geq 0} \|f_n\|_{\Binfty_\kap}$. We used
Lemma~\ref{lem:martingales}
and the supermartingale property of~$M_t$ to obtain the last line. This leads to
\begin{equation}
  \label{eq:normconvop}
\big\| P_t^f - P_t^{f_n}\big\|_{\mathcal{B}(\Binfty_W)}\leq C \|f - f_n\|_{\Binfty_\kap}
\xrightarrow[n\to +\infty]{}0.
\end{equation}
We know by Lemma~\ref{lem:specLf} that~$\lambda(f)$ and~$\lambda(f_n)$ are associated with
the isolated largest eigenvalue of the operators~$P_t^f$ and~$P_t^{f_n}$ respectively.
Therefore,~\eqref{eq:normconvop} shows that the
approximation is strongly stable (we refer to~\cite{chatelin1981spectral}, in particular
the definitions in Section~2.2 and Proposition~2.11),
so~\cite[Proposition~2.2]{chatelin1981spectral}
ensures that $\lambda(f_n)\to \lambda(f)$ as $n\to +\infty$. This shows that the function
$\lambda:\Binfty_\kap(\X)\to\R$ is continuous and concludes the proof.


\subsubsection{Proof of Theorem~\ref{theo:InormH}}
\label{sec:Idecomp}
The proof, inspired by~\cite{bodineau2008large}, relies on two ideas:
performing a Witten transform inside the variational representation~\eqref{eq:IDV}
and separating the symmetric and antisymmetric parts of the generator~$\Lc$.
We write~$d\nu=\rho\,d\mu=\e^v\,d\mu$ and assume first that $v\in\Cinfc(\X)$
instead of~$\H^1(\nu)$.
Starting from~\eqref{eq:IDV}, we consider a function~$u$ of the form
\begin{equation}
  \label{eq:uwitten}
u = \e^{\frac{\psi}{2}}\sqrt{\rho},\quad \psi \in\Cinfc(\X).
\end{equation}
We call this choice ``variational Witten transform'' for its similarity with the standard
Witten transform~\cite{witten1982supersymmetry,helffer2006semi,lelievre2016partial} and its
use in the variational formula~\eqref{eq:IDV} satisfied by~$I$.
Since $u= \e^{\frac{\psi}{2} + \frac{v}{2}}$ with $v,\psi\in\Cinfc(\X)$ it is clear
that~$u\in\mathcal{D}^+(\Lc)$. This follows by noting that,
using the shorthand notation~$w=\psi/2 +v/2\in\Cinfc(\X)$, we have
\[
-\frac{\Lc u}{u} =- \e^{-w}\Lc \e^w = -\Lc w -\frac{1}{2} |\sigma^T\nabla w|^2\in\Cinfc(\X)
\subset \Binfty_\kap(\X).
\]
Moreover, it holds $u=\e^w>0$ and~$u$ is constant outside a compact set, so
$u\in\Binfty_W(\X)$ and it holds $u\in \mathcal{D}^+(\Lc)$.

We now rewrite the expression in~\eqref{eq:IDV} for~$u$ given by~\eqref{eq:uwitten},
using again the notation~$w=\psi/2 +v/2$:
\[
-\intX\frac{\Lc u}{u}\,d\nu = -\intX \Lc w \,d\nu
-\frac{1}{2}\intX |\sigma^T\nabla w|^2\,d\nu.
\]
Recalling that $S=\sigma\sigma^T/2$ and expanding $w=\psi/2 +v/2$, we obtain
\begin{equation}
  \label{eq:Luratio}
  \begin{aligned}
    -\intX\frac{\Lc u}{u}\,d\nu & = -\frac{1}{2}\intX \Lc \psi \,d\nu
    -\frac{1}{2}\intX \Lc v \,d\nu \\
    & \ \ \ - \frac{1}{4}\intX \nabla \psi \cdot S \nabla\psi \,d\nu
    - \frac{1}{2}\intX \nabla v \cdot S \nabla\psi \,d\nu
    - \frac{1}{4}\intX \nabla v \cdot S \nabla v \,d\nu.
  \end{aligned}
\end{equation}
We now decompose~$\Lc$ into symmetric and antisymmetric parts. First, it holds
\begin{equation}
  \label{eq:intsymm}
  \begin{aligned}
    -\frac{1}{2}\intX \Lc \psi \,d\nu & = -\frac{1}{2}\intX (\LS \psi)\,\e^v \,d\mu
    -\frac{1}{2}\intX (\LA \psi) \,d\nu \\
    & = \frac{1}{2}\intX \nabla \psi\cdot S\nabla v \,d\nu
    -\frac{1}{2}\intX (\LA \psi) \,d\nu.
  \end{aligned}
\end{equation}
On the other hand, using that~$\LA$ is a first order differential operator
satisfying $\LA^*\ind = -\LA \ind = 0$, we obtain
\[
\intX (\LA v)\,\e^v \,d\mu = \intX (\LA \e^v) \,d\mu = \intX (\LA^*\ind)\, \e^v \,d\mu= 0.
\]
As a result
\begin{equation}
  \label{eq:intantisym}
-\frac{1}{2}\intX \Lc v \,d\nu  = -\frac{1}{2}\intX (\LS v)\,\e^v \,d\mu 
-\frac{1}{2}\intX (\LA v)\,\e^v \,d\mu
= \frac{1}{2}\intX \nabla v\cdot S \nabla v \,d\nu.
\end{equation}
By plugging~\eqref{eq:intsymm}-\eqref{eq:intantisym} into~\eqref{eq:Luratio},
we obtain
\begin{equation}
  \label{eq:Ludecomp}
-\intX\frac{\Lc u}{u}\,d\nu = \frac{1}{4}\intX \nabla v \cdot S \nabla v\,d\nu
-\frac{1}{2}\intX (\LA \psi) \,d\nu - \frac{1}{4}\intX \nabla \psi \cdot S \nabla \psi\,d\nu.
\end{equation}
The first term in the above equation reads (recalling that $\rho=\e^v$)
\[
\frac{1}{4}\intX \nabla v \cdot S \nabla v\,d\nu = \intX \nabla(\sqrt{\rho})\cdot
S \nabla(\sqrt{\rho})\,d\mu.
\]
By density of~$\Cinfc(\X)$ in~$\H^1(\mu)$, the above expression is valid for any~$\rho$
such that~$\sqrt{\rho}\in\H^1(\mu)$. The above computation shows that this condition is equivalent
to assuming that $v\in\H^1(\nu)$, and
\[
\frac{1}{4}\intX \nabla v \cdot S \nabla v\,d\nu = \frac{1}{4}|v|_{\H^1(\nu)}^2,
\]
which does not involve the function~$\psi\in\Cinfc(\X)$.
Moreover, since~$\LA$ is a first order differential operator, antisymmetric on~$L^2(\mu)$, it holds
\[
\intX (\LA \psi) \,d\nu = -\intX (\LA \e^v) \psi\,d\mu = -\intX (\LA v) \psi\,d\nu.
\]
As a result,~\eqref{eq:Ludecomp} rewrites
\begin{equation}
  \label{eq:Ludecomp2}
-\intX\frac{\Lc u}{u}\,d\nu = \frac{1}{4}|v|_{\H^1(\nu)}^2
+\frac{1}{2}\intX (\LA v) \psi \,d\nu - \frac{1}{4}|\psi|_{\H^1(\nu)}^2,
\end{equation}
and this expression is finite for any~$\psi\in\Cinfc(\X)$.

Our goal is now to take the supremum
over functions~$\psi\in\Cinfc(\X)$ in~\eqref{eq:Ludecomp2}, and prove that this is enough
to obtain the supremum over~$\mathcal{D}^+(\Lc)$. We consider for this
the terms depending on~$\psi$
in~\eqref{eq:Ludecomp2} and, using the duality between~$\H^1(\nu)$
and~$\H^{-1}(\nu)$ (see~\cite[Section~2, Claim~F]{komorowski2012fluctuations}) we obtain
\begin{equation}
  \label{eq:ineqH1}
\begin{aligned}
\frac{1}{2}\intX (\LA v)\psi \,d\nu - \frac{1}{4}\intX \nabla \psi \cdot S \nabla \psi\,d\nu
& \leq \frac{1}{2}|\LA v|_{\H^{-1}(\nu)}|\psi|_{\H^1(\nu)} - \frac{1}{4}|\psi|_{\H^1(\nu)}^2
\\ &\leq \frac{1}{4\varepsilon}|\LA v|_{\H^{-1}(\nu)}^2
- \frac{1}{4}(1-\varepsilon)|\psi|_{\H^1(\nu)}^2,
\end{aligned}
\end{equation}
where we used Young's inequality with $\varepsilon<1$ to obtain the second line.
Since~$\LA v\in \H^{-1}(\nu)$, the supremum over the
functions~$\psi\in\Cinfc(\X)$ takes the value~$-\infty$
when~$\psi\notin\H^1(\nu)$. Therefore, by density of $\Cinfc(\X)$ in~$\H^1(\nu)$, the supremum
over the functions of the form~\eqref{eq:uwitten} for~$\psi\in\Cinfc(\X)$ recovers the supremum
over~$\mathcal{D}^+(\Lc)$ and it holds
\begin{equation}
  \label{eq:Ifinaldense}
I(\nu) = \frac{1}{4}|v|_{\H^1(\nu)}^2 +\frac{1}{4}|\LA v|_{\H^{-1}(\nu)}^2,
\end{equation}
by definition of the $\H^{-1}(\nu)$-norm in Section~\ref{sec:setting}, which concludes the
proof.

\begin{remark}
  \label{rem:extendrho}
  We have proved our result for measures of the form $d\nu=\e^v \,d\mu$. Considering
  more general measures~$\nu\ll\mu$ is made difficult because the Radon--Nikodym
  derivative $\rho = d\nu/d\mu$ may vanish on some region of~$\X$, hence the definition
  of~$\LA(\log \rho)$ is not clear. Given~\eqref{eq:ineqH1}, we see
  that we can give a sense to our computations provided~$\LA(\log \rho)$ defines a linear
  form on~$\H^1(\nu)$, namely: there exists~$C>0$ such that
  \[
  \forall\,\psi \in\H^1(\nu),\quad
  \left|\intX  \psi \LA(\log\rho) \,d\nu\right| \leq C \| \psi \|_{\H_1(\nu)}.
  \]
  We find it however clearer to work directly with exponential perturbations of
  the invariant measure~$\mu$.
\end{remark}


\subsubsection{Proof of Corollary~\ref{cor:Iantisym}}
\label{sec:Iantisym}
The proof follows from the variational formulation of Theorem~\ref{theo:InormH}. Indeed, let us
rewrite~\eqref{eq:Iantisym} as
\begin{equation}
  \label{eq:intpsi}
\IA(\nu) = - \frac{1}{2} \inf_{\psi\in \H^1(\nu)}\ \I_\nu(\psi),
\end{equation}
where~$\nu$ is fixed and satisfies the assumptions of the theorem, and
\[
\I_\nu(\psi) =  \frac{1}{2}\intX \mathscr{C}( \psi,\psi)\, d\nu
 -\intX\psi (\LA v) \, d\nu. 
 \]
By~\cite[Section~2, Claim~F]{komorowski2012fluctuations}, we can identify
$\H^{-1}(\nu)$ with the dual of~$\H^1(\nu)$, so that~$\I_\nu$ reads
\[
\forall\,\psi\in \H^1(\nu),\quad \I_\nu(\psi) =\frac{1}{2} |\psi|_{\H^1(\nu)}^2 -
\langle \LA v,\psi\rangle_{\H^{-1}(\nu),\H^1(\nu)}.
\]
Denoting by $\widetilde{\nabla}$ the adjoint of the gradient
operator in~$L^2(\nu)$, standard results of calculus of variations
show that the minimum in~\eqref{eq:intpsi} is attained at a unique $\psi_v\in\H^1(\nu)$ solution to
\begin{equation}
  \label{eq:intPDE}
\widetilde{\nabla}( S \nabla \psi_v) = \LA v.
\end{equation}
Inserting~$\psi_v$ solution to~\eqref{eq:intPDE} in~\eqref{eq:intpsi} leads to
\begin{equation}
  \label{eq:intIA}
\IA(\nu) =\frac{1}{4} \intX \mathscr{C}(\psi_v,\psi_v)\,d\nu,
\end{equation}
which concludes the proof.

\appendix

\section{Tools for large deviations principles}
\label{sec:tools}
In this section, we remind some large deviations concepts (using the abuse of
notation discussed at the beginning of Section~\ref{sec:proofs} for denoting expectations and probabilities).
For a Polish space~$\Y$, we denote by~$\Y'$ its topological dual (the set of continuous
linear functionals over~$\Y$). 
We first recall the definition of an exponentially tight family of measures.
A family of measures~$(\pi_t)_{t\geq 0}$ over a Polish space~$\Y$ is called exponentially tight
if for any $N<+\infty$, there exists a (pre)compact set~$\Gamma_N\subset\Y$ such that
\[
\underset{t \to +\infty}{\limsupb}\, \frac{1}{t}\, \log \pi_t\big( \Gamma_N^c\big) < - N.
\]
In words, exponential tightness means that the measures~$(\pi_t)_{t\geq 0}$ concentrate
exponentially fast over compact sets. This property is used in large deviations to turn an upper bound
over compact sets into an upper bound over all closed sets.

We now define the cumulant function.
Consider a family of measures~$(\pi_t)_{t\geq 0}$ over a Polish space~$\Y$.
The logarithmic moment generating function is defined as in~\cite[Section~4.5]{dembo2010large}:
for any $t\geq 0$, $f\in\Y'$ and a random variable~$Z_t$ distributed according to~$\pi_t$,
\begin{equation}
  \label{eq:apCGF}
  \Lambda_t(f) = \log \E\left[ \e^{ \langle f , Z_t\rangle_{\Y',\Y}}\right]
  = \log \int_{\Y} \e^{\langle f , y\rangle_{\Y',\Y} } \pi_t(dy).
\end{equation}
The scaled cumulant generating function is defined by
\begin{equation}
  \label{eq:apSCGF}
  \bar{\Lambda}_t(f) = \frac{1}{t}\Lambda_t(tf).
\end{equation}

Let us relate this quantity with the objects introduced in Section~\ref{sec:LDP}. In our
situation, we consider fluctuations of the empirical measure~$L_t\in\mathcal{M}(\X)$
(where~$\mathcal{M}(\X)$ is the space of measures with finite mass),
so~$\Y=\mathcal{M}(\X)$ and for $\Gamma\in \mathcal{M}(\X)$,
\[
\pi_t (\Gamma) = \proba_x\left( L_t \in \Gamma\right).
\]
On the other hand,~$f$ belongs to a space of functions,
typically~$\Y'=\mathcal{M}(\X)' =\Binfty(\X)$ when the $\tau$-topology is considered.
In practice we may restrict ourselves to probability
measures because the rate function is infinite otherwise. We see that considering
$L_t\in \mathcal{P}_\kap(\X)$ leads to choosing $f\in\Binfty_\kap(\X)$.
In any case the duality relation~\eqref{eq:apCGF} reads in this case
\[
\Lambda_t(f) = \log \int_{\PX} \e^{\langle f , L_t\rangle_{\Y',\Y}}\proba_x\left( L_t \in dy\right)
=\log \E_x\left[ \e^{L_t(f)}\right]=\log \E_x\left[ \e^{\frac{1}{t}\int_0^t f(X_s)\,ds}\right],
\]
so that $\bar{\Lambda}_t(f)$ coincides with the argument of the limit in~\eqref{eq:SCGF}. With
these preliminaries, we are
in position to state the key theorem for the results in this work, which goes back
to~\cite{gartner1977large,ellis2007entropy}
and is presented for instance in~\cite[Corollary~4.6.14]{dembo2010large}.
We recall that a rate function is said to be \emph{good} if its level sets are compact
for the considered topology.

\begin{theorem}[Projective limit - Gärtner--Ellis]
  \label{theo:GE}
  Let~$(\pi_t)_{t\geq 0}$ be an exponentially tight family of probability measures on
  a Polish space~$\Y$. Assume that
  \[
  \Lambda(\cdot) = \lim_{t\to+\infty} \bar{\Lambda}_t(\cdot) 
  \]
  is finite valued over~$\Y'$ and Gateau-differentiable. Then~$(\pi_t)_{t\geq 0}$ satisfies a
  large deviations principle
  over~$\Y$ with good rate function~$\Lambda^*$, the Legendre--Fenchel transform of~$\Lambda$.
\end{theorem}

\section{Proof of Proposition~\ref{prop:nonlin}}
\label{sec:nonlin}
The proposition is a consequence of the equality
\[
\Ps=-\frac{\Lc W}{W} =  \theta \left( - \Lc V - \frac{\theta}{2} |\sigma^T\nabla V|^2\right).
\]
Since $|\sigma^T\nabla V|$ has compact level sets and $\Ps \sim |\sigma^T \nabla V|^2$
by~\eqref{eq:nonlin},~$\Ps$ has compact level sets. Since~$V$ has compact level sets,
for $\varepsilon < \theta/2$ it holds $\mathscr{W}\ll W$ and $\mathscr{W}^2\leq C_1 W$
for some constant $C_1 >0$. Moreover, outside a compact set, the function
\[
\frac{\Ps}{-\frac{\Lc \mathscr{W}}{\mathscr{W}}} =\frac{\theta}{\varepsilon}
\frac{  (-\Lc V - \frac{\theta}{2}|\sigma\nabla V|^2)  }{ ( -\Lc V
  - \frac{\varepsilon}{2}|\sigma\nabla V|^2) }
\]
is bounded above and below since the numerator and denominator are both equivalent
to~$|\sigma^T \nabla V|^2$, so the second condition in~\eqref{eq:restriction} holds. Finally,
\[
\begin{aligned}
  - 2 \frac{\Lc \mathscr{W}}{\mathscr{W}} = 2 \varepsilon \left( -\Lc V
  - \frac{\varepsilon}{2}|\sigma\nabla V|^2\right)
  & = 2 \frac{\varepsilon}{\theta}\theta \left( -\Lc V  - \frac{\theta}{2}|\sigma\nabla V|^2\right)
  + \varepsilon(\theta - \varepsilon)|\sigma\nabla V|^2
  \\
  & = 2 \frac{\varepsilon}{\theta} \Ps + \varepsilon(\theta - \varepsilon)|\sigma\nabla V|^2.
\end{aligned}
\]
Since $\Psi\sim |\sigma\nabla V|^2$, we may choose $\varepsilon$ small enough so as
to obtain
\[
- 2 \frac{\Lc \mathscr{W}}{\mathscr{W}} \leq \Psi + C_2,
\]
for some constant $C_2 \in \R$. This proves the third item of~\eqref{eq:restriction}.

We finally turn to the proof of~\eqref{eq:Lyapunov_kappaW}. For this we compute
\[
\Lc (\kap W) = \kap \Lc W + W \Lc \kap + (\sigma^T \nabla \kap )\cdot (\sigma^T \nabla W).
\]
Hence, using that $W(x) = \e^{\theta V(x)}$, for any~$\eta>0$ it holds
\[
  \begin{aligned}
    \frac{\Lc (\kap W)}{\kap W}
    & = \frac{\Lc W}{W} + \frac{\Lc \kap}{\kap} +  (\sigma^T \nabla \log W)\cdot(\sigma^T \nabla \log \kap ) \\
    & \leq - \Psi + \frac{\Lc \kap}{\kap} + \frac{\eta}{2} |\sigma^T \nabla \log W|^2+ \frac{1}{2\eta} |\sigma^T \nabla \log \kap|^2 \\
    & = - \Psi + \frac{\eta}{2}\theta^2 |\sigma^T \nabla V|^2 + \frac{\Lc \kap}{\kap}
    + \frac{1}{2\eta} |\sigma^T \nabla \log \kap|^2.
  \end{aligned}
\]
Since $\Psi \sim |\sigma^T \nabla V|^2$ at infinity and~\eqref{eq:condkappa} holds,
this shows that~\eqref{eq:Lyapunov_kappaW} is satisfied when choosing~$\eta>0$ sufficiently small.

\section{Proof of Lemma~\ref{lem:lyapunovlangevin}}
\label{sec:lyapunovlangevin}
The proof relies on manipulations similar to those of~\cite{mattingly2002ergodicity}.
A simple computation shows that
\begin{equation}
-\frac{\Lc_\gamma W}{W}(q,p) = \varepsilon q \cdot \nabla V - \gamma \varepsilon^2|q|^2
+ \gamma\varepsilon (1 - 2\theta) p\cdot q
+ \theta \gamma( 1 - \theta) |p|^2  - \varepsilon |p|^2 -\theta \gamma d.
\end{equation}
For any~$\eta >0$ it holds
\[
p\cdot q \geq -\eta\frac{|q|^2}{2} - \frac{|p|^2}{2\eta}.
\]
As a result, Assumption~\ref{as:langevin} leads to
\[
  \begin{aligned}
  -\frac{\Lc_\gamma W}{W}(q,p) & \geq  |q|^2 \left( c_V\varepsilon - \gamma \varepsilon^2 - \frac{\eta\gamma\varepsilon}{2}
  (1 - 2\theta)\right)
  + |p|^2 \left( \theta\gamma - \theta^2 \gamma - \varepsilon - \frac{\gamma \varepsilon}{2\eta}
    (1 - 2\theta)\right) \\
  & \qquad -\theta \gamma d - \varepsilon C_V.
  \end{aligned}
\]
Since~$\theta > 0$, it holds
\[
-\frac{\Lc_\gamma W}{W}(q,p) \geq a |q|^2 + b |p|^2 - C,
\]
with
\[
  a  = \varepsilon\left(c_V - \frac{\eta\gamma}{2}\right)  - \gamma\varepsilon^2, 
\quad
b  = \theta (1 - \theta) \gamma - \varepsilon - \frac{\gamma \varepsilon}{2\eta},
\quad C = \theta \gamma d + \varepsilon C_V.
  \]
  The claim follows for~$\theta\in(0,1)$ by choosing $\eta<2c_V/\gamma$ and~$\varepsilon>0$
  sufficiently small.





\end{document}